\documentclass[a4paper,10pt,oneside,reqno]{amsart}

\usepackage[foot]{amsaddr}
\usepackage{amssymb}
\usepackage{latexsym}
\usepackage{amsmath}
\usepackage{amsthm}
\usepackage{amsfonts}
\usepackage{amssymb}
\usepackage[utf8]{inputenc}
\usepackage{biblatex}
\usepackage{amsmath,amscd}
\usepackage{graphicx}
\usepackage[total={6.3in, 9in}]{geometry}
\usepackage{scalerel,stackengine,relsize}
\stackMath
\usepackage{amssymb}
\usepackage{bbm}
\usepackage{ifthen}
\usepackage{tikz}
\usepackage{tikz-3dplot}
\usepackage{bm}
\usepackage{amsmath,amssymb,amsfonts,amsthm,amsopn,dsfont,mathrsfs}
\usepackage{esint}
\usepackage{graphicx}
\usepackage{enumerate}
\usepackage{bigints}
\usepackage{fancyhdr}
\usepackage[T1]{fontenc}
\usepackage{mathtools}
\mathtoolsset{showonlyrefs=true}
\usepackage{verbatim}
\usepackage{hyperref}
\hypersetup{
    colorlinks=true,
    linkcolor=black,
    filecolor=magenta,      
    urlcolor=cyan,
    citecolor=magenta,
    pdfpagemode=FullScreen,
    }
\newtheorem{Th}{Theorem}[section]
\newtheorem{Prop}{Proposition}[section]
\newtheorem{Co}{Corollary}[section]
\newtheorem{Lm}{Lemma}[section]
\newtheorem{Tha}{Theorem}[section]
\newtheorem{Lma}{Lemma}[section]
\newtheorem{Propa}{Proposition}[section]
\theoremstyle{definition}
\newtheorem{Dfi}{Definition}[section]
\newtheorem{Rm}{Remark}[section]
\newtheorem{Rma}{Remark}[section]

\numberwithin{equation}{section}

\def\La{\Lambda}
\def\La{\Lambda}
\def\ti{\tilde}
\def\lf{\left}
\def\rg{\right}
\def\al{\alpha}
\def\la{\lambda}
\def\ep{\varepsilon}
\def\eps{\varepsilon}
\def\ov{\overline}
\def\Om{\Omega}
\def\om{\omega}
\def\p{\partial}
\def\ds{\displaystyle}

\DeclareMathOperator{\dist}{dist}
\newcommand{\n}{\mathbb{N}}
\newcommand{\g}{\mathfrak{g}}
\newcommand{\z}{\mathbb{Z}}
\renewcommand{\r}{\mathbb{R}}
\renewcommand{\c}{\mathbb{C}}
\renewcommand{\b}{\mathbb{B}}
\newcommand{\s}{\mathbb{S}}
\newcommand{\A}{\mathfrak{a}}
\newcommand{\opwedge}{\mathlarger{\mathlarger{\wedge}}}
\renewcommand{\le}{\leqslant}
\renewcommand{\ge}{\geqslant}
\renewcommand{\setminus}{\smallsetminus}
\renewcommand{\H}{\mathscr{H}}
\renewcommand{\L}{\mathcal{L}}
\newcommand{\be}{\begin{equation}}
\newcommand{\ee}{\end{equation}}
\newcommand{\bes}{\begin{equation*}}
\newcommand{\ees}{\end{equation*}}
\newcommand{\sep}{0mm}
\renewcommand{\H}{\mathscr{H}}
\renewcommand{\L}{\mathcal{L}}
\newcommand{\R}{\mathbb{R}}
\newcommand{\N}{\mathbb{N}}

\newcommand{\Z}{\mathbb{Z}}

\DeclareMathOperator{\vol}{vol}

\DeclareMathOperator{\ad}{Ad}

\newcommand\reallywidehat[1]{%
	\savestack{\tmpbox}{\stretchto{%
			\scaleto{%
				\scalerel*[\widthof{\ensuremath{#1}}]{\kern-.6pt\bigwedge\kern-.6pt}%
				{\rule[-\textheight/2]{1ex}{\textheight}}%WIDTH-LIMITED BIG WEDGE
			}{\textheight}% 
		}{0.5ex}}%
	\stackon[1pt]{#1}{\tmpbox}%
}
\usepackage{accents}

\title[]{Coulomb Gauges and Regularity for Stationary Weak Yang--Mills Connections in Supercritical Dimension} 

\author[R. Caniato]{Riccardo Caniato}
\address{ \mbox{\it{(R. Caniato)} } California Institute of Technology, Department of Mathematics, 1200 E California Blvd, MC 253-37, CA 91125, United States of America}
\email{rcaniato@caltech.edu}

\author[T. Rivi\`ere]{Tristan Rivi\`ere}
\address{\mbox{\it{(T. Rivi\`ere)} } ETh Z\"urich, Department of Mathematics, R\"amistrasse 101, Z\"urich 8092, Switzerland}
\email{tristan.riviere@math.ethz.ch}

\date{\today}

\begin{document}
\begin{abstract}
    \noindent
    We prove that stationary Yang--Mills fields in dimensions 5 belonging to the variational class of \textit{weak connections} are smooth away from a closed singular set $S$ of vanishing 1-dimensional Hausdorff measure. Our proof is based on an $\eps$-regularity theorem, which generalizes to this class of weak connections the existing previous $\eps$-regularity results by G. Tian for smooth connections, by Y. Meyer and the second author for Sobolev and approximable connections, and by T. Tao and G. Tian for admissible connections (which are weak limits of smooth Yang--Mills fields). On the path towards establishing $\eps$-regularity, a pivotal step is the construction of controlled Coulomb gauges for general weak connections under small Morrey norm assumptions.
\end{abstract}
\allowdisplaybreaks
\maketitle
\begin{center}
    \textbf{ MSC (2020): 58E15 (primary)}; 49Q15, 49Q20, 53C65, 81T13 (secondary).
\end{center}
\tableofcontents
\section{Introduction}
The study of the variations of the Yang--Mills Lagrangian has known a spectacular development since the very first analytical works by K. Uhlenbeck, which have been central in producing new invariants of differential structures on topological $4$-manifolds. In particular, S. Donaldson proved his celebrated result on the existence of non-smoothable topological $4$-manifolds by studying properties of the the moduli space of the anti-self-dual instantons\footnotemark{} over such manifolds (see \cite{donaldson2}, \cite{donaldson}). Examples of these manifolds had already been constructed by Freedman in \cite{Freedman}. Moreover, Taubes proved the existence of uncountably many fake $\R^4$'s by means of gauge theoretic methods in \cite{Taubes}. For a complete discussion of these topics, we refer the reader to \cite{DK} or \cite{FU}. 

Due to the successful use of the Yang--Mills energy in dimension 4, it is natural to explore its behavior in higher dimensions, specifically in the \textit{supercritical} regime. In fact, S. Donaldson and R. Thomas outlined a research program in this direction in \cite{donaldson-thomas}, \cite{thomas}. However, analyzing the Yang--Mills Lagrangian becomes increasingly challenging in dimensions higher than 4, where we are led to consider\textit{singular} solutions, which naturally emerge in this more complicated context.

Locally, the Yang--Mills Lagrangian is defined as follows. Let $A$ be a 1-form on the flat $n$-dimensional unit ball ${\b}^n$ taking values into the Lie Algebra $\g$ of a compact matrix Lie group $G$. The \textit{Yang--Mills energy} of this ``connection $1$-form'' is given by
\begin{align}\label{r-I.1}
    \mbox{YM}(A):=\int_{\b^n}|dA+A\wedge A|^2\,d\L^n,
\end{align}
\footnotetext{\,Instantons in dimension 4 are very special critical points of the Yang--Mills functional solving a first order PDE, in a similar fashion as holomorphic functions  are very special critical points of the Dirichlet energy in dimension 2.}
where $\L^n$ denotes the \textit{Lebesgue measure} on $\r^n$, $A\wedge A$ is the $\g$-valued 2-form given by
\begin{align}
    A\wedge A(X,Y):=[A(X),A(Y)] \qquad\forall\, X,Y\in\r^n
\end{align}
is the Lie Algebra bracket on $\g$. The curvature $F_A:=dA+A\wedge A$ is the sum of a \textit{linear operation} on $A$ (i.e. $\,dA$) and a \textit{bilinear} one (i.e. $A\wedge A$). One of the main difficulties in the analysis of the Yang--Mills functional is to understand which one of the two is ``taking over'' the other along sequences with uniformly bounded Yang--Mills energy. As we will see shortly, the general answer to this question strongly depends on the dimension $n$ of the base manifold.

The Yang--Mills Lagrangian is conformally invariant in dimension 4. This makes the 4-dimensional case critical from a purely analytic perspective, as we are explaining in the forthcoming paragraphs. The Yang--Mills Lagrangian is invariant under the following \textit{gauge change operation}:
\begin{align}\label{r-I.2}
    \mbox{YM}(A^g)=\mbox{YM}(A),\quad\mbox{ where }\quad A^g:=g^{-1}\,dg+g^{-1}\,A\,g
\end{align}
for any choice of map $g$ from $\b^n$ into $G$. These maps are called \textit{local gauge transformations} or simply \textit{local gauges} and they realize the so called {\it local gauge group}. The equality between the Yang--Mills energy of $A$ and the Yang--Mills energy of $A^g$ is a direct consequence of the following identity\footnote{\,The action of the local gauge group on a given connection --- i.e. an equivariant horizontal plane distribution in the principal $G$-bundle which is represented in a local trivialization by a $\g$ valued 1-forms $A$ in the base --- is converted into the adjoint action of $G$ on $\g$ at the level of the curvature which is itself measuring the lack of integrability of this plane distribution (recall that the curvature is the vertical projection of the bracket of horizontal lifts of vector fields in the base).}
\begin{align}\label{r-I.2-a}
    F_{A^g}=g^{-1}F_Ag \quad\Rightarrow\quad |F_{A^g}|=|F_A|.
\end{align}
This huge invariance group is both a source of difficulties and a big advantage in studying this Lagrangian. 

Our starting point consists in considering the space of connection forms modulo this gauge group action. One of the main challenges in the field consists in proving the existence of a gauge $g$ in which the connection form $A^g$ is ``optimally'' controlled\footnote{The adjective ``optimally'' refers to the smallest possible classical function space that $A^g$ belongs to.} by its Yang--Mills energy $\text{YM}(A^g)=\text{YM}(A)$. For instance, in order to make the functional \text{YM} as much coercive as possible, a reasonable quest suggested by the abelian case ($G=U(1)$) in electromagnetism consists in looking for the  {\it Coulomb condition} to be fulfilled. This amounts to finding a local gauge $g$ such that
\begin{align}\label{r-I.3}
    \begin{cases}
        d^\ast A^g=0 & \mbox{ in }\b^n\\
        \langle A^g, x\rangle=0 & \mbox{ on }\p\b^n.
    \end{cases}
\end{align}
This condition is equivalent to the following non-linear elliptic PDE
\begin{align}\label{r-I.4}
    \begin{cases}
        -\text{div}\lf(g^{-1}\nabla g\rg)=\mbox{div}(g^{-1}Ag)\ & \mbox{ in }\b^n\\
        -g^{-1}\p_\nu g= g^{-1}\lf<A,x\rg> g & \mbox{ on }\p\b^n.
    \end{cases}
\end{align}
We shall come back to the difficulty of solving \eqref{r-I.4} later in this introduction but, assuming such a $g$ has been obtained, we control
\begin{align}\label{r-I.5}
    C^{-1}\|A^g\|^2_{W^{1,2}(\b^n)}\le \int_{\b^4}|dA^g|^2+|d^\ast A^g|^2\ d\L^n\le  \mbox{YM}(A)+\int_{\b^n}|A^g\wedge A^g|^2\ d\L^n.
\end{align}
From the Sobolev embedding theorem, for $n>2$ we have\footnote{For $n=2$, we have $$W^{1,2}(\b^n)\hookrightarrow \bigcap_{p=1}^{+\infty}L^{p}(\b^n).$$} 
\begin{align}\label{r-I.6}
    W^{1,2}(\b^n)\hookrightarrow L^{\frac{2n}{n-2}}(\b^n).
\end{align}
In dimension $n=3,4$,  H\"older inequality implies that $ L^{\frac{2n}{n-2}  }(\b^n) \hookrightarrow L^4(\b^n)$. Hence, for $n\le 4$ we obtain the bound
\begin{align}\label{r-I.7}
    \int_{\b^n}|A^g\wedge A^g|^2\ d\L^n\le \|A^g\|_{L^4(\b^n)}^4\le C\|A^g\|^4_{W^{1,2}(\b^n)},
\end{align}
and finally
\begin{align}\label{r-I.8}
    \|A^g\|^2_{W^{1,2}(\b^n)}\le C\,\mbox{YM}(A)+\ C\|A^g\|^4_{W^{1,2}(\b^n)}.
\end{align}
Assuming now that a smallness condition of the form $C\|A^g\|^2_{W^{1,2}(\b^n)}\le 2^{-1}$ is known, one gets in return the ``a priori'' estimate
\begin{align}\label{r-I.9}
    \|A^g\|^2_{W^{1,2}(\b^n)}\le C\,\mbox{YM}(A).
\end{align}
One of the main achievements of \cite{uhlenbeck-connections-with-lp} is to convert the a priori estimate \eqref{r-I.9} into an existence result for \eqref{r-I.4} such that \eqref{r-I.9} eventually holds, provided $\mbox{YM}(A)$ is small enough. 

This achievement is not straightforward at all in dimension $n\ge 4$. A first attempt would be to use the variational nature of the problem. Indeed, equation \eqref{r-I.4} happens to be the Euler--Lagrange equation of
\begin{align}\label{r-I.10}
    \|A^g\|_{L^2(\b^n)}^2=\int_{\b^4}|dg+Ag|^2\ d\L^n.
\end{align}
Hence, independently of dimension, a solution to \eqref{r-I.4} can be obtained by a direct minimization of \eqref{r-I.10} and by applying the fundamental principles of the calculus of variations. Nevertheless, this variational strategy is hitting a serious regularity issue in the sense that, for a generic $A\in L^2(\b^n)$, a minimizer $g$ of \eqref{r-I.10} is a priori only in $W^{1,2}$ and in general \eqref{r-I.9} is not satisfied by any of these minimizers\footnote{Even for a smooth data $A$, minimizers to the Dirichlet energy for maps from $\b^4$ into $G$ are known to be at most in $W^{2,(\frac{3}{2},\infty)}(\b^4,G)$ (see \cite{schoen-uhlenbeck}) and are certainly not automatically in $W^{2,2}(\b^4,G)$, which must be the case if $A^g$ is in $W^{1,2}$.} 

Here, another dichotomy appears within the low dimensions $n\le 4$. For \underline{$n<4$}, because of the Sobolev embedding
\begin{align}
    W^{2,2}(\b^n)\hookrightarrow C^0(\b^n),
\end{align}
the group multiplication
\begin{align}\label{ad-1}
\begin{split}
    \mathcal{M}: W^{2,2}(\b^n,G)^2&\rightarrow W^{2,2}(\b^n,G)\\[2mm]
    (g,h)&\mapsto gh
\end{split}
\end{align}
is smooth between the two Banach manifolds $W^{2,2}(\b^n,G)^2$ and $W^{2,2}(\b^n,G)$. This allows to implement an argument based on the local inversion theorem in order to prove that any $A\in W^{1,2}(\wedge^1\b^n\otimes\g)$ satisfying $\operatorname{YM}(A)<\ep$ admits a local gauge $g\in W^{2,2}(\b^n,G)$ solving \eqref{r-I.4} and \eqref{r-I.9}. 

Coming now to the critical dimension 4, the group multiplication map $\mathcal M$ is still well-defined, since the group $G$ is assumed to be compact, but $\mathcal M$ ceases to be continuous. Moreover, $W^{2,2}(\b^n,G)$ looses its natural Banach manifold structure. These facts prevent implementing the strategy involving the direct use of the local inversion theorem. K. Uhlenbeck instead developed a very clever continuity argument leading to a $W^{1,2}$ controlled representative satisfying \eqref{r-I.9}, under small $\operatorname{YM}$-energy assumption. 

To summarize, the  balance between the linear part $dA$ and the bilinear part $A\wedge A$ of the curvature form can be settled in favor of the linear and more regularizing part $dA$ as long as $\text{YM}(A)$ is small enough and \underbar{up to dimension $4$}. This enables to implement classical variational strategies for $\mbox{YM}$ within the framework of \underbar{Sobolev connections}\footnote{\label{f}A Sobolev $W^{k,p}$-connection on a smooth principal $G$-bundle $\pi:P\rightarrow M^4$ is given by a collection of $W^{k,p}$ $\g$-valued 1-forms $A_U$ on each open set $U$ over which $P$ is trivial (i.e. $\pi^{-1}(U)\simeq U\times G $ as principal $G$-bundle isomorphism) and related to each other by the classical gauge  equivalence relations
\begin{align}
A_V:=g^{-1}_{UV} d g_{UV}+g^{-1}_{UV}\,A_U\, g_{UV}\quad \mbox{ on }U\cap V
\end{align}
where $g_{UV}$ are the $G$-valued transition functions defining the smooth bundle $P$ (see for instance \cite{FU}).} on \underbar{smooth principal $G$-bundles} (see e.g. \cite{sedlacek} for minimization procedures).

In dimension larger than $4$, the Sobolev embedding $W^{1,2}\hookrightarrow L^4$ does not hold anymore. This fundamental fact compromises Uhlenbeck's procedure to extract controlled gauges. Even worse, one can produce a sequence of smooth $\mathfrak{su}(2)$-valued 1 forms $\{A_k\}_{k\in\n}$ on $\b^5$ such that
\begin{align}\label{r-I.11}
    \begin{cases}
    A_k\rightharpoonup A_\infty & \mbox{ weakly in }L^2(\b^5)\\
    dA_k+A_k\wedge A_k\rightharpoonup F_\infty & \mbox{ weakly in }L^2(\b^5)\\
    d^\ast A_k=0 & \mbox{ in }\b^5.
\end{cases}
\end{align}
and
\begin{align}\label{r-I.12}
    \operatorname{Spt}\ d\lf(\operatorname{tr}(F_\infty\wedge F_{\infty})\rg)=\ov{\b^5}.
\end{align}
Assume that there exists $B\in W^{1,2}(\wedge^1\b^5\otimes\mathfrak{su}(2))$ such that
\begin{align}\label{r-I.13}
    F_\infty=dB+B\wedge B=F_B.
\end{align}
Then, for any smooth function $\varphi$ compactly supported in $\b^5$ the coarea formula gives
\begin{align}\label{r-I.14}
    \int_{B^5} d\varphi\wedge \mbox{tr}(F_B\wedge F_B)=\int_{-\infty}^{+\infty}d\L^1(s)\int_{\varphi^{-1}(s)}\mbox{tr}(F_B\wedge F_B).
\end{align}
Thanks to Sard's theorem and Fubini's theorem, for $\H^1$-a.e. $s\in\r$, $\varphi^{-1}(s)$ is a smooth closed 4-dimensional manifold and the restriction of $B$ to this submanifold is in $W^{1,2}$. On such a 4-dimensional manifold,  by a straightforward strong approximation procedure for $B$ in $W^{1,2}(\wedge^1\varphi^{-1}(s)\otimes\mathfrak{su}(2))$, we can derive the following identity from the classical expression of the \textit{transgression form} for the second Chern class:
\begin{align}\label{r-I.15}
    \mbox{tr}(F_B\wedge F_B)=d\lf[\mbox{tr}\lf(B\wedge dB+\frac{1}{3}\,B\wedge [B,B] \rg)\rg].
\end{align}
This implies that
\begin{align}\label{r-I.16}
    \int_{\varphi^{-1}(s)}\mbox{tr}(F_B\wedge F_B)=0 \quad\mbox{ for $\H^1$-a.e. }s\in\r.
\end{align}
Combining \eqref{r-I.14} and \eqref{r-I.16}, we get
\begin{align}\label{r-I.17}
    d\lf(\mbox{tr}(F_B\wedge F_B)\rg)=0\quad\mbox{ in }{\mathcal D}'(\b^5).
\end{align}
This fact is contradicting \eqref{r-I.12} and we conclude that the weak limit of the smooth curvatures $F_{A_k}$ on $\b^5$ cannot be, even locally, the curvature of a Sobolev $W^{1,2}$-connection. For sequences of smooth Yang--Mills fields (smooth critical points of YM) with uniformly bounded energy, the possibility for their weak limits\footnote{These objects are called \textit{admissible Yang--Mills connections} in \cite{tian}.} not to satisfy \eqref{r-I.17} is not excluded at all. In fact, we believe that it is possible to produce such sequences where \eqref{r-I.17} is violated.

\medskip

These facts left the  variational geometric analysis community in some perplexity and the following question arose naturally:

\medskip

\centerline{\textit{What is the space of weak limits of curvatures of smooth Yang--Mills connections}}

\centerline{\textit{in supercritical dimension?}}

\medskip 

The above considerations are excluding the space of the curvatures of Sobolev connections, which has been the unique framework adopted so far to approach the variational issues related to the Yang--Mills lagrangian in subcritical and critical dimension. 

The work in \cite{PR1} was motivated by this question and brought an answer to it in the abelian case. In this framework, a $2$-form $F$ on a closed oriented connected surface $\Sigma$ is the curvature of some complex line bundle $E$ over $\Sigma$ if and only if
\begin{align}\label{r-I.18}
    \int_{\Sigma}F\in2\pi{\Z}\ \quad\mbox{ and then }\quad c_1(E)=\int_{\Sigma} F.
\end{align}
The authors introduced on $\b^3$ the space of \textit{weak $L^p$-curvatures}
\begin{align}
    \mathcal{F}^p_{\Z}:=\lf\{F\in L^p(\b^3)\ :\ \forall\, \varphi\in\operatorname{Lip}_c(\b^3)\ \mbox{ and for $\H^1$-a.e. } s\in\r,\ \int_{\varphi^{-1}(s)} F\in 2\pi{\Z}\rg\}.
\end{align}
The main result in \cite{PR1} establishes that for every $p>1$ the space $\mathcal{F}^p_{\Z}$ is sequentially closed for the weak convergence in $L^p$. This statement and some complementary results were made more precise and extended to higher dimensions in \cite{caniato} and \cite{caniato-gaia}.

Inspired by the abelian case, M. Petrache and the second author introduced the space of \textit{weak connections} on a $n$-dimensional manifold $N^n$. The first main idea consists first, for the dimensions $n\le 4$, in ``wrapping together'' the space of principal $G$-connections for any possible principal $G$-bundles in a single definition. We define
\be
\label{ad-2}
{\mathbb A}_G(N^n):=\lf\{ \begin{array}{c}
\nabla:=(U_i,A_i)_{i\in I}\ ; \ (U_i)_{i\in I}\mbox{ realizes an open cover of $N^n$}\\[3mm]
A_i\in W^{1,2}(\wedge^1 U_i\otimes\g),\ 
\forall\ i\ne j\ \ \exists\, g_{ij}\in W^{2,2}(U_i\cap U_j, G)\ \mbox{ s. t. }\\[3mm]
A_j=A_i^{g_{ij}}:=g_{ij}^{-1}d{g_{ij}}+g_{ij}^{-1}\,A_i\,g_{ij}\quad\mbox{ on }U_i\cap U_j\\[3mm]
\forall\ i\ne j\ne k\quad g_{ij}\,g_{jk}\,g_{ki}\equiv\operatorname{id}_G\quad\mbox{ on }U_i\cap U_j\cap U_k\\[2mm]
|F_\nabla|\in L^2(N^n)
\end{array} 
\rg\}.
\ee 
This is the definition of Sobolev connections considered in the classical analytical works on gauge theory since the early  eighties (\cite{DK}, \cite{FU}). In order to implement analysis arguments, we need to find some natural generalization of Sobolev connections to higher dimension that enjoys a good closure property. Aiming to do so, it is tempting to work with a single form to represent the connection. The problem of constructing a global gauge to any Sobolev connection in dimension at most 4 was considered for the first time in \cite{petrache-riviere-global-gauges}. 
\begin{Th} 
\label{glob-gau} \cite{petrache-riviere-global-gauges} Let $\nabla\in{\mathbb A}_{\operatorname{SU}(2)}(N^n)$ be a Sobolev connection of some principal $\operatorname{SU}(2)$-bundle over a closed oriented Riemannian manifold $N^n$  of dimension $n\le 4$. Then there exists $A\in L^{4,\infty}(\wedge^1N^n\otimes {\mathfrak{su}(2)})$ such that, about every point, there exists locally a $W^{1,(4,\infty)}$ trivialization in which
\begin{align}
   \nabla=d+A, 
\end{align}
where $L^{4,\infty}$ is the weak Marcinkiewicz space 
\begin{align}
    L^{4,\infty}(N^n):\lf\{f \mbox{ is measurbale and } \ \sup_{\lambda>0}\lambda\,\lf| \lf\{ x\in N^n\ ;\ |f(x)|>\la\rg\}\rg|^{\frac{1}{4}}<+\infty\rg\}\
\end{align}
with $|\cdot|$ being the measure induced by the volume form of $N^n$.
\end{Th}
Representing a smooth connection on a non trivial bundle by a global 1-form is obviously impossible if one does not give up regularity. In fact, for instance, if the bundle is a non trivial $\operatorname{SU}(2)$-bundle over the 4-sphere, one can prove that no smooth connection $\nabla$ has a global representative in $L^4$. Hence, the regularity $L^{4,\infty}$ is optimal in that sense\footnote{ In \cite{petrache-riviere-global-gauges}, the authors were asking the question whether a global gauge in the optimal space $L^{4,\infty}(\wedge^1 N^n\otimes {\mathfrak{su}(2)})$ and satisfying  simultaneously the Coulomb condition $d^\ast A=0$. A partial answer to this question which is still open as such is given in \cite{WY}.}.

Theorem \ref{glob-gau} is leading naturally to the following definition
\begin{align}\label{ad-3}
    A_G(N^n):=\lf\{ A\in L^2(\wedge^1N^n\otimes\g) \mbox{ : } F_A\in L^2(N^n) \mbox{ and }\ \exists \mbox{ locally } g\in W^{1,2} \mbox{ s.t. }A^g\in W^{1,2}\rg\}.
\end{align}
Thanks to this theorem, for $n\le 4$ we have
\begin{align}\label{ad-4}
    A_{\operatorname{SU}(2)}(N^n)={\mathbb A}_{\operatorname{SU}(2)}(N^n)
\end{align}
Finally, again for $n\le 4$, one can consider the following apparently weaker definition
\begin{align}\label{weak-co-L4}
    \A_G(N^n):=\lf\{ A\in L^2(\wedge^1N^n\otimes\g) \mbox{ : } F_A\in L^2(N^n) \mbox{ and } \exists \mbox{ locally } g\in W^{1,2} \mbox{ s.t. } A^g\in L^4\rg\}.
\end{align}
This definition more flexible than the definition of $A_G$, because it extends to spaces $X^n$ which are bi-Lipschitz homeomorphic to a smooth Riemannian oriented closed manifold $N^n$ and $n\le 4$. Moreover, we prove in the Appendix A of the present paper that for any Riemannian manifold of dimension less or equal than $4$ and for any compact Lie group $G$
\begin{align}\label{ad-5}
    A_G(N^n)=\A_G(N^n).
\end{align}
In \cite{Riv}, a proof of the sequential weak closure of $\A_G(M^4)$ under Yang--Mills energy control is given for every closed oriented 4-dimensional Riemannian manifold $N^4$. The proof is based on the analysis mostly developed by K. Uhlenbeck during the 80s in \cite{uhlenbeck-connections-with-lp}, \cite{Uh2} and \cite{Uh3}, combined with some more recent arguments involving the use of interpolation spaces introduced in this context by the second author in \cite{Riv2}.

Coming now to the  supercritical dimensions, as underlined earlier one would also wish to produce a class of ``objects'' containing smooth connections and which is weakly sequentially closed under Yang--Mills energy control exclusively. In order to do so, the main idea in \cite{PR1} is to propose an inductive definition by mean of generic slicing.  Let $n>4$ and denote by ${\mathcal L}(N^n)$ the space of Lipschitz functions whose level sets are almost always bi-Lipschitz equivalent to a smooth manifold. We introduce the following space:
\begin{align}
    \label{weak-co-high}
    \A_G(N^n):=\lf\{ 
    \begin{array}{c}
    A\in L^2(\wedge^1 N^n\otimes\g) \mbox{ : } F_A\in L^2(N^n)\\[3mm]
    \,\mbox{ s.t. }\, \forall\, \varphi\in {\mathcal L}(N^n)
    \,\mbox{ and for $\H^1$-a.e. } s\in\r\\[3mm]
    \iota_{\varphi^{-1}(s)}^\ast A\in \A_G(\varphi^{-1}(s))
    \end{array}
    \rg\},
\end{align}
where $\iota_{\varphi^{-1}(s)}$ is the canonical inclusion of $\varphi^{-1}(s)$ in $M^m$.

\medskip
In \cite{petrache-riviere-na}, a proof of the  sequential weak closure of $\A_G(M^5)$ under Yang--Mills energy control is proposed\footnote{The proof of the sequential weak closure in \cite{petrache-riviere-na} is based on the Proposition 2.1 in the same paper. As discussed in Remark \ref{rm-PR}, \cite[Proposition 2.1]{petrache-riviere-na} has the missing term $\|A\|_{L^2}^2$ on the right-hand-side of the inequality (2.2). This was first noticed by S. Sil. Our Proposition \ref{Proposition: harmonic extension on the good cubes} below is a suitable replacement of \cite[Proposition 2.1]{petrache-riviere-na} as explained in section II. We give a complete and detailed proof of the sequential weak closure of $\A_G(M^5)$ under controlled Yang--Mills energy in \cite{CaRi1}}. More precisely, the authors show that
\begin{align}
\label{r-I.19}
\begin{cases}
\,\,\forall\ \{A_k\}_{k\in\n}\subset\A_G(N^5)\,\mbox{ s.t. }\, \displaystyle{\limsup_{k\rightarrow +\infty}\mbox{YM}(A_k)<+\infty}\\[3mm]
\,\,\exists\,\{A_{k_n}\}_{n\in {\N}},\,\,A\in\A_G(N^5)\,\mbox{ and }\,\{g_n\}_{n\in\n}\subset W^{1,2}(N^5,G)\\[3mm]
\,\,\mbox{s.t. }A_{k_n}^{g_n}\rightharpoonup A\ \mbox{ in }L^2(N^5)\,\mbox{ and }\,\mbox{YM}(A)\le\displaystyle{\liminf_{k\rightarrow +\infty}\mbox{YM}(A_{k_n})}.
\end{cases}
\end{align}
The proof of \eqref{r-I.19} uses a strong approximation property of elements in $\A_G(\b^5)$ by connection forms which are smooth away from finitely many points, modulo gauge transformations. More specifically, in \cite{petrache-riviere-na} the authors show that for any $A\in\A_G(\b^5)$ there exists a sequence of $\{A_k\}_{k\in\n}\subset\A_G(\b^5)$ such that each $A_k$ is gauge equivalent to a connection form in $\b^5$ which is smooth away from finitely many points and
\begin{align}
\label{r-I.20}
\begin{cases}
    \, A_k\rightarrow A & \mbox{ strongly in }L^2(\b^5),\\[2mm]
    \, F_{A_k}\rightharpoonup F_A & \mbox{ weakly in }L^2(\b^5),\\[2mm]
    \,\mbox{tr}\lf(F_{A_k}\wedge F_{A_k}\rg) \rightharpoonup\mbox{tr}\lf(F_{A}\wedge F_{A}\rg) & \mbox{ weakly in }{\mathcal D}'(\b^5).
\end{cases}
\end{align}
The space $\mathcal F_G(\b^5)$ of smooth connections away from isolated points is the smallest space such that the strong approximation property \eqref{r-I.20} holds true. This makes ${\mathcal F}_G(\b^5)$ a natural subspace in $\A_G(\b^5)$, in the same way as the space $R^\infty(\b^3,\s^2)$ of maps in the Sobolev space $W^{1,2}(\b^3,\s^2)$ that are smooth away from finitely many isolated topological singularities is the smallest subspace in $W^{1,2}(\b^3,\s^2)$ being sequentially dense with respect to the $W^{1,2}$-norm (see \cite[Theorem 4]{bethuel-zheng}).

Because of the sequential weak closure property \eqref{r-I.19}, $\A_G(N^5)$ is a space in which variational problems related to the Yang--Mills lagrangian on $N^5$ are well-posed. We can then define the notion of weak Yang--Mills connections.
\begin{Dfi}[Weak Yang--Mills connections]
    Let $G$ be a compact matrix Lie group. We say that $A\in\A_G(\b^5)$ is a \textit{weak Yang--Mills connection} on $\b^5$ if 
    \begin{align}\label{ym}
        d_A^*F_A=0 \quad\Longleftrightarrow \quad \sum_{i=1}^5\p_{x_i}(F_A)_{ij}+[A_i,(F_A)_{ij}]\ \quad\forall \ j=1\cdots 5\quad\mbox{ distributionally on } \b^5,
    \end{align}
    i.e.
    \begin{align}
        \int_{\b^5}\lf<F_A,d_A\varphi\rg>=0 \qquad\forall\,\varphi\in C_c^{\infty}(\wedge^1\b^5\otimes\g),
    \end{align}
where we have
\begin{align*}
   d_A\varphi:=d\varphi+[A\wedge\varphi]=d\varphi+A\wedge\varphi+\varphi\wedge A, \qquad\forall\varphi\in C_c^{\infty}(\wedge^1\b^5\otimes\g).
\end{align*}
and $\langle\,\cdot\,,\,\cdot\,\rangle$ denotes the scalar product on $\wedge^2\b^5\otimes{\mathfrak g}$. 
\end{Dfi}
Among weak Yang--Mills connections,  we shall be particularly interested with the ones that satisfy the following \textit{stationarity condition} (which is automatically satisfied by smooth solutions to \eqref{ym} or by YM-energy minimizers for instance).
\begin{Dfi}[Stationary weak Yang--Mills connections]
    Let $G$ be a compact matrix Lie group. We say that a weak Yang--Mills connection $A\in\A_G(\b^5)$ on $\b^5$ is stationary if 
    \begin{align}\label{stationarity condition-intro}
        \left.\frac{d}{dt}\right|_{t=0}\operatorname{YM}(\Phi_t^*A)=0,
    \end{align}
    for every smooth $1$-parameter group of diffeomorphisms $\Phi_t$ of $\b^5$ with compact support.
\end{Dfi}
%
% \begin{Rm}\label{Remark: gauge equivalence preserve stationary weak Yang-Mills property}
%     Note that if $A$ is a stationary weak Yang--Mills connection and $\tilde A\in W^{1,2}\cap L^4$ is another $\g$-valued $1$-form on $\b^5$ such that $\tilde A=A^g$ for some $g\in W^{1,2}(\b^5,G)$, then $\tilde A$ is a stationary weak Yang--Mills connection as well. The proof of such fact follows by direct computation exploiting the gauge invariance of the Yang--Mills functional.
% \end{Rm}
%
If $A$ is a stationary Yang--Mills connection, by standard methods it can be shown that the following \textit{monotonicity property} holds true: for every given $x\in \b^5$, the function
\begin{align}\label{monotonicity-intro}
    \big(0,\dist(x,\partial\b^5)\big)\ni\rho\to\frac{e^{c\Lambda\rho}}{\rho}\int_{B_{\rho}(x)}\lvert F_A\rvert^2\, d\L^5
\end{align}
is non-decreasing, where $c>0$ is a universal constant and $\Lambda$ depends on $B_1(x)$. In particular, we have 
\begin{align*}
 \sup_{\substack{x\in B_{\frac{1}{2}}(0),\\0<\rho<\frac{1}{4}}}\frac{1}{\rho}\int_{B_\rho(x)}|F_A|^2\ d\L^5 \le C\ \int_{B_1^5(0)}|F_A|^2\ d\L^5,
\end{align*}
for some constant $C>0$ independent on $A$. We then naturally introduce the following spaces which are known as Morrey-Sobolev space for any domain $\Om\subset \R^n$
\begin{align}
    M^0_{p,q}(\Om):=\Bigg\{f\in L^p(\Omega) \mbox{ : }  \lvert f\rvert_{M^0_{p,q}(\Om)}^p:=\sup_{\substack{x\in\Om,\\\rho>0}}\frac{1}{\rho^{n-pq}}\int_{B_\rho(x)\cap\Om}|f|^p\ d{\mathcal L}^n\Bigg\}.
\end{align}
It is strongly motivated by the analysis of weak stationary Yang--Mills Fields to ask whether Uhlenbeck's Coulomb gauge extraction extends in the higher dimension 5 to weak connections having a curvature with small Morrey $M^0_{2,2}$-norm. The following theorem answers positively to this question and is one of the main results of the present work.
\begin{Th}
\label{th-main}
Let $G$ be a compact matrix Lie group. There exists $\ep_G>0$ such that for every weak connection $A\in\mathfrak{a}_G(\b^5)$ satisfying
\begin{align}\label{ad-6}
    \lvert F_A\rvert_{M^0_{2,2}(\b^5)}^2=\sup_{\substack{x\in\b^5,\\\rho>0}}\frac{1}{\rho}\int_{B_\rho(x)\cap\b^5}\lvert F_A\rvert^2\ d\L^5<\ep_G
\end{align}
there exists $g\in W^{1,2}(\b^5,G)$ such that
\begin{align}\label{ad-7}
    d^\ast A^g=0,
\end{align}
and
\begin{align}\label{ad-8}
    \rvert\nabla A^g\rvert_{M^0_{2,2}(\b^5)}^2=\sup_{\substack{x\in\b^5,\\\rho>0}}\frac{1}{\rho}\int_{B_\rho(x)\cap\b^5}\sum_{i=1}^5|\p_{x_i}A^g|^2\ d{\mathcal L}^5\le C_G\lvert F_A\rvert_{M^0_{2,2}(\b^5)}^2,
\end{align}
where $C_G>0$ is a constant depending only on $G$.
\end{Th}
This result has been conjectured to hold in \cite{petrache-riviere-na}. Such a Coulomb gauge extraction theorem has been first established in \cite{tian} for smooth connections and in \cite{meyer-riviere} under the assumption that the connection can be approximated strongly by smooth connections with curvatures having small Morrey norm \eqref{ad-6}. Later on, the same statement was proved in a particular case assuming that the connection is a weak limit of smooth Yang--Mills fields (see \cite{tao-tian}). Such weak limits are smooth away from a closed codimension 4 rectifiable set and and referred to as \textit{admissible Yang--Mills connections}. We also remark that energy identities and bubbling analysis for Yang--Mills fields in supercritical dimension are due to the subsequent works of the second author and A. Naber--D. Valtorta, in \cite{riviere-energy-quantization} and \cite{naber-valtorta} respectively.

In \cite{tao-tian}, the authors prove a strong approximability property of admissible Yang--Mills connections by smooth connections with small Morrey norm (see \cite[Proposition 4.4]{tao-tian}). In \cite[Theorem 34]{swarnendu}, the author shows the existence of Morrey norm controlled local Coulomb gauges in supercritical dimension by exploiting an approximation procedure, in the same spirit as in \cite{meyer-riviere}. However, approximating connections in the (stronger) Morrey norm requires additional assumptions\footnote{In particular, in \cite[Theorem 34]{swarnendu} the author exploits a ``vanishing Morrey norm'' condition.} which are not available in our context. 

The main achievement of the present paper is to prove that {\bf any weak connection satisfying \eqref{ad-6}} (which includes all the previous cases) can be approximated by smooth connections with small Morrey norms.

Combining Theorem \ref{th-main} and the main result in \cite{meyer-riviere} we can derive the following $\eps$-regularity statement by using the same arguments presented in \cite[Section 4]{meyer-riviere}. 
\begin{Th}[$\eps$-regularity]\label{Theorem: eps-regularity-intro}
    Let $G$ be a compact matrix Lie group. There exists $\eps_G\in (0,1)$ such that for every stationary weak Yang--Mills $A\in\A_G(\b^5)$ satisfying 
    \begin{align*}
        \operatorname{YM}(A)=\int_{\b^5}\lvert F_A\rvert^2\, d\L^5<\eps_G
    \end{align*}
    there exist $g\in W^{1,2}(B_{\frac{1}{2}}(0),G)$ such that $A^g\in C^{\infty}(B_{\frac{1}{2}}(0))$.
\end{Th}
Finally, standard covering arguments give the following bound on the singular set of stationary weak Yang--Mills connections, which is the main result of the present paper. 
\begin{Th}\label{Theorem: regularity for stationary weak Yang-Mills connections-intro}
     Let $G$ be a compact matrix Lie group and let $A\in\A_G(\b^5)$ be a stationary weak Yang--Mills connection on $\b^5$. Then
     \begin{align*}
         \H^1(\operatorname{Sing}(A))=0,
     \end{align*}
     where $\H^1$ is the $1$-dimensional Hausdorff measure in $\r^5$ and $\operatorname{Sing}(A)\subset\b^5$ is the singular set of $A$, given by $\operatorname{Sing}(A):=\b^5\smallsetminus\operatorname{Reg}(A)$ where
     \begin{align*}
         \operatorname{Reg}(A):=\lf\{x\in\b^5 \mbox{ s.t. } \exists \,\rho>0,\ g\in W^{1,2}(B_{\rho}(x),G) \mbox{ s.t. } A^g\in C^{\infty}(B_{\rho}(x))\rg\}.
     \end{align*}
\end{Th}
\subsection*{Organization of the paper} 
As mentioned above, the proof of the main Theorem \ref{th-main} consists in approximating strongly in $L^2$ every weak connection $A$ (i.e. every $A\in\mathfrak{a}_G(\b^5)$) whose curvature has a small Morrey $M^0_{2,2}$-norm by a sequence of smooth $\g$-valued 1-forms $A_j$ having small Morrey $M^0_{2,2}$-norms. 

Section II is devoted to the construction of the building blocks. It contains two main results, explaining how to extend in an ``optimal way'' inside a cube a given connection at its boundary. We propose two ways of extending this connection. First, we assume some smallness condition on the $L^2$-norms both of the curvature and of the connection at the boundary (Corollary \ref{Corollary: extension in the interior of good cubes}). Then, we assume smallness of the $L^2$-norm of the curvature only (Corollary \ref{Corollary: extension in the interior of bad cubes}). 

In Section III we introduce some terminology. We say that a cube is \textit{good} if on its boundary both the $L^2$-norms of the curvature and of the connection are small. We call a cube \textit{bad} if on its boundary just the $L^2$-norm of the curvature is small.\footnote{This terminology in the dichotomy between good and bad cubes is reminiscent of the one introduced in the framework of the strong approximation of  weak connections in $\b^5$ (\cite{petrache-riviere-na} and \cite{CaRi1}). Nevertheless, the meaning that we associate to it is different, in the sense that assuming the smallness condition of the Morrey norm allows us to decompose the domain into cubes at the boundary of which the $L^2$-norm of the curvature is always small. The dichotomy between good and bad cubes is made on the base of the smallness of $L^2$-norm of the connection only.} Then, in a second step, by the mean of the coarea formula and the mean value theorem we prove the existence of a so called \textit{admissible covers} by small cubes of comparable sizes, so that the $L^2$-norm of the curvature is small on the boundary of each of the cubes. 

Section IV is devoted to what is called the ``first smoothification''. In the first smoothification, we replace the initial connection in every cube of a chosen admissible cover by mean of the extensions introduced in the previous section. When the cube is good, we use Corollary \ref{Corollary: extension in the interior of good cubes} whilst, if the cube is bad, we exploit by Corollary \ref{Corollary: extension in the interior of bad cubes}. The main result in Section IV is Theorem \ref{Theorem: approximation under controlled traces of the curvatures}. The consequence of the ``first smoothification'' is that the newly obtained connection forms $A_{i,\La}$ converging strongly in the $L^2$-norm to $A$ as $i\rightarrow +\infty$ and $\La\rightarrow +\infty$, still having a small Morrey norm of the curvature, enjoys the following property: every trace on a generic cube of size comparable to the size of the admissible covering satisfies the small $L^2$-condition. 

In Section V we shall proceed to the ``second  smoothification'', that is, the $L^2$ strong approximation of the initial weak connection $A$ by a sequence $A_i$ of smooth connections whose curvature has small Morrey $M^0_{2,2}$-norm (Theorem \ref{Theorem: smooth approximation under controlled Morrey norm}). The proof of the second smoothification goes as follows. Starting from the approximating sequences $A_{i,\La}$ given by the first smoothification, we take a grid of size comparable to the size of the admissible cover associated to $A_{i,\La}$ and we apply the replacement results of Section II (the building blocks Corollary \ref{Corollary: extension in the interior of good cubes} and Corollary \ref{Corollary: extension in the interior of bad cubes}) on each of the disjoint cubes of the grid iteratively in such a way that, in the procedure, each cube to be replaced is facing at least one cube which has not been replaced yet. 

In Section VI, we combine the approximation given by the second smoothification with the main result in \cite{meyer-riviere} in order to prove our main theorem (Theorem \ref{th-main}).

\subsection*{Acknowledgements}
The authors express their gratitude to Mircea Petrache and Swarnendu Sil for the useful preliminary discussions about these topics. Additionally, the first author would like to extend their appreciation to Federica Cecchetto, Sofia Fatigoni, Federico Franceschini, Filippo Gaia, Alessandro Liberatore and Martijn S. S. L. Oei for the insightful conversations and thoughtful support during the writing of this paper. A large part of this work has been written while the two authors were members of the Simons Laufer Mathematical Sciences Institute. They would like to thank the Institute for the perfect working conditions and the hospitality.
%
%\newpage 
%
\section{The building blocks for the approximation theorems}\label{section: the building blocks for the approximation theorems}
\subsection{Extension of weak connections}\label{subsection: extension of weak connections}
In this subsection we build the fundamental statements that will be used in Section \ref{section: approximation of weak connections under Morrey norm control} in order to prove the strong $L^2$-approximation theorems for weak connections (Theorem \ref{Theorem: approximation under controlled traces of the curvatures} and Theorem \ref{Theorem: smooth approximation under controlled Morrey norm}) under Morrey norm control. In particular, we will need Corollaries \ref{Corollary: extension in the interior of good cubes} and \ref{Corollary: extension in the interior of bad cubes} to extend weak connections from the boundary of $5$-cubes to their interior. If we can assume the $L^2$-smallness of both the connection and its curvature on the boundary of the cube, then we will use Corollary \ref{Corollary: extension in the interior of good cubes}. In case we can only assume the  $L^2$-smallness of the curvature of the connection on the boundary, we will exploit Corollary \ref{Corollary: extension in the interior of bad cubes}.
\subsubsection{Extension under \texorpdfstring{$L^2$}{Z}-smallness of the connection and its curvature}
\allowdisplaybreaks
To ease the reading, throughout this subsection we will denote by $``d_{\r^5}"$ and $``d_{\s^4}"$ the standard differential of $k$-forms respectively on $\r^5$ and on the round sphere $\s^4$. More precisely we have
\[
d_{\s^4}:=\iota_{\s^4}^\ast d_{\r^5},
\]
where $\iota_{\s^4}$ is the canonical embedding of $\s^4$ into $\r^5$. The following proposition is one of the building blocks of our approximation procedure.

\begin{Prop}[Harmonic extension under smallness condition on $F_A$ and $A$]\label{Proposition: harmonic extension on the good cubes}
Let $G$ be a compact matrix Lie group and let $f:\r^5\to\r^5$ be a bi-Lipschitz homeomorphism such that $f\in W_{loc}^{1,\infty}(\r^5,\r^5)$ and
\begin{align}\label{I.1}
    \langle d_{\s^4}f_i,d_{\s^4}f_j\rangle_{L^2(\s^4)}=0, \qquad\forall\, i,j=1,...,5\ \mbox{ s.t. }\ i\neq j. 
\end{align}
Let
\begin{align}
    C_f:=\sqrt{\sum_{i=1}^5\frac{1}{\|d_{\s^4}f_i\|_{L^2(\s^4)}^2}}.
\end{align}
There are constants $\eps(G,f)\in(0,1)$ and $C(G,f)$ depending only on $G$, $C_f$ and on $\|d_{\r^5}f\|_{L^{\infty}(\b^5)}$ such that for any $A\in\A_G(\s^4)$ satisfying
\begin{align}\label{equation: smallness condition on S^4}
    \|F_{A}\|_{L^2(\s^4)}+\|A\|_{L^2(\s^4)}<\eps(G,f)
\end{align}
the following facts hold.
\begin{enumerate}[(i)]
    \item There exist $g\in W^{1,2}(\s^4,G)$ and a $\g$-valued 1-form $\tilde{A}\in L^5(\wedge^1\b^5\otimes\g)$ such that
    \begin{align*}
         \iota_{\s^4}^*\tilde A=A^{g}
    \end{align*}
    and 
    \begin{align}
        \nonumber
        \|A^g\|_{L^4(\s^4)}&\le C(G,f)\ \lf(\|F_A\|_{L^2(\s^4)}+\|A-\iota_{\s^4}^*f^*\bar A\|_{L^2(\s^4)}+\|A\|_{L^2(\s^4)}\rg),\\
        \label{VIII-0}
        \|d_{\s^4}g\|_{L^2(\s^4)}&\le C(G,f)\ \lf(\|F_A\|_{L^2(\s^4)}+\|A-\iota_{\s^4}^*f^*\bar A\|_{L^2(\s^4)}+\|A\|_{L^2(\s^4)}^2\rg),\\
        \label{VIII}
        \ds \|F_{\tilde A}\|_{L^{\frac{5}{2}}(\b^5)}&\le C(G,f)\ \lf(\|F_{A}\|_{L^2(\s^4)}+\|A   \|_{L^2(\s^4)}\lf\|A-\iota_{\s^4}^*f^*\bar A\rg\|_{L^2(\s^4)}+\|A\|_{L^2(\s^4)}^3\rg),
    \end{align}
    for every constant $\g$-valued $1$-form $\bar A$ on $\r^5$.
    \item There exists $\tilde g\in W^{1,2}(\b^5,G)$ satisfying
    \begin{align}\label{ti-g}
        \|\ti{g}-\operatorname{id}_G\|_{L^4(\b^5)}+\|d\ti{g}\|_{L^2(\b^5)} \le C(G,f)\, \lf(\|F_A\|_{L^2(\s^4)}+\|A-\iota_{\s^4}^*f^*\bar A\|_{L^2(\s^4)}+\|A\|_{L^2(\s^4)}^2\rg),
    \end{align}
    for every constant $\g$-valued $1$-form $\bar A$ on $\r^5$, such that the $\g$-valued $1$-form $\hat A:=\tilde A^{\tilde g^{-1}}\in L^2(\b^5)$ satisfies the following properties.
    \begin{enumerate}[(a)]
        \item $F_{\hat A}\in L^2(\b^5)$.
        \item $\iota_{\s^4}^*\hat A=A\in L^2(\s^4)$.
        \item Let $\Omega\subset\r^5$ be an open set such that $\Omega\cap \b^5$ has a $4$-dimensional compact Lipschitz boundary which can be included in a union of $N$ submanifolds of $\ov{\b^5}$ of class $C^2$. Then we have
            \be
            \label{equation: estimate on the trace of the curvature}
                   \|F_{\hat A}\|_{L^2(\partial\Omega\cap\b^5))}\le K_G \lf(\|F_{A}\|_{L^2(\s^4)}+\|A\|_{L^2(\s^4)}\lf\|A-\iota_{\s^4}^*f^*\bar A\rg\|_{L^2(\s^4)}
             +\|A\|_{L^2(\s^4)}^3\rg)
            \ee
            and
            \be
            \label{L4-norm-tiAA}
            \|\ti{A}\|_{L^4(\partial\Omega\cap\b^5)}\le K_G\, \lf(\|F_{A}\|_{L^2(\s^4)}+\|A\|_{L^2(\s^4)}\lf\|A-\iota_{\s^4}^*f^*\bar A\rg\|_{L^2(\s^4)}
             +\|A\|_{L^2(\s^4)}\rg),
             \ee
             for every constant $\g$-valued $1$-form $\bar A$ on $\r^5$, where $K_G=K_G(\Omega\cap\b^5)>0$ depends only on $G$ and on $\Omega\cap\b^5$ (that is, on the number $N$ of submanifolds containing $\p(\Om\cap\b^5)$ as well as their $C^2$ norms), on $C_f$ and on $\|d_{\r^5}f\|_{L^{\infty}(\b^5)}$.
        \end{enumerate}
    \end{enumerate}
    Moreover,
    \begin{align}\label{equation: estimate for the harmonic extension}
        \ds\big\|\ti{A}-f^*\bar A\big\|_{L^5(\b^5)}&\le C(G,f)\,\lf(\|F_A\|_{L^2(\s^4)}+\|A-\iota_{\s^4}^*f^*\bar A\|_{L^2(\s^4)}+\|A\|_{L^2(\s^4)}^2\rg),\\
        \label{d-tiA}
        \ds\big\|d\ti{A}\big\|_{L^{\frac{5}{2}}(\b^5)}&\le C(G,f)\,\lf(\|F_A\|_{L^2(\s^4)}+\|A-\iota_{\s^4}^*f^*\bar A\|_{L^2(\s^4)}+\|A\|_{L^2(\s^4)}^2\rg),\\
        \ds\big\|\hat{A}-f^*\bar A\big\|_{L^2(\b^5)}&\le C(G,f)\,\lf(\|F_A\|_{L^2(\s^4)}+\|A-\iota_{\s^4}^*f^*\bar A\|_{L^2(\s^4)}+\|A\|_{L^2(\s^4)}^2\rg),
        \label{hat-A-L2}
    \end{align}   
    for every constant $\g$-valued $1$-form $\bar A$ on $\r^5$.
\end{Prop}
\begin{proof}[\textbf{\textup{Proof of Proposition \ref{Proposition: harmonic extension on the good cubes}}}]  
Notice that since $f=(f_1,...,f_5)$ is a bi-Lipschitz homeomorphism, we have that $\|d_{\s^4}f_i\|_{L^2(\s^4)}\neq 0$ for every $i=1,...,5$. Let $\xi$ be the constant $\g$-valued $1$-form on $\r^5$ given by
\begin{align}
    \xi:=\sum_{i=1}^5\frac{1}{\|d_{\s^4}f_i\|_{L^2(\s^4)}^2}\langle A,d_{\s^4}f_i\rangle_{L^2(\s^4)}d_{\r^5}x_i,
\end{align}
where $\{x_i\}_{i=1,...,5}$ represent the standard euclidean coordinates on $\r^5$. We have
\begin{align}\label{equation: estimate on f^*xi}
    |\xi|&=\sqrt{\sum_{i=1}^5\lvert \xi_i\rvert^2}\le \sqrt{\sum_{i=1}^5\frac{1}{\|d_{\s^4}f_i\|_{L^2(\s^4)}^2}}\ \|A\|_{L^2(\s^4)}=C_f\|A\|_{L^2(\s^4)},
\end{align}
with
\begin{align}\label{cf}
    C_f:=\sqrt{\sum_{i=1}^5\frac{1}{\|d_{\s^4}f_i\|_{L^2(\s^4)}^2}}.
\end{align}
Let $\eta\in L^2(\s^4)$ be given by 
\begin{align}
    \eta:=A-\iota_{\s^4}^*f^*\xi=A-\sum_{i=1}^5\lf\langle A,\frac{d_{\s^4}f_i}{\|d_{\s^4}f_i\|_{L^2(\s^4)}}\rg\rangle_{L^2(\s^4)}\frac{d_{\s^4}f_i}{ \|d_{\s^4}f_i\|_{L^2(\s^4)}}.
\end{align}
Notice that, since $\eta$ is the $L^2$-orthogonal projection of $A$ on the linear subspace $\{d_{\s^4}f_i\}_{i=1,...,5}^{\perp}\subset L^2(\s^4)$, we have
\be
\label{etaL2}
    \|\eta\|_{L^2(\s^4)}\le\|A-\iota_{\s^4}^*f^*\bar A\|_{L^2(\s^4)},
\ee
for every constant $\g$-valued $1$-form $\bar A$ on $\r^5$. Moreover, 
\begin{align*}
    d_{\s^4}\eta=d_{\s^4}A-d_{\s^4}(\iota_{\s^4}^*f^*\xi)=d_{\s^4}A-\iota_{\s^4}^*f^*(d_{\r^5}\xi)=d_{\s^4}A.
\end{align*}
% Let $\varphi\in W_0^{3,q^*}(\b^5)$ be the unique solution of 
% \begin{align*}
%     \Delta_{\r^5}\varphi=-d_{\r^5}^*f^*\xi\in W^{1,q^*}(\b^5).
% \end{align*}
% Then, we have that $h_{\xi}:=f^*\xi+d_{\r^5}\varphi\in W^{2,q^*}(\b^5)$ satisfies
% \begin{align*}
%     \begin{cases}
%         d_{\r^5}^*h_{\xi}=d_{\r^5}^*f^*\xi+\Delta_{\r^5}\varphi=0 & \mbox{ on }  \b^5,\\
%         d_{\r^5}h_{\xi}=d_{\r^5}f^*\xi=0 & \mbox{ on } \b^5,\\
%         \iota_{\s^4}^*h_{\xi}=\iota_{\s^4}^*f^*\xi+\iota_{\s^4}^*d_{\r^5}\varphi=0 & \mbox{ on } \s^4,
%     \end{cases}
% \end{align*}
% where the last inequality follows because $\varphi$ has vanishing trace on $\s^4$. 
For future reference, we record the simple bound
\begin{align}\label{XVIII}
    \|f^*\xi\|_{L^\infty({\b^5})}\le C\|d_{\r^5}f\|_{L^{\infty}(\b^5)}\,\lvert\xi\rvert\le C\|d_{\r^5}f\|_{L^{\infty}(\b^5)}\,C_f\|A\|_{L^2(\s^4)}=C\tilde C_f\|A\|_{L^2(\s^4)},
\end{align}
where $C_f>0$ is given by \eqref{cf} and we let $\tilde C_f:=\|d_{\r^5}f\|_{L^{\infty}(\b^5)}\,C_f$.
\medskip

Since $d_{\r^5}(f^*\xi\wedge f^*\xi)=0$ and $f^*\xi\in L^\infty(\b^5)$, for every exponent $p\in[1,+\infty)$ there exists a unique\footnote{\,The form $\al$ can be obtained variationally by minimizing 
\[
\int_{\b^5}|\al|^2\ d\L^5
\]
among any form satisfying $d_{\r^5}\alpha=f^*\xi\wedge f^*\xi$.} $\alpha\in W^{1,p}(\wedge^1\b^5\otimes\g)$ such that
\begin{align*}
    \begin{cases}
 \ds       d_{\r^5}\alpha=f^*\xi\wedge f^*\xi\ &\mbox{ in } \b^5,\\[\sep]
  \ds      d_{\r^5}^*\alpha=0 &\mbox{ in } \b^5,\\[\sep]
     \ds   \alpha(\partial_r)=0 &\mbox{ in } \b^5.
    \end{cases}
\end{align*}
Moreover, from \cite{IM}, for any $p\in[1,+\infty)$ we have 
\begin{align}\label{equation: estimate for alpha}
     \|\alpha\|_{W^{1,p}({\b^5})}\le C_p\,\|f^*\xi\|^2_{L^{2p}(\b^5)}\le C_p\,\|f^*\xi\|^2_{L^{\infty}(\b^5)}\le C_p\tilde C^2_f\ \|A\|^2_{L^2(\s^4)},
\end{align}
where $C_p>0$ depends only on $p\in[1,+\infty)$. In particular, we deduce
\be
\label{C0al}
 \|\alpha\|_{L^\infty({\b^5})}\le C\tilde C_f^2\|A\|^2_{L^2(\s^4)}.
\ee
We define
\begin{align*}
    \omega:=\eta+\iota_{\s^4}^*\alpha=A-\iota_{\s^4}^*f^*\xi+\iota_{\s^4}^*\alpha\in L^2(\s^4). 
\end{align*}
Observe that
\begin{align}
\begin{split}
\label{Fomega}
    F_{\omega}&=F_{\eta+\iota_{\s^4}^*\alpha}=d_{\s^4}\eta+\iota_{\s^4}^*(f^*\xi\wedge f^*\xi)+\eta\wedge\eta+\eta\wedge\iota_{\s^4}^*\alpha\wedge+\iota_{\s^4}^*\alpha\wedge\eta+\iota_{\s^4}^*(\alpha\wedge\alpha)\\[\sep]
    &=\big(F_{\eta}+\iota_{\s^4}^*(f^*\xi\wedge f^*\xi)\big)+\eta\wedge\iota_{\s^4}^*\alpha\wedge+\iota_{\s^4}^*\alpha\wedge\eta+\iota_{\s^4}^*(\alpha\wedge\alpha).
\end{split}
\end{align}
Since $d(f^*\xi)=0$, we also have
\begin{align}
\begin{split}
\label{FA}
    F_{A}&=F_{\eta+\iota_{\s^4}^*f^*\xi}=F_{\eta}+\eta\wedge\iota_{\s^4}^*f^*\xi+\iota_{\s^4}^*f^*\xi\wedge\eta+\iota_{\s^4}^*(f^*\xi\wedge f^*\xi)\\[\sep]
    &=\big(F_{\eta}+\iota_{\s^4}^*(f^*\xi\wedge f^*\xi)\big)+\eta\wedge\iota_{\s^4}^*f^*\xi+\iota_{\s^4}^*f^*\xi\wedge\eta.
\end{split}
\end{align}
Combining \eqref{Fomega} and \eqref{FA} we obtain 
\be
\label{FomegaFA}
    F_{\omega}=F_A-\eta\wedge\iota_{\s^4}^*f^*\xi-\iota_{\s^4}^*f^*\xi\wedge\eta+\eta\wedge\iota_{\s^4}^*\alpha\wedge+\iota_{\s^4}^*\alpha\wedge\eta+\iota_{\s^4}^*(\alpha\wedge\alpha).
\ee
Thus,
\begin{align}
\label{FomegaL2}
    \nonumber
    \|F_{\omega}\|_{L^2(\s^4)}&\le C\lf(\|F_A\|_{L^2(\s^4)}+\|d_{\r^5}f\|_{L^{\infty}(\b^5)}\lvert\xi\rvert\|\eta\|_{L^2(\s^4)}+\|\alpha\|_{L^{\infty}(\s^4)}\|\eta\|_{L^2(\s^4)}+\|\alpha\|_{L^{\infty}(\s^4)}^2\rg)\\[\sep]
    \nonumber
    &\le C\lf(\|F_A\|_{L^2(\s^4)}+ \tilde C_f\|A\|_{L^2(\s^4)}\|\eta\|_{L^2(\s^4)}+\tilde C_f^2\|A\|_{L^2(\s^4)}^2\|\eta\|_{L^2(\s^4)}+\tilde C_f^4\|A\|_{L^2(\s^4)}^4\rg)\\[\sep]
    &\le C\lf(\|F_A\|_{L^2(\s^4)}+\tilde C_f(1+\tilde C_f\ep_G)\|A\|_{L^2(\s^4)}\|\eta\|_{L^2(\s^4)}+\tilde C_f^4\|A\|_{L^2(\s^4)}^4\rg)\\[\sep]
    \nonumber
    &\le C\lf(\|F_A\|_{L^2(\s^4)}+\tilde C_f(1+\tilde C_f\ep_G+\tilde C_f^3\eps_G^2)\|A\|_{L^2(\s^4)}^2\rg)\\[\sep]
    \nonumber
    &\le C\lf(\eps_G+\tilde C_f(1+\tilde C_f\ep_G+\tilde C_f^3\eps_G)\eps_G^2\rg),
\end{align}
for some universal constant $C>0$. By assumption, since $A\in\A_G(\s^4)$ and since the curvature of $A$ is small enough in $L^2$-norm, there exists $\hat g\in W^{1,2}(\s^4,G)$ such that $A^{\hat g}\in W^{1,2}(\s^4)$. We observe that
\begin{align*}
    \omega^{\hat g}&=\hat g^{-1}d_{\s^4}\hat g+\hat g^{-1}(\eta+\iota_{\s^4}^*\alpha)\hat g\\[\sep]
    &=\hat g^{-1}d_{\s^4}\hat g+\hat g^{-1}(A-\iota_{\s^4}^*f^*\xi+\iota_{\s^4}^*\alpha)\hat g=A^{\hat g}+\hat g^{-1}\iota_{\s^4}^*(\alpha-f^*\xi)\hat g\in L^4(\s^4).
\end{align*}
and 
\begin{align*}
    \|F_{\omega^{\hat g}}\|_{L^2(\s^4)}&=\|F_{\omega}\|_{L^2(\s^4)}\le C\lf(\eps_G+\tilde C_f(1+\tilde C_f\ep_G+\tilde C_f^3\eps_G^2)\eps_G^2\rg).
\end{align*}
If $\eps_G$ in \eqref{equation: smallness condition on S^4} is small enough, we can apply Proposition \ref{L4-gauge-sphere}  to $\omega^{\hat g}$ and we get the existence of a gauge $h\in W^{1,4}(\s^4,G)$ such that, by letting $g:=\hat g h\in W^{1,2}(\s^4,G)$, we have $d_{\s^4}^*\omega^g=0$ and
\be
\label{Agomegag}
\om^g=A^g-g^{-1} \iota_{\s^4}^*f^*\xi\, g+g^{-1}\,\iota_{\s^4}^*\alpha\, g.
\ee
Moreover,
\begin{align}
\begin{split}
\label{equation: change of gauge on S^4}
   \|\omega^g\|_{W^{1,2}(\s^4)}&\le C\|F_{\omega}\|_{L^2(\s^4)}\\[\sep]
   &\le C\lf(\|F_A\|_{L^2(\s^4)}+\tilde C_f(1+\tilde C_f\ep_G)\|A\|_{L^2(\s^4)}\|\eta\|_{L^2(\s^4)}+\tilde C_f^4\|A\|_{L^2(\s^4)}^4\rg),
\end{split}
\end{align}
where $C>0$ only depends on $G$. We have
\begin{align*}
    \|A^g\|_{L^4(\s^4)}&\le \|\om^g\|_{L^4(\s^4)}+\|\al-f^\ast\xi\|_{L^4(\s^4)}\\[\sep]
    &\le C(G,f)\ \lf(\|F_A\|_{L^2(\s^4)}+\|A-\iota_{\s^4}^*f^*\bar A\|_{L^2(\s^4)}+\|A\|_{L^2(\s^4)}\rg).
\end{align*}
Observe that by substituting $g$ with $g\,g_0$ for some $g_0\in G$ we still have $d_{\s^4}^*\omega^g=0$ and \eqref{equation: change of gauge on S^4} with the constant $C>0$ being unchanged. Since we have $d_{\s^4}g=g\omega^g-\omega g$, recalling that $\eps_G<1$ by assumption, we get
\begin{align}
\label{equation: estimate on gauge refined}
    \nonumber
    \|d_{\s^4}g\|_{L^2(\s^4)}&\le C\ \lf(\|\omega^g\|_{L^2(\s^4)}+\|\omega\|_{L^2(\s^4)}\rg)\\[\sep]
    \nonumber
    &\le C\lf(\|F_{\omega}\|_{L^2(\s^4)}+\|\eta\|_{L^2(\s^4)}+\|\al\|_{L^2(\s^4)}\rg)\\[\sep]
    &\le C\lf(\|F_A\|_{L^2(\s^4)}+\tilde C_f(1+\tilde C_f\ep_G+\tilde C_f^3\eps_G^2)\|A\|_{L^2(\s^4)}^2\rg)\\[\sep]
    \nonumber
    &\quad+C\|\eta\|_{L^2(\s^4)}+C\tilde C^2_f\ \|A\|^2_{L^2(\s^4)}\\
    \nonumber
    &\le C\lf(\|F_A\|_{L^2(\s^4)}+\|\eta\|_{L^2(\s^4)}+\hat C_f\|A\|_{L^2(\s^4)}^2\rg),
\end{align}
where $\hat C_f:=\tilde C_f(1+2\tilde C_f+\tilde C_f^3)$. Sobolev--Poincar\'e inequality gives the existence of $C>0$ (independent on $A$) such that
\begin{align}\label{equation: estimate on gauge I}
    \|g-\bar g\|_{L^4(\s^4)}\le C\|d_{\s^4}g\|_{L^2(\s^4)},
\end{align}
where $\bar g$ is the average of $g$ on $\s^4$. Thus, we deduce the existence of $x_0\in\s^4$ such that
\begin{align}\label{equation: estimate on gauge II}
    \lvert g(x_0)-\bar g\rvert\le C\|d_{\s^4}g\|_{L^2(\s^4)}.
\end{align}
Replacing $g$ by $gg^{-1}(x_0)$ and combining \eqref{equation: estimate on gauge I} and \eqref{equation: estimate on gauge II} we obtain
\begin{align}
\label{equation: estimate on the gauge}
    \|g-\operatorname{id}_G\|_{L^4(\s^4)}&\le C\, \|d_{\s^4}g\|_{L^2(\s^4)}\\[\sep]
    \nonumber
    &\le C\lf(\|F_A\|_{L^2(\s^4)}+\|\eta\|_{L^2(\s^4)}+\hat C_f\|A\|_{L^2(\s^4)}^2\rg).
\end{align}
We denote $\tilde g:=g(x/|x|)\in W^{1,2}(\b^5)$ the radial extension of $g$ in $\b^5$. A straightforward estimate gives
\begin{align}
\label{equation: estimate on the radial extension of the gauge}
    \nonumber
    \|\tilde g-\operatorname{id}_G\|_{L^4(\b^5)}&=\lf(\int_{\b^5}\lf|g\lf(\frac{x}{|x|}\rg)-\operatorname{id}_G\rg|^4\ d\L^5(x)\rg)^\frac{1}{4}\\[\sep]
    &\le\lf(\int_0^1 r^4\,d\L^1(r)\int_{\s^4}\lf|g-\operatorname{id}_G\rg|^4\ d\H^4\rg)^\frac{1}{4}\le\|g-\operatorname{id}_G\|_{L^4(\s^4)}\\[\sep]
    \nonumber
    &\le C\lf(\|F_A\|_{L^2(\s^4)}+\|\eta\|_{L^2(\s^4)}+\hat C_f\|A\|_{L^2(\s^4)}^2\rg).
\end{align}
Moreover, using \eqref{equation: estimate on gauge refined}, we get
\begin{align}
\label{dtildeg}
    \nonumber
    \lf(\int_{\b^5}|d\ti{g}|^2\ d\L^5\rg)^\frac{1}{2}&\le C\lf(\int_{\b^5}\lf\lvert d{g}\lf( \frac{x}{|x|}\rg)\rg\rvert^2\frac{1}{\lvert x\rvert^2}\, d\L^5(x)\rg)^\frac{1}{2}\\[\sep]
    &\le  C\lf(\int_0^1 r^2\, d\L^1(r)\rg)^{\frac{1}{2}}\lf(\int_{\s^4}\lvert d_{\s^4}g^2\rvert\, d\H^4\rg)^\frac{1}{2}\\[\sep]
    \nonumber
    &\le C\lf(\|F_A\|_{L^2(\s^4)}+\|\eta\|_{L^2(\s^4)}+\hat C_f\|A\|_{L^2(\s^4)}^2\rg).
\end{align}
Let $\tilde\omega$ be the unique minimizers of 
\begin{align}
    \inf\bigg\{\int_{\b^5}\big(\lvert d_{\r^5}C\rvert^2+\lvert d_{\r^5}^*C\rvert^2\big)\, dx^5 \mbox{ : } \iota_{\s^4}^*C=\omega^g\bigg\}.
\end{align}
Classical analysis for differential forms gives that $\tilde\omega$ solves
\begin{align}
    \begin{cases}
        d_{\r^5}^*\tilde\omega=0 & \mbox{ in } \b^5,\\
        d_{\r^5}^*d_{\r^5}\tilde\omega=0 & \mbox{ in } \b^5,\\
        \iota_{\s^4}^*\tilde\omega=\omega^g=A^g+g^{-1}\iota_{\s^4}^*(\alpha-f^\ast\xi) g & \mbox{ on } \partial\b^5
    \end{cases}
\end{align}
and that $\tilde\omega\in(W^{\frac{3}{2},2}\cap C^\infty)(\b^5)$. By classical elliptic regularity theory and Sobolev embedding theorem, we have the estimates
\begin{align}\label{equation: estimate on tilde omega}
\begin{split}
    \|\tilde\omega\|_{W^{1,\frac{5}{2}}(\b^5)}&\le  C\|\tilde\omega\|_{W^{\frac{3}{2},2}(\b^5)}\le C\|\omega^g\|_{W^{1,2}(\s^4)}\\
    &\le C\big(\|F_A\|_{L^2(\s^4)}+\|A\|_{L^2(\s^4)}\|\eta\|_{L^2(\s^4)}+\|A\|_{L^2(\s^4)}^4\big),
\end{split}
\end{align}
for some constant $C>0$ depending only on $G$ and on $f$ (notice that the last inequality follows from \eqref{equation: change of gauge on S^4}). Recall the continuous linear embedding
\begin{align}\label{Sobolev embedding}
    W^{1,\frac{5}{2}}(\b^5)\hookrightarrow L^5(\b^5).
    %\hookrightarrow L^4(\b^5).
\end{align}
Hence, in particular
\begin{align}
\begin{split}
\label{tiom}
    \|\ti{\om}\|_{L^5(\b^5)}&\le C\|\tilde\omega\|_{W^{1,\frac{5}{2}}(\b^5)}\le  C\|\tilde\omega\|_{W^{\frac{3}{2},2}(\b^5)}\le C\|\omega^g\|_{W^{1,2}(\s^4)}\\[\sep]
    &\le C\lf(\|F_A\|_{L^2(\s^4)}+\|A\|_{L^2(\s^4)}\|\eta\|_{L^2(\s^4)}+\|A\|_{L^2(\s^4)}^4\rg).
\end{split}
\end{align}
\medskip
We let $\tilde A\in L^5(\wedge^1\b^5\otimes{\g})$ be given by
\begin{align}
    \tilde A=\tilde\eta+f^*\xi=\tilde{\om}-\al+\zeta+f^*\xi,
\end{align}
where we denote
\begin{align}\label{tildeeta}
    \tilde{\eta}:=\tilde{\om}-\al+\zeta
\end{align}
and $\zeta$ is constructed as follows. We apply Proposition \ref{lm-exten} to $G$ and we find a smooth Riemannian manifold $M_G$ and a measurable map
\begin{align}
    \operatorname{Ext}(g):\,\,&M_G \longrightarrow  W^{1,\frac{5}{2}}(\b^5,G)\\[\sep]
    &\quad p \longmapsto\, g_p   
\end{align}
such that \eqref{res-prop}, \eqref{b-001} and \eqref{b-002} hold. We then choose
\begin{align}
    \zeta&:=\alpha-f^\ast\xi-\fint_{M_G}g_p^{-1}(\alpha-f^\ast\xi)g_p\, d\vol_{M_G}(p)\\[\sep]
    &\,=\fint_{M_G}\lf(\operatorname{id}_{G}-g_p^{-1}\rg)(\alpha-f^\ast\xi)g_p\, d\vol_{M_G}(p)+\fint_{M_G}(\al-f^\ast\xi)\lf(\operatorname{id}_{G}-g_p\rg)\, d\vol_{M_G}(p).
\end{align}
We notice that $\iota_{\s^4}^*\zeta=  0 $ and $\iota_{\s^4}^*\ti{\om}=A^g +g^{-1}\,\iota_{\s^4}(\al-f^\ast\xi)\,g $, so that $\iota_{\s^4}^*\tilde A=A^g$ by construction\footnote{ At this stage, in order to control the $L^2$-norm of $F_{\ti{A}}$ in $\b^5$ it could be tempting to estimate separately $\|d\ti{A}\|_{L^2(\b^5)}$ and $\|\ti{A}\wedge  \ti{A}\|_{L^2(\b^5)}$. This however  is  not going to lead to the expected estimate because each of the terms require to estimate respectively $\|d\al\|_{L^2(\b^5)}$ and $\|f^\ast\xi\wedge f^\ast\xi\|_{L^2(\b^5)}$ which would give in the r.h.s of \eqref{VIII} a term proportional to $\|A\|^2_{L^2(\b^5)}$ as in \cite[Proposition 2.1]{petrache-riviere-na} but which is preventing to prove the desired approximability property~\ref{r-I.20}. While in the combination $ d\ti{A}+\ti{A}\wedge \ti{A}$  these two contributions fortunately cancel each other which is one of the main advantage of this construction.}.

\medskip
\noindent
We have also
\begin{align}\label{dzeta}
\begin{split}
    d\zeta&:=f^\ast\xi\wedge f^\ast\xi -\fint_{M_G}g_p^{-1}\,f^\ast\xi\wedge f^\ast\xi \, g_p\ dp-\fint_{M_G}dg_p^{-1}\wedge(\al-f^\ast\xi)\ g_p\, d\vol_{M_G}(p)\\[\sep]
    &\quad+\fint_{M_G}g_p^{-1}\,(\al-f^\ast\xi)\wedge dg_p\, d\vol_{M_G}(p)\\[\sep]
    &\,=\fint_{M_G}\lf(\operatorname{id}_{G}-g_p^{-1}\rg)\,f^\ast\xi\wedge f^\ast\xi \ g_p\, d\vol_{M_G}(p)+\fint_{M_G}  f^\ast\xi \wedge f^\ast\xi \ \lf(\operatorname{id}_{G}-g_p\rg)\, d\vol_{M_G}(p)\\[\sep]
    &\quad-\fint_{M_G}dg_p^{-1}\wedge(\al-f^\ast\xi)\ g_p\, d\vol_{M_G}(p)+\fint_{M_G}g_p^{-1}\,(\al-f^\ast\xi)\wedge dg_p\, d\vol_{M_G}(p).
\end{split}   
\end{align}

\medskip
\noindent
Using \eqref{5demi}, we have
\begin{align}\label{new-zeta-1}
\begin{split}
    \|d\zeta\|_{L^{\frac{5}{2}}(\b^5)}&\le C\,\|f^\ast\xi\|^2_{L^\infty(\b^5)}\ \lf\|\fint_{M_G}\lf|g_p-\operatorname{id}_{G}\rg|\, d\vol_{M_G}(p)\rg\|_{L^{\frac{5}{2}}(\b^5)}\\[\sep]
    &\quad+C\,\|\al-f^\ast\xi\|_{L^\infty(\b^5)}\lf\|\fint_{M_G}\lf|d_xg_p\rg|\, d\vol_{M_G}(p)\rg\|_{L^{\frac{5}{2}}(\b^5)}\\[\sep]
    &\le C_f\, \|A\|_{L^2(\s^4)}\ \|dg\|_{L^2(\s^4)}\\[\sep]
    &\le\, C\,\|A\|_{L^2(\s^4)}\ \lf(\|F_A\|_{L^2(\s^4)}+\|A-\iota_{\s^4}^*f^*\bar A\|_{L^2(\s^4)}+\|A\|_{L^2(\s^4)}^2\rg)
\end{split}
\end{align}
and, by Poincar\'e--Sobolev inequality,
\begin{align}\label{new-zeta-2}
\begin{split}
    \|\zeta\|_{L^5(\b^5)}&\le\|\al-f^\ast\xi\|_\infty\ \lf|\fint_{M_G}\int_{\b^5}\lf|g_p-\operatorname{id}_{G}\rg|^5\, d\L^5\,d\vol_{M_G}(p)\rg|^{\frac{1}{5}}\\[\sep]
    &\le C_{f,G}\, \|A\|_{L^2(\s^4)}\ \  \lf|\fint_{M_G}\int_{\b^5}\lf|g_p-\operatorname{id}_{G}\rg|^{5}\, d\L^5\,d\vol_{M_G}(p)\rg|^{\frac{1}{5}}\\[\sep]
    &\le C_{f,G}\, \|A\|_{L^2(\s^4)}\ \ \lf|\fint_{M_G}\int_{\b^5}\lf|d_xg_p\rg|^{\frac{5}{2}}\, d\L^5\,d\vol_{M_G}(p)\rg|^{\frac{2}{5}}\\[\sep]
    &\le C_f\, \|A\|_{L^2(\s^4)}\ \|dg\|_{L^2(\s^4)}\\[\sep]
    &\le C\,\|A\|_{L^2(\s^4)}\ \lf(\|F_A\|_{L^2(\s^4)}+\|A-\iota_{\s^4}^*f^*\bar A\|_{L^2(\s^4)}+\|A\|_{L^2(\s^4)}^2\rg).
\end{split}
\end{align}
Hence,
\begin{align}\label{new-zeta-3}
    \|d\zeta+\zeta\wedge\zeta\|_{L^{\frac{5}{2}}(\b^5)}\le C\,\|A\|_{L^2(\s^4)}\ \lf(\|F_A\|_{L^2(\s^4)}+\|A-\iota_{\s^4}^*f^*\bar A\|_{L^2(\s^4)}+\|A\|_{L^2(\s^4)}^2\rg).
\end{align}
We now explicitly compute
\begin{align}\label{equation: expansion of the curvature of the harmonic extension}
\begin{split}
    F_{\tilde A}&=d_{\r^5}\tilde A+\tilde A\wedge\tilde A=d_{\r^5}(\tilde\omega+f^*\xi-\alpha+\zeta)+(\tilde\omega+f^*\xi-\alpha+\zeta)\wedge(\tilde\omega+f^*\xi-\alpha+\zeta)\\[\sep]
    &= F_{\tilde\omega}+F_{\zeta}-f^*\xi\wedge f^*\xi+\tilde\omega\wedge f^*\xi-\tilde\omega\wedge\alpha+\tilde\omega\wedge\zeta+{f}^*\xi\wedge\tilde\omega+{f}^*\xi\wedge{f}^*\xi-{f}^*\xi\wedge\alpha\\[\sep]
    &\quad+{f}^*\xi\wedge\zeta+-\alpha\wedge\tilde\omega-\alpha\wedge{f}^*\xi+\alpha\wedge\alpha+\zeta\wedge\tilde\omega+\zeta\wedge{f}^*\xi-\zeta\wedge\alpha\\[\sep]
    &=F_{\tilde\omega}+F_{\zeta}+(\tilde\omega\wedge{f}^*\xi+{f}^*\xi\wedge\tilde\omega)-(\tilde\omega\wedge\alpha+\alpha\wedge\tilde\omega)+(\tilde\omega\wedge\zeta+\zeta\wedge\tilde\omega)\\[\sep]
    &\quad-({f}^*\xi\wedge\alpha+\alpha\wedge{f}^*\xi)+({f}^*\xi\wedge\zeta+\zeta\wedge{f}^*\xi)-(\alpha\wedge\zeta+\zeta\wedge\alpha)+\alpha\wedge\alpha.
\end{split}
\end{align}
Hence, we obtain
\begin{align}\label{new-zeta-4}
\begin{split}
    \|F_{\ti{A}}\|_{L^{\frac{5}{2}}(\b^5)}&\le C\, \|d\ti{\om}\|_{L^{\frac{5}{2}}(\b^5)}+  \|\ti{\om}\|^2_{L^5(\b^5)}+\|F_\zeta\|_{L^{\frac{5}{2}}(\b^5)}+C\,\|\al\|_{L^\infty(\b^5)}\, \|f^\ast\xi\|_{L^\infty(\b^5)}\\[\sep]
    &\quad+\, C\, \|\zeta\|_{L^{\frac{5}{2}}(\b^5)}\, \lf(\|f^\ast\xi\|_{L^\infty(\b^5)}+\|\al\|_\infty\rg)+C\,\|\al\|_\infty^2\ \\[\sep]
    &\le  C\, \lf(\|F_A\|_{L^2(\s^4)}+\|A\|_{L^2(\s^4)}\,\|A-\iota_{\s^4}^*f^*\bar A\|_{L^2(\s^4)}+\|A\|_{L^2(\s^4)}^3\rg).
\end{split}
\end{align}

\medskip 
\noindent
Observe that
\begin{align}\label{tildAaver}
    \tilde A-f^*\bar A=\tilde\eta+f^*\xi-f^*\bar A=\tilde\eta+f^*(\xi-\bar A),
\end{align}
where $\ti{\eta}:=\ti{\om}-\al+\zeta$. We have
\begin{align}\label{new-zeta-5}
\begin{split}
    \|\tilde{\eta}\|_{L^5(\b^5)}&\le \|\tilde{\om}\|_{L^5(\b^5)}+\|\al\|_{L^5(\b^5)}+\|\zeta\|_{L^5(\b^5)}\\[\sep]
    &\le C\lf(\|F_A\|_{L^2(\s^4)}+\|A\|_{L^2(\s^4)}\|\eta\|_{L^2(\s^4)}+\|A\|_{L^2(\s^4)}^4\rg)+C\tilde C_f^2\|A\|^2_{L^2(\s^4)}\\[\sep]
    &\quad + C\tilde C_f\|A\|_{L^2(\s^4)}\lf(\|F_A\|_{L^2(\s^4)}+\|A-\iota_{\s^4}^*f^*\bar A\|_{L^2(\s^4)}+\hat C_f\|A\|_{L^2(\s^4)}^2\rg)\\[\sep]
    &\quad+ C\tilde C_f^2\|A\|^2_{L^2(\s^4)}\lf(\|F_A\|_{L^2(\s^4)}+\|A-\iota_{\s^4}^*f^*\bar A\|_{L^2(\s^4)}+\hat C_f\|A\|_{L^2(\s^4)}^2\rg).
\end{split}
\end{align}
We have also
\begin{align}
\begin{split}
\label{XXI}
  \|f^*(\xi-\bar A)\|_{L^5(\b^5)}&\le C\|d_{\r^5}f\|_{L^{\infty}(\b^5)}\|\xi-\bar A\|_{L^5(\b^5)}\\[\sep]
  &\le C\|d_{\r^5}f\|_{L^{\infty}(\b^5)}\lvert\xi-\bar A\rvert=C\|d_{\r^5}f\|_{L^{\infty}(\b^5)}\bigg(\sum_{i=1}^5\lvert\xi_i-\bar A_i\rvert^2\bigg)^{\frac{1}{2}},
\end{split}
\end{align}
where $C>0$ is a universal constant. Notice that, by Cauchy--Schwarz inequality, we have
\begin{align}
    \lvert\xi_i-\bar A_i\rvert^2=\frac{\lvert\langle A-\iota_{\s^4}^*f^*\bar A,d_{\s^4}f_i\rangle_{L^2(\s^4)}\rvert^2}{\|d_{\s^4}f_i\|_{L^2(\s^4)}^4}\le\frac{\|A-\iota_{\s^4}^*f^*\bar A\|_{L^2(\s^4)}^2}{\|d_{\s^4}f_i\|_{L^2(\s^4)}^2}, \qquad\forall i=1,...,5,
\end{align}
which by \eqref{XXI} implies
\begin{align}\label{XXII}
    \|f^*(\xi-\bar A)\|_{L^5(\b^5)}\le C\tilde C_f\|A-\iota_{\s^4}^*f^*\bar A\|_{L^2(\s^4)},
\end{align}
for some universal constant $C>0$. We now need an estimate on $L^2$-norm of $\tilde\eta$. Using \eqref{tildAaver}, \eqref{etaL2}, \eqref{tildeeta} and \eqref{XXII}, we have
\begin{align}\label{iiAtilde}
    \nonumber
    \|\tilde A-f^*\bar A\|_{L^5(\b^5)}&\le\|\tilde\eta\|_{L^5(\b^5)}+\|f^*(\xi-\bar A)\|_{L^5(\b^5)}\\[\sep]
    \nonumber
    &\le C\,\lf[\|F_A\|_{L^2(\s^4)}+\|A\|_{L^2(\s^4)}\|A-\iota_{\s^4}^*f^*\bar A\|_{L^2(\s^4)}+\|A\|_{L^2(\s^4)}^4\rg]+C\tilde C_f^2\|A\|^2_{L^2(\s^4)}\\[\sep]
    &\quad + C\tilde C_f\|A\|_{L^2(\s^4)}\lf(\|F_A\|_{L^2(\s^4)}+\|A-\iota_{\s^4}^*f^*\bar A\|_{L^2(\s^4)}+\hat C_f\|A\|_{L^2(\s^4)}^2\rg)\\[\sep]
    \nonumber
    &\quad+ C\tilde C_f^2\|A\|^2_{L^2(\s^4)}\lf(\|F_A\|_{L^2(\s^4)}+\|A-\iota_{\s^4}^*f^*\bar A\|_{L^2(\s^4)}+\hat C_f\|A\|_{L^2(\s^4)}^2\rg)\\[\sep]
    \nonumber
    &\quad + C\tilde C_f\|A-\iota_{\s^4}^*f^*\bar A\|_{L^2(\s^4)},
\end{align}
where $C$ only depends on $G$. Thus, we get \eqref{equation: estimate for the harmonic extension} for $B:=\tilde A$. We have
\begin{align}\label{diff-tilde-A}
\begin{split}
    \|d\ti{A}\|_{L^{\frac{5}{2}}(\b^5)}&\le \|d\ti{\om}\|_{L^{\frac{5}{2}}(\b^5)}+\|d\al\|_{L^{\frac{5}{2}}(\b^5)}+\|d\zeta\|_{L^{\frac{5}{2}}(\b^5)}\\[\sep]
    &\le C\,\lf(\|F_A\|_{L^2(\s^4)}+\|A\|_{L^2(\s^4)}\|A-\iota_{\s^4}^*f^*\bar A\|_{L^2(\s^4)}+\|A\|_{L^2(\s^4)}^4\rg)+C\tilde C_f^2\|A\|^2_{L^2(\s^4)}\\[\sep]
    &\quad + C\, \|A\|_{L^2(\s^4)}\ \lf(\|F_A\|_{L^2(\s^4)}+\|A-\iota_{\s^4}^*f^*\bar A\|_{L^2(\s^4)}+\|A\|_{L^2(\s^4)}^2\rg).
\end{split}
\end{align}
Hence, we obtain \eqref{d-tiA}.

\medskip
\noindent
We now prove \eqref{hat-A-L2}. Notice that
\begin{align}
\begin{split}
\label{FAhat}
    F_{\hat A}&=d_{\r^5}\hat A+\hat A\wedge\hat A=d_{\r^5}(\tilde A^{\tilde g^{-1}})+\tilde A^{\tilde g^{-1}}\wedge\tilde A^{\tilde g^{-1}}\\[\sep]
    &=\tilde g\,F_{\tilde A}\,\tilde g^{-1}\in L^{2}(\b^5) 
\end{split}
\end{align}
and
\begin{align}
\begin{split}
    \label{traceAhat}
    \iota_{\s^4}^*\hat A&=\iota_{\s^4}^*(\tilde gd_{\r^5}(\tilde g^{-1})+\tilde g\tilde A\tilde g^{-1})=\iota_{\s^4}^*(-d_{\r^5}\tilde g\tilde g^{-1}+\tilde g\tilde A\tilde g^{-1})\\[\sep]
    &=-d_{\s^4}gg^{-1}+gA^gg^{-1}=-d_{\s^4}gg^{-1}+g(g^{-1}d_{\s^4}g+g^{-1}Ag)g^{-1}\\[\sep]
     &=A\in L^2(\s^4).
\end{split}
\end{align}
We have
\begin{align}
\begin{split}
\label{29Ahat}
    \|\hat{A}-f^*\bar A\|_{L^2(\b^5)}&\le\|\hat{A}-\tilde{A}\|_{L^2(\b^5)}+ \|\tilde{A}-f^*\bar A\|_{L^2(\b^5)}\\[\sep]
    &\le\| \ti{g}\,\ti{A}\,\ti{g}^{-1}-    \ti{A}\|_{L^2(\b^5)}+\|\ti{g}\,d\ti{g}^{-1}\|_{L^2(\b^5)}+ \|\tilde{A}-f^*\bar A\|_{L^2(\b^5)}\\[\sep]
    &\le\|\tilde g-\operatorname{id}_G\|_{L^4(\b^5)}\ \|\ti{A}\|_{L^4(\b^5)}+\|d\ti{g}\|_{L^2(\b^5)}+\|\tilde{A}-f^*\bar A\|_{L^2(\b^5)}.
\end{split}
\end{align}
Combining \eqref{dtildeg}, \eqref{iiAtilde}, \eqref{29Ahat} and \eqref{iiAtilde}, we obtain \eqref{hat-A-L2}.

\medskip
\noindent
Coming now to the trace estimates 
\allowdisplaybreaks
\begin{align}\label{zetaL4}
\begin{split}
&\|\zeta\|_{L^4(\p\Om\cap\b^5)}\le 2\ \|\alpha-f^\ast\xi\|_{L^\infty(\b^5)}\lf\|\fint_{M_G}\lf|g_p-\operatorname{id}_{G}\rg|\, d\vol_{M_G}(p)\rg\|_{L^4(\partial\Omega\cap\b^5)}\\[\sep]
&\le 2\ \|\alpha-f^\ast\xi\|_{L^\infty(\b^5)}\lf\|\fint_{M_G}\lf|g_p-\operatorname{id}_{G}\rg|\, d\vol_{M_G}(p)\rg\|_{L^2(\p\Om\cap \b^5)}^{\frac{1}{2}}\lf\|\fint_{M_G}\lf|g_p-\operatorname{id}_{G}\rg|\, d\vol_{M_G}(p)\rg\|_{L^\infty(\p\Om\cap \b^5)}^{\frac{1}{2}}\\[\sep]
&\le C(G)\,\|\alpha-f^\ast\xi\|_{L^\infty(\b^5)}\lf\|\fint_{M_G}\lf|g_p-\operatorname{id}_{G}\rg|\, d\vol_{M_G}(p)\rg\|_{L^2(\partial\Omega\cap \b^5)}^{\frac{1}{2}}\\[\sep]
&\le C(G)\,\|\alpha-f^\ast\xi\|_{L^\infty(\b^5)}\bigg(\fint_{M_G}\lf\|g_p-\operatorname{id}_{G}\rg\|_{L^2(\partial\Omega\cap \b^5)}\, d\vol_{M_G}(p)\bigg)^{\frac{1}{2}}\\[\sep]
&\le  C(G)\,\|\alpha-f^\ast\xi\|_{L^\infty(\b^5)}\lf(\fint_{M_G}\lf\|g_p-\operatorname{id}_{G}\rg\|^2_{L^2(\partial\Omega\cap \b^5)}\, d\vol_{M_G}(p)\rg)^{\frac{1}{4}}\\[\sep]
&\le C(G,\Omega)\,\|\alpha-f^\ast\xi\|_{L^\infty(\b^5)}\lf\|g-\operatorname{id}_{G}\rg\|_{W^{1,2}(\s^4)}^{\frac{1}{2}}\\[\sep]
&\le C(G,\Omega)\lf(\tilde C_f\|A\|_{L^2(\s^4)}+\tilde C_f^2\|A\|^2_{L^2(\s^4)}\rg)\lf(\|F_A\|_{L^2(\s^4)}+\|\eta\|_{L^2(\s^4)}+\hat C_f\|A\|_{L^2(\s^4)}^2\rg)^{\frac{1}{2}}.
\end{split}
\end{align}
We have also using \eqref{dzeta}
\begin{align}
\label{dzetaL2}
    \nonumber
    \|d\zeta\|_{L^2(\partial\Omega\cap\b^5)}&\le C\,\|f^\ast\xi\|^2_{L^\infty(\b^5)}\lf\|\fint_{M_G}\lf|g_p-\operatorname{id}_{G}\rg|\, d\vol_{M_G}(p)\rg\|_{L^2(\partial\Omega\cap\b^5)}\\[\sep]
    \nonumber
    &\quad+C\,\|\al-f^\ast\xi\|_{L^\infty(\b^5)}\lf\|\fint_{M_G}\lf|dg_p\rg|\, d\vol_{M_G}(p)\rg\|_{L^2(\partial\Omega\cap\b^5)}\\[\sep]
    \nonumber
    &\le C\,\|f^\ast\xi\|^2_{L^\infty(\b^5)}\fint_{M_G}\lf\|g_p-\operatorname{id}_{G}\rg\|_{L^2(\partial\Omega\cap\b^5)}\, d\vol_{M_G}(p)\\[\sep]
    \nonumber
    &\quad+C\,\|\alpha-f^\ast\xi\|_{L^\infty(\b^5)}\fint_{M_G}\lf\|dg_p\rg\|_{L^2(\partial\Omega\cap\b^5)}\, d\vol_{M_G}(p)\\[\sep]
    \nonumber
    &\le C\,\|f^\ast\xi\|^2_{L^\infty(\b^5)}\bigg(\fint_{M_G}\lf\|g_p-\operatorname{id}_{G}\rg\|^2_{L^2(\partial\Omega\cap \b^5)}\, d\vol_{M_G}(p)\bigg)^{\frac{1}{2}}\\[\sep]
    &\quad+C\,\|\al-f^\ast\xi\|_{L^\infty(\b^5)}\bigg(\fint_{p\in M_G}\lf\|dg_p\rg\|^2_{L^2(\p\Om\cap \b^5)}\, d\vol_{M_G}(p)\bigg)^{\frac{1}{2}}.
\end{align}
Using Proposition \ref{lm-exten} together with \eqref{XVIII} and \eqref{C0al}, we finally get
\begin{align}\label{dzetaL2-bis}
\|d\zeta\|_{L^2(\p\Om\cap\b^5)}\le C_G(\Omega)\big(\tilde C_f\|A\|_{L^2(\s^4)}+\tilde C_f^2\|A\|^2_{L^2(\s^4)}\big)\big(\|F_A\|_{L^2(\s^4)}+\|\eta\|_{L^2(\s^4)}+\hat C_f\|A\|_{L^2(\s^4)}^2\big).\qquad\quad
\end{align}
Moreover, thanks to the continuous embedding $W^{\frac{3}{2},2}(\b^5)\hookrightarrow W^{1,2}(\partial\Omega\cap\b^5)$, we have
\begin{align}
\begin{split}
\label{tracetiomOmega}
    \|\tilde\omega\|_{L^4(\partial\Omega\cap\b^5)}+\|d_{\partial\Omega}\tilde\omega\|_{L^2(\partial\Omega\cap\b^5)}&\le C(\Omega)\,\|\tilde\omega\|_{W^{\frac{3}{2},2}(\b^5)}\le C(\Omega)\,\|\omega^g\|_{W^{1,2}(\s^4)}\\[\sep]
    &\le C(\Omega)\lf(\|F_A\|_{L^2(\s^4)}+\|A\|_{L^2(\s^4)}\|\eta\|_{L^2(\s^4)}+\|A\|_{L^2(\s^4)}^4\rg),
\end{split}
\end{align}
where $C(\Om)$ depends only on the $C^2$-norm of the $C^2$ components of $\partial(\Omega\cap\b^5)$, $G$, $C_f$ and $\|df\|_{L^{\infty}(\b^5)}$. Combining now \eqref{XVIII}, \eqref{C0al}, \eqref{zetaL4}, \eqref{dzetaL2-bis}, \eqref{tracetiomOmega} and \eqref{equation: expansion of the curvature of the harmonic extension}, we obtain \eqref{equation: estimate on the trace of the curvature}. 

\medskip
\noindent
Combining \eqref{zetaL4}, \eqref{tracetiomOmega}, \eqref{C0al} and the definition of $\xi$, we obtain
\begin{align}
    \|\hat{A}^{\ti{g}}\|_{L^4(\partial\Omega\cap\b^5)}=\|\ti{A}\|_{L^4(\partial\Omega\cap\b^5)}\le K_G\, \lf(\|F_{A}\|_{L^2(\s^4)}+\|A\|_{L^2(\s^4)}\lf\|A-\iota_{\s^4}^*f^*\bar A\rg\|_{L^2(\s^4)}+\|A\|_{L^2(\s^4)}\rg).
\end{align}
This concludes the proof of Proposition \ref{Proposition: harmonic extension on the good cubes}.
\end{proof}

\bigskip
\noindent
From now on, we will use the following notation:
\begin{align*}
    \lvert x\rvert_{\infty}&:=\sup_{i=1,...,5}\lvert x_i\rvert,\\
    \lvert x\rvert&:=\bigg(\sum_{i=1}^5x_i^2\bigg)^{\frac{1}{2}},
\end{align*}
for every $x=(x_1,...,x_5)\in\r^5$.  We have obviously
\begin{align*}
    \sqrt{5}^{-1}\, \lvert x\rvert\le \lvert x\rvert_{\infty}\le  \lvert x\rvert \qquad\forall\, x\in\r^5.
\end{align*}
Let $\varphi\in W_{loc}^{1,\infty}(\r^5,\r^5)$ be the bi-Lipschitz homeomorphism given by
\begin{align*}
    \varphi(x):=\begin{cases}
        \displaystyle{\frac{\lvert x\rvert}{\lvert x\rvert_{\infty}}}x & \mbox{ if } x\in\R^5\smallsetminus\{0\}\\
        0 & \mbox{ if } x=0.
    \end{cases}
\end{align*}
\begin{Co}\label{Corollary: extension in the interior of good cubes}
Let $G$ be a compact matrix Lie group. Let $Q\subset\r^5$ be any open cube with edge-length $k\eps$, where $k>0$ is a universal constant. There are constants $\eps_G\in(0,1)$ and $C_G>0$ depending only on $G$ such that for any $A\in\A_G(\p Q)$ satisfying
\begin{align}\label{equation: smallness condition on boundary of cube}
    \|F_{A}\|_{L^2(\partial Q)}+\eps^{-1}\|A\|_{L^2(\partial Q)}<\eps_G
\end{align}
the following facts hold.
\begin{enumerate}[(i)]
    \item There exist $g\in W^{1,2}(\partial Q,G)$ and a $\g$-valued 1-form $\tilde{A}\in L^5(Q)$ such that for every 4-dimensional face $F$ of $\p Q$ we have
    \begin{align}\label{equation: properties of the harmonic extension good cubes}
    \begin{cases}
        \|A^g\|_{W^{1,2}(F)}\hspace{-2mm}&\le C_G\,\lf(\|F_A\|_{L^2(\partial Q)}+\eps^{-1}\|A\|_{L^2(\partial Q)}\rg),\\[\sep]
        \iota_{F}^*\tilde A=A^{g}
    \end{cases}
    \end{align}
    and 
    \begin{align}\label{equation: estimate for the curvature on the good cubes}
        \|F_{\tilde A}\|_{L^{\frac{5}{2}}(Q)}&\le C_G\lf(\|F_{A}\|_{L^2(\partial Q)}+\eps^{-2}\,\|A\|_{L^2(\partial Q)}\,\|A-\iota_{\partial Q}^*\bar A\|_{L^2(\partial Q)}+\eps^{-3}\|A\|_{L^2(\partial Q)}^3\rg),\\
        \|dg\|_{L^{2}(\partial Q)}&\le C_G\lf(\eps\|F_{A}\|_{L^2(\partial Q)}+\|A-\iota_{\partial Q}^*\bar A\|_{L^2(\partial Q)}+\eps^{-1}\|A\|_{L^2(\partial Q)}^2\rg),
    \end{align}
    for every constant $\g$-valued $1$-form $\bar A$ on $\r^5$.
    \item There exists $\tilde g\in W^{1,2}(Q,G)$ such that the $\g$-valued 1-form given by $\hat A:=\tilde A^{\tilde g^{-1}}\in L^2(Q)$ satisfies the following properties.
        \begin{enumerate}
            \item $F_{\hat A}\in L^2(Q)$.
            \item $\iota_{\s^4}^*\hat A=A\in L^2(\partial Q)$.
            \item Let $Q'\subset\r^5$ be an open cube such that $Q'\cap Q$ is a rectangle of minimum edge-length $\alpha\eps>0$ where $\alpha>0$ is a universal constant. Then, we have 
            \begin{align}\label{equation: estimate on the trace of the curvature on cube}
            \begin{split}
                \|F_{\hat A}\|_{L^2(\partial Q'\cap Q)}^2&\le C_G\,\big(\|F_{A}\|_{L^2(\partial Q)}^2+\eps^{-4}\|A\|_{L^2(\partial Q)}^2\|A-\iota_{\partial Q}^*\bar A\|_{L^2(\partial Q)}^2+\eps^{-6}\|A\|_{L^2(\partial Q)}^6\big),
            \end{split}
            \end{align}
            \begin{align}
            \begin{split}\label{L4-norm-tiA}
                \|\ti{A}\|_{L^4(\partial Q'\cap Q)}\le C_G\, \lf(\|F_{A}\|_{L^2(\p Q)}+\ep^{-2}\,\|A\|_{L^2(\p Q)}\lf\|A-\iota_{\p Q}^*\bar A\rg\|_{L^2(\partial Q)}
             +\ep^{-1}\,\|A\|_{L^2(\partial Q)}\rg),
            \end{split}
            \end{align}
            for every constant $\g$-valued $1$-form $\bar A$ on $\r^5$.
        \end{enumerate}
\end{enumerate}
Moreover, we have the estimates
\be    
    \label{equation: estimate for the harmonic extension-cube}
    \begin{array}{l}
    \ds\|\ti{A}-\bar A\|_{L^5(Q)}\le C_G\,\lf(\|F_A\|_{L^2(\p Q)}+\eps^{-1}\,\|A-\iota_{\p Q}^*\bar A\|_{L^2(\p Q)}+\eps^{-2}\,\|A\|_{L^2(\p Q)}^2\rg),
    \end{array}
    \ee
    \be
    \label{d-tiA-c}
    \begin{array}{l}
    \ds\|d\ti{A}\|_{L^{\frac{5}{2}}(Q)}\le C_G\,\lf(\|F_A\|_{L^2(\p Q)}+\eps^{-1}\,\|A-\iota_{\p Q}^*\bar A\|_{L^2(\p Q)}+\eps^{-2}\,\|A\|_{L^2(\p Q)}^2\rg),
    \end{array}
    \ee
    \be    
    \label{hat-A-L2-c}
    \begin{array}{l}
    \ds\|\hat{A}-\bar A\|_{L^2(Q)}\le C_G\,\sqrt{\ep}\,\lf(\|F_A\|_{L^2(\p Q)}+\eps^{-1}\,\|A-\iota_{\p Q}^*\bar A\|_{L^2(\p Q)}+\eps^{-2}\,\|A\|_{L^2(\p Q)}^2\rg),
    \end{array}
    \ee
    for every constant $\g$-valued $1$-form $\bar A$ on $\r^5$.
\end{Co}
\begin{proof}[\textbf{\textup{Proof of Corollary \ref{Corollary: extension in the interior of good cubes}}}]
Corollary \ref{Corollary: extension in the interior of good cubes} is deduced by applying Proposition \ref{Proposition: harmonic extension on the good cubes} with $f=\varphi$ to the 1-form $\big((\eps\,\cdot\,+\,c_Q)\circ\varphi\big)^*A$, where $c_Q$ denotes the center of the cube $Q$, and then pulling the resulting data back on $Q$ by $\big((\eps\,\cdot\,+\,c_Q)\circ\varphi\big)^{-1}$.
\end{proof}

\subsubsection{Extension under \texorpdfstring{$L^2$}{Z}-smallness of the curvature only}
\begin{Prop}[Harmonic extension under smallness condition on $F_A$ only]\label{Proposition: harmonic extension on the bad cubes}
Let $G$ be a compact matrix Lie group. There are constants $\eps_G\in(0,1)$ and $C_G>0$ depending only on $G$ such that for any $A\in\A_G(\s^4)$ satisfying
\begin{align}\label{equation: smallness condition on S^4/2}
    \|F_{A}\|_{L^2(\s^4)}<\eps_G
\end{align}
the following facts hold.
\begin{enumerate}[(i)]
    \item There exist $g\in W^{1,2}(\s^4,G)$ and a $\g$-valued 1-form $\tilde{A}\in(W^{\frac{3}{2},2}\cap C^\infty)(\b^5)$ satisfying
        \be
        \label{X-40}\lf\{
            \begin{array}{l}
            \ds    \|\tilde A\|_{W^{\frac{3}{2},2}(\b^5)}\le C_G\|F_A\|_{L^2(\s^4)},\\[\sep]
              \ds  \|A^g\|_{W^{1,2}(\s^4)}\le C_G\|F_A\|_{L^2(\s^4)},\\[\sep]
        \ds        \iota_{\s^4}^*\tilde A=A^{g},
            \end{array}
            \rg.
        \ee
        and
        \begin{align}
            \|dg\|_{L^2(\s^4)}\le C_G\big(\|F_A\|_{L^2(\s^4)}+\|A\|_{L^2(\s^4)}\big).
        \end{align}  
    \item There exists $\tilde g\in W^{1,2}(\b^5,G)$ such that the $\g$-valued 1-form $\hat A:=\tilde A^{\tilde g^{-1}}\in L^2(\b^5)$ satisfies the following properties.
    \begin{enumerate}[(1)]
        \item $F_{\hat A}\in L^2(\b^5)$.
        \item $\iota_{\s^4}^*\hat A=A\in L^2(\s^4)$.
        \item   \be
                \label{co-ball}
                \begin{array}{l}
                \ds\int_{\b^5}|\hat{A}|^2\ d\L^5\le C(G)\ \lf(\|F_A\|^2_{L^2(\s^4)}+\|A\|^2_{L^2(\s^4)}\rg).
                \end{array}
                \ee
        \item Let $\Omega\subset\r^5$ be an open set such that $\Omega\cap \b^5$ has a $4$-dimensional compact Lipschitz boundary which can be included in a union of $N$ submanifolds of $\ov{\b^5}$ of class $C^2$. Then we have
            \begin{align}\label{equation: estimate on the trace of the curvature/2}
                \|F_{\hat A}\|_{L^2(\partial\Omega\cap\b^5)}\le K_G\|F_{A}\|_{L^2(\s^4)},
            \end{align}
            and
            \be
            \label{L4-bad-trace}
                \|\ti{A}\|_{L^4(\partial\Omega\cap\b^5)}\le K_G\|F_{A}\|_{L^2(\s^4)},
            \ee
            where $K_G=K_G(\Omega\cap\b^5)>0$ depends only on $G$ and on $\Omega\cap\b^5$ (that is the number $N$ of submanifolds containing $\p\Om$ as well as their $C^2$ norms).
    \end{enumerate}
\end{enumerate}
\end{Prop} 
\begin{proof}[\textbf{\textup{Proof of Proposition \ref{Proposition: harmonic extension on the bad cubes}}}]
To ease the reading, throughout this proof we will denote by ``$d_{\r^5}$" and ``$d_{\s^4}$" the standard differential of $k$-forms on $\r^5$ and on the round sphere $\s^4$ respectively. First, notice that since $A\in\A_G(\s^4)$, by definition there exists locally $\hat g\in W^{1,2}$ such that $A^{\hat g}\in W^{1,2}$. For $\eps_G$ in \eqref{equation: smallness condition on S^4/2}  small enough the existence  of such a $g$ is global and $A$ has a representative globally on $\s^4$ which is in $W^{1,2}(\wedge^1\s^4)$. We can apply Uhlenbeck's Coulomb gauge extraction theorem (see Proposition \ref{appendix: global Coulomb gauge extraction on a sphere}) to $A^{\hat g}$ and we get the existence of a gauge $h\in W^{1,4}(\s^4,G)$ such that, letting $g:=\hat g h\in W^{1,2}(\s^4,G)$, we have $d_{\s^4}^*A^g=0$ and
\begin{align}\label{equation: change of gauge on S^4/2}
    \|A^g\|_{W^{1,2}(\s^4)}\le C(G)\|F_A\|_{L^2(\s^4)},
\end{align}
where $C(G)>0$ depends only on $G$. Observe that by changing $g$ into $g\,g_0$ we still have a solution of \eqref{equation: change of gauge on S^4} with the constant $C(G)>0$ being unchanged. We have $d_{\s^4}g=g\,A^g-A\, g$, hence
\begin{align}\label{X-31}
\begin{split}
   \|d_{\s^4}g\|_{L^2(\s^4)}&\le C(G)\lf(\|A^g\|_{L^2(\s^4)}+\|A\|_{L^2(\s^4)}\rg)\\[\sep]
   &\le C(G)\lf(\|F_A\|_{L^2(\s^4)}+\|A\|_{L^2(\s^4)}\rg).
\end{split}
\end{align}
Poincar\'e inequality gives the existence of $C>0$ (independent on $A$) such that
\begin{align}\label{equation: estimate on gauge I/2}
    \|g-\bar g\|_{L^2(\s^4)}\le C\ \|d_{\s^4}g\|_{L^2(\s^4)},
\end{align}
where $\bar g$ is the average of $g$ on $\s^4$. Thus, we deduce the existence of $x_0\in\s^4$ such that
\begin{align}\label{equation: estimate on gauge II/2}
    \lvert g(x_0)-\bar g\rvert\le C\ \|d_{\s^4}g\|_{L^2(\s^4)}.
\end{align}
Replacing $g$ by $gg^{-1}(x_0)$ and combining \eqref{equation: estimate on gauge I/2} and \eqref{equation: estimate on gauge II/2} we obtain
\begin{align}
    \|g-\operatorname{id}_G\|_{L^2(\s^4)}\le C\, \|dg\|_{L^2(\s^4)}\le C(G)\lf(\|F_A\|_{L^2(\s^4)}+\|A\|_{L^2(\s^4)}\rg).
\end{align}
We denote $\tilde g:=g(x/|x|)$ the radial extension of $g$ in $\b^5$ (here ``$\lvert\,\cdot\,\rvert$'' stands for the standard Euclidean norm of $x$). A straightforward estimate gives
\begin{align}\label{equation: estimate on the radial extension of the gauge/2}
  \|\tilde g-\operatorname{id}_G\|_{L^2(\b^5)}\le\|g-\operatorname{id}_G\|_{L^2(\s^4)}\le  C(G) \lf(\|F_A\|_{L^2(\s^4)}+\|A\|_{L^2(\s^4)}\rg).
\end{align}
Moreover, by the coarea formula we have
\begin{align}\label{gau-bad}
\begin{split}
    \int_{\b^5}|d\ti{g}|^2\ d\L^5&=\int_0^1\rho^{-2}\,\int_{\p B_r(0)} \bigg\lvert dg\bigg(\frac{x}{\lvert x\rvert}\bigg)\bigg\rvert^2\ d\mathscr{H}^4(x)\ d\L^1(\rho)\le C\ \|dg\|^2_{L^2(\s^4)}\\[\sep]
    &\le C(G)\lf(\|F_A\|^2_{L^2(\s^4)}+\|A\|^2_{L^2(\s^4)}\rg).
\end{split}
\end{align}
\noindent
We now extend $A^g$ by $\tilde A$ that we choose to be the unique minimizer of 
\begin{align}
    \inf\bigg\{\int_{\b^5}\big(\lvert d_{\r^5}G\rvert^2+\lvert d_{\r^5}^*G\rvert^2\big)\, d\L^5 \mbox{ : } \iota_{\s^4}^*G=A^g\bigg\}.
\end{align}
Classical analysis for differential forms gives that $\tilde{A}\in(W^{\frac{3}{2},2}\cap C^\infty)(\b^5)$ solves
\be
\label{X-32}
\lf\{
    \begin{array}{l}
      \ds  d_{\r^5}^*\tilde A=0,\\[\sep]
       \ds d_{\r^5}^*d_{\r^5}\tilde A=0,\\[\sep]
      \ds  \iota_{\s^4}^*\tilde A=A^g.
    \end{array}
    \rg.
\ee
Moreover, the following estimate holds:
\begin{align}
    \|\tilde A\|_{W^{1,\frac{5}{2}}(\b^5)}\le C\ \|\tilde A\|_{W^{\frac{3}{2},2}(\b^5)}\le C\ \|A^g\|_{W^{1,2}(\s^4)}\le C(G)\ \|F_A\|_{L^2(\s^4)}.
\end{align}
Thus, (i) follows. 
Let $\hat A:=\tilde A^{\tilde g^{-1}}\in L^2(\b^5)$. Notice that
\be
\label{X-33}
    F_{\hat A}=d_{\r^5}\hat A+\hat A\wedge\hat A=d_{\r^5}(\tilde A^{\tilde g^{-1}})+\tilde A^{\tilde g^{-1}}\wedge\tilde A^{\tilde g^{-1}}=\tilde gF_{\tilde A}\tilde g^{-1}\in L^{2}(\b^5)
\ee
and
\begin{align}
    \iota_{\s^4}^*\hat A&=\iota_{\s^4}^*(\tilde gd_{\r^5}(\tilde g^{-1})+\tilde g\tilde A\tilde g^{-1})=\iota_{\s^4}^*(-d_{\r^5}\tilde g\tilde g^{-1}+\tilde g\tilde A\tilde g^{-1})\\
    &=-d_{\s^4}gg^{-1}+gA^gg^{-1}=-d_{\s^4}gg^{-1}+g(g^{-1}d_{\s^4}g+g^{-1}Ag)g^{-1}=A\in L^2(\s^4).
\end{align}
We have also
\begin{align}\label{connex-ball}
\begin{split}
    \ds\int_{\b^5}|\hat{A}|^2\ d\L^5&=\int_{\b^5}|\ti{g}\ti{A}\ti{g}^{-1}+\ti{g}\,d\ti{g}^{-1}|^2\ d\L^5\le 2\, \int_{\b^5}|\ti{A}|^2\ d\L^5+2\, \int_{\b^5}|d\ti{g}|^2\ d\L^5\\[\sep]
    &\le C(G) \lf(\|F_A\|^2_{L^2(\s^4)}+\|A\|^2_{L^2(\s^4)}\rg).
\end{split}
\end{align}
%Starting from \eqref{equation: estimate for curvature of the harmonic extension} for $\tilde A$, by exploiting estimate \eqref{equation: estimate on the radial extension of the gauge/2} we get the respective estimate for $\hat A$.  
Lastly, fix any open set $\Omega\subset\r^5$ such that $\Omega\cap\b^5$ has Lipschitz boundary which can be included in a union of $N$ submanifolds of $\ov{\b^5}$ of class $C^2$. Recalling that $\eps_G<1$ by assumption, we have
\begin{align}\label{X-34}
\begin{split}
    \|F_{\hat A}\|_{L^2(X)}&=\|F_{\tilde A}\|_{L^2(X)}
    \le C\left(\|\nabla \tilde A\|_{L^2\left(X)\right)}+\| \tilde A\|_{L^2\left( X\right)}^2\right)\\[\sep]
    %&\le C\ \|\tilde A\|_{W^{1,2}\left(\partial(\Omega\cap\b^5)\right)}+\|\nabla\tilde A\|_{W^{\frac{1}{2},2}\left(\partial(\Omega\cap\b^5)\right)}^2\right)\\
    &\le C(\Om)\|\tilde A\|_{W^{\frac{3}{2},2}\left(\b^5\right)}\le C(\Om,G)\ \|F_A\|_{L^2(\s^4)},
\end{split}
\end{align}
where $C(\Om,G)>0$ depends only on $G$ and on $\Omega\cap\b^5$. This concludes the proof of Proposition~\ref{Proposition: harmonic extension on the bad cubes}.
\end{proof}

\medskip\noindent
Completely analogously to the way we proved Corollary \ref{Corollary: extension in the interior of good cubes}, by exploiting Proposition \ref{Proposition: harmonic extension on the bad cubes} one derives from proposition~\ref{Proposition: harmonic extension on the bad cubes} the following statement.
\begin{Co}\label{Corollary: extension in the interior of bad cubes}
Let $G$ be a compact matrix Lie group. Let $Q\subset\r^5$ be any open cube with edge-length $k\eps$, where $k>0$ is a universal constant. There are constants $\eps_G\in(0,1)$ and $C_G>0$ depending only on $G$ such that for any $A\in {\mathfrak a}_G(\p Q)$ satisfying
\be
\label{equation: smallness condition on boundary of cube/2}
\|F_{A}\|_{L^2(\p Q)}<\eps_G
\ee
the following facts hold.
\begin{enumerate}[(i)]
    \item There exist a gauge $g\in W^{1,2}(\partial Q,G)$ and a 1-form $\tilde{A}\in W^{1,\frac{5}{2}}(Q)$ such that for every 4-dimensional face $F$ of $\p Q$ we have
        \begin{align}\label{equation: properties of the harmonic extension bad cubes}
            \begin{cases}
                \|d\tilde A\|_{L^{\frac{5}{2}}(Q)}+ \|\tilde A\|_{L^{5}(Q)}\le C_G\,\|F_A\|_{L^2(\partial Q)},\\
                \|A^g\|_{W^{1,2}(F)}\le C_G\|F_A\|_{L^2(\partial Q)},\\
                \iota_{F}^*\tilde A=A^{g}
            \end{cases}
        \end{align}
    and 
    \begin{align}\label{equation: estimate for the curvature on the bad cubes}
        \|F_{\tilde A}\|_{L^{\frac{5}{2}}(Q)}&\le C_G\,\|F_{A}\|_{L^2(\partial Q)},\\
        \|dg\|_{L^{2}(\partial Q)}&\le C_G\big(\eps\|F_{A}\|_{L^2(\partial Q)}+\|A\|_{L^2(\partial Q)}\big).
    \end{align}
    \item There exists $\tilde g\in W^{1,2}(Q,G)$ such that the $\g$-valued 1-form $\hat A:=\tilde A^{\tilde g^{-1}}\in L^2(Q)$ satisfies the following properties.
        \begin{enumerate}[(a)]
            \item $F_{\hat A}\in L^2(Q)$.
            \item $\iota_{\s^4}^*\hat A=A\in L^2(\partial Q)$.
            \item Let $Q'\subset\r^5$ be an open cube such that $Q'\cap Q$ is a rectangle of minimum edge-length $\alpha\eps>0$ where $\alpha>0$ is a universal constant. Then, we have 
                \begin{align}\label{equation: estimate on the trace of the curvature on cube/2}
                    \|F_{\hat A}\|_{L^2(\partial Q'\cap Q)}\le K_G\|F_{A}\|_{L^2(\partial Q)},
                \end{align}
                and
                \be
                \label{L4-tr}
                \|\ti{A}\|_{L^4(\partial Q'\cap Q)}\le K_G\|F_{A}\|_{L^2(\partial Q)},
                \ee
                where $K_G>0$ depends only on $G$.
        \end{enumerate}
\end{enumerate}
\end{Co}
\begin{Rm}
\label{rm-PR}
  \cite[Proposition 2.1]{petrache-riviere-na}, has a missing term on the right-hand-side of equation (2.2). This has been noticed by S. Sil. Proposition \ref{Proposition: harmonic extension on the good cubes} is a suitable replacement of \cite[Proposition 2.1]{petrache-riviere-na}. 
  % It is important to note that so far we have not been able to find a method that allows us to obtain the two estimates \eqref{equation: estimate for the harmonic extension} and \eqref{equation: estimate for curvature of the harmonic extension} simultaneously. Therefore, for further applications, it will be necessary to choose on a case-by-case basis which of the two is better to approximate.
\end{Rm}

%%%%%%%%%%%%%%%%%%%%%%%%%%%%%%%%%%%%%%%%%%%%%%%%%%%%%%%%%%%%%%%%%%%%%%%%%%%%%%%%%%%%%%%%%%%%%%%%%%%

\subsection{Construction of optimally regular gauges on the boundary of cubes}
We start by recalling the following optimal extension theorem, whose proof can be found in \cite[Theorem 2]{petrache-van-schaftingen}.
\begin{Prop}\label{p-ext}
    Let $n\in\n$ be such that $n\ge 1$ and let $N\hookrightarrow\r^{k}$ be a closed embedded submanifold of $\r^{k}$. Then, for every $u\in W^{\frac{n}{n+1},n+1}(\s^{n},N)$ there exists an extension $\tilde u\in W^{1,(n+1,\infty)}(\b^{n+1},N)$ such that $\tilde u|_{\s^n}=u$ in the sense of traces.
\end{Prop}
\medskip
\noindent
Next, we leverage on the Proposition \ref{p-ext} to prove the following gluing lemma.
\begin{Lm}\label{l-gluing}
    Let $G$ be a compact matrix Lie group and let $n\in\n$ be such that $n\ge 2$. Let
    \begin{align}
    \begin{split}
        \s_+^n&:=\{x=(x_1,...,x_{n+1})\in\s^n \mbox{ : } x_{n+1}>0\}.
    \end{split}
    \end{align}
    Then, for every $g\in W^{1,n}(\s_+^n,G)$ there exists $\tilde g\in W^{1,(n+1,\infty)}(\b^{n+1},G)$ such that $\tilde g|_{\s_+^n}=g$.
\end{Lm}
\begin{proof}[\textbf{\textup{Proof of Lemma \ref{l-gluing}}}]
        Let 
        \begin{align}
            \s_-^n&:=\{x=(x_1,...,x_{n+1})\in\s^n \mbox{ : } x_{n+1}<0\}.
        \end{align}
        and define $h\in W^{1,n}(\s_-^n,G)$ by
        \begin{align}
            h(x_1,...,x_{n+1}):=g(x_1,...,-x_{n+1}) \qquad\forall\,x=(x_1,...,x_{n+1})\in\s_-^n.
        \end{align}
        Let $\hat g\in W^{1,n}(\s^n,G)$ be given by 
        \begin{align}
            \hat g:=\begin{cases}g & \mbox{ on } \s_+^n\\h & \mbox{ on } \s_-^n\end{cases}
        \end{align}
        Since $W^{1,n}(\s^n,G)\hookrightarrow W^{\frac{n}{n+1},n+1}(\s^{n},G)$, by Proposition \ref{p-ext} there exists $\tilde g\in W^{1,(n+1,\infty)}(\b^{n+1},N)$ such that $\tilde g|_{\s^n}=\hat g$ in the sense of traces. Since $\hat g=g$ on $\s_+^n$, this concludes the proof of Lemma \ref{l-gluing}.
\end{proof}
\begin{Co}\label{c-gluing}
    Let $G$ be a compact matrix Lie group. There exist constants $\eps_G\in(0,1)$ and $C_G>0$ depending only in $G$ such that the following holds. Let $Q\subset\r^5$ be an open cube in $\r^5$ and let $\mathscr{F}$ be any non-empty proper subset of $4$-dimensional faces of $\partial Q$. Let
    \begin{align}
        \Omega:=\bigcup_{F\in\mathscr{F}}F.
    \end{align}
    Then, for every $g\in W^{1,4}(\Omega,G)$ there exists an extension $\tilde g\in W^{1,5}(Q,G)$ such that $\tilde g|_{\Omega}=g$ in the sense of traces.
\end{Co}
\begin{proof}[\textbf{\textup{Proof of Corollary \ref{c-gluing}}}]
        Let $\Phi:\overline{Q}\to\overline{\b^5}$ be a bi-Lipschitz homeomorphism such that 
        \begin{align}
            \Omega=\s_+^4, \qquad \partial\Omega=\Gamma, \qquad \partial Q\smallsetminus\Omega=\s_-^4.
        \end{align}
        By apply Lemma \ref{l-gluing} to $(\Phi^{-1})^*g\in W^{1,4}(\s_+^4,G)$ and pullback the resulting extension on $\b^5$ by $\Phi$ we get the desired $\tilde g$. This concludes the proof of Corollary \ref{c-gluing}.
\end{proof}

%%%%%%%%%%%%%%%%%%%%%%%%%%%%%%%%%%%%%%%%%%%%%%%%%%%%%%%%%%%%%%%%%%%%%%%%%%%%%%%%%%%%%%%%%%%%%%%%%%%

\section{The notion of admissible covers and good and bad cubes}\label{section: admissible covers, good and bad cubes}
\subsection{The notion of admissible cubic \texorpdfstring{$\ep_i$}{Z}-covers}
From now on, for every $c\in\r^5$ and $\eps>0$ we let
\begin{align*}
	Q_{\eps}(c):=\bigg(\hspace{-1.5mm}-\frac{\eps}{2},\frac{\eps}{2}\bigg)^5+c
\end{align*}
be the open cube with center $c$, edge-length $\eps>0$ and faces parallel to the coordinate planes. We shall be using the following definition.
\begin{Dfi}
\label{df-admissible Lipschitz domain}
Let $A\in {\mathfrak a}_G(\b^5)$, $c\in\r^5$ and $\rho>0$. We say that $A\in {\mathcal A}_G(Q_\rho(c))$ if there exists an $L^2$ 1-form that we denote 
$\iota_{\p Q_\rho(c)}^\ast A$ such that
\begin{enumerate}[(i)]
\item
\begin{align*}
	\lim_{\ep\rightarrow 0}\frac{1}{\ep}\int_{\rho-\ep}^{\rho+\ep}\int_{\p Q_\rho(c)}\lf|\iota_{ D_r }^\ast\iota_{\p Q_r(c)}^\ast A-\iota_{\p Q_\rho(c)}^\ast A\rg|^2\ d\H^4\, d\L^1(r)=0,
\end{align*}
where $D_r(x):= \frac{\rho}{r}\,(x-c)+c$.
\item
\begin{align*}
	\iota_{\p Q_\rho(c)}^\ast A\in {\mathfrak a}_G(\p Q_\rho(c)).
\end{align*}
\end{enumerate}
\end{Dfi}

\noindent
As a direct consequence of Fubini and Lebesgue theorems, we have the following proposition.
\begin{Prop}
\label{pr-admissible-edge} Let $A\in {\mathfrak a}_G(\b^5)$ and $c\in \b^5$, then for $\L^1$-a.e. $\rho>0$  such that $Q_\rho(c)\subset {\b}^5$ there holds $A\in{\mathcal A}_G(\p Q_\rho(c))$. A positive number $\rho$ such that $A\in{\mathcal A}_G(\p Q_\rho(c))$ is called an \textit{admissible edge-length} for $A$.
\end{Prop}

\allowdisplaybreaks
\noindent
For the purposes of the present subsection, given any $\eps\in\big(0,\frac{1}{4}\big)$ we let 
\begin{align*}
    \mathscr{C}_{\eps}:=\eps\z^5\cap\overline{Q_{1-2\eps}(0)}.
\end{align*}
Note that
\begin{align}\label{equation: finite intersection property}
    \#\{c'\in\mathscr{C}_{\eps} \mbox{ : } Q_{2\eps}(c)\cap Q_{2\eps}(c')\neq\emptyset\}\le N \qquad\forall\, c\in\mathscr{C}_{\eps},
\end{align}
where $N\in\n$ is independent on $\eps$. 
\begin{Lm}[Choice of an admissible cubic cover]\label{Lemma: choice of an admissible cubic cover}
Let $A\in\A_G(Q_1^5(0))$. There exists a universal constant $K>0$ such that for a sequence  $\{\eps_i\}_{i\in\n}\in\big(0,\frac{1}{4}\big)$ satisfying $\eps_i\to 0$ for $i\to +\infty$, for $i$ large enough we can find a family of admissible edge-lengths $\{\rho_{i,c}\}_{c\in\mathscr{C}_{\eps_i}}\subset\big(\frac{3}{2}\eps_i,2\eps_i\big)$ for which the following facts hold.
\begin{enumerate}[(i)]
    \item $A\in{\mathcal A}_G(\partial Q_{\rho_{i,c}}(c))$ for every $c\in\mathscr{C}_{\eps_i}$. 
    \item For every $c\in\mathscr{C}_{\eps_i}$ we have
    \begin{align}\label{X-35a}
    	 \int_{\partial Q_{\rho_{i,c}}(c)}\lvert F_A\rvert^2\, d\H^4\le\frac{K}{2\eps_i}\int_{Q_{2\eps_i}(c)}\lvert F_A\rvert^2\, d\L^5
    \end{align}
    and
    \begin{align} \label{X-355a}
    	\int_{\partial Q_{\rho_{i,c}}(c)}\lvert A\rvert^2\, d\H^4\le\frac{K}{2\eps_i}\int_{Q_{2\eps_i}(c)}\lvert A\rvert^2\, d\L^5.
    \end{align}
    \item It holds that
        \begin{align}
            \label{equation: convergence of A minus its averages on an admissible cover}
            &\lim_{i\to+\infty}\eps_i\sum_{c\in\mathscr{C}_{\eps_i}}\int_{\partial Q_{\rho_{i,c}}(c)}\big\lvert A-(A)_{Q_{2\eps_i}(c)}\big\rvert^2\, d\H^4=0,
        \end{align}
          \begin{align}
            \label{app-grid}
            &\lim_{i\to+\infty}\sum_{c\in\mathscr{C}_{\eps_i}}\int_{ Q_{\rho_{i,c}}(c)}\big\lvert A-(A)_{Q_{2\eps_i}(c)}\big\rvert^2\, d\L^5=0
        \end{align}
        and
        \begin{align}
            \label{equation: convergence of F_A minus its averages on an admissible cover}
            &\lim_{i\to +\infty}\eps_i\sum_{c\in\mathscr{C}_{\eps_i}}\int_{\partial Q_{\rho_{i,c}}(c)}\big\lvert F_A-(F_A)_{Q_{2\eps_i}(c)}\big\rvert^2\, d\H^4=0,
        \end{align}
        where we have used the following notation:
        \[
        (A)_{Q_{2\eps_i}(c)}:=\fint_{Q_{2\eps_i}(c)} A \, d\L^5\quad\quad\mbox{ and }\quad\quad (F_A)_{Q_{2\eps_i}(c)}:=\fint_{Q_{2\eps_i}(c)} F_A \, d\L^5.
        \]
\end{enumerate}
\end{Lm}
\begin{proof}[\textbf{\textup{Proof of Lemma \ref{Lemma: choice of an admissible cubic cover}}}]
Fix any $\eps\in\big(0,\frac{1}{4}\big)$. By assumption, for every $c\in\mathscr{C}_{\eps}$ there exists a full $\L^1$-measure set $R_{A,c}\subset\big(\frac{3}{2}\eps,2\eps\big)$ such that $\iota_{\partial Q_{\rho}(c)}^*A\in\A_G(\partial Q_{\rho}(c))$ for every $\rho\in R_{A,c}$. Given a constant $K>0$ (to be fixed), for any $c\in\mathscr{C}_{\eps}$ we define the set
\begin{align}
   E_{\eps,K,c}:=\bigg\{\rho\in\bigg(\frac{3}{2}\eps,2\eps\bigg) \mbox{ s.t. } A\notin {\mathfrak a}_G(\partial Q_{\rho}(c))\quad\mbox{or }\quad\int_{\partial Q_{\rho}(c)}\lvert F_A\rvert^2\,d\H^4>\frac{K}{2\eps}\int_{Q_{2\eps}(c)}\lvert F_A\rvert^2\, d\L^5\bigg\}.
\end{align}
By integration on $E_{\eps,K,c}$ and Fubini together with mean value theorem we get
\begin{align}
    \frac{K}{2\eps}\bigg(\int_{Q_{2\eps}(c)}\lvert F_A\rvert^2\, d\L^5\bigg)\L^1(E_{\eps,K,c})&<\int_{E_{\eps,K,c}}\int_{\partial Q_{\rho}(c)}\lvert F_A\rvert^2\,d\H^4\,d\L^1(\rho)\\
    &\le\int_0^{2\eps}\int_{\partial Q_{\rho}(c)}\lvert F_A\rvert^2\,d\H^4\,d\L^1(\rho)\\
    &=\int_{Q_{2\eps}(c)}\lvert F_A\rvert^2\,d\L^5.
\end{align}
This implies that
\begin{align}\label{X-35-aa}
	\L^1(E_{\eps,K,c})<\frac{2\eps}{K}, \qquad\forall\,c\in\mathscr{C}_{\eps}.
\end{align}
Fix $i\in\n$ and let $ A_i\in C_c^{\infty}(\opwedge^1Q_1^5(0)\otimes\mathfrak{g})$ be such that
\begin{align}
    &\int_{Q_1^5(0)}\lvert A_i-A\rvert^2\, d\L^5\le 2^{-i}.
\end{align}
Let 
\begin{align}
    G_{\eps,K,i}&:=\bigg\{\rho\in\bigg(\frac{3}{2}\eps,2\eps\bigg) \mbox{ s.t. } \eps\sum_{c\in\mathscr{C}_{\eps}}\int_{\partial Q_{\rho}(c)}\lvert A_i-A\rvert^2\, d\H^4>K\int_{Q_1^5(0)}\lvert A_i-A\rvert^2\, d\L^5\bigg\},\\
    H_{\eps,K,i}&:=\bigg\{\rho\in\bigg(\frac{3}{2}\eps,2\eps\bigg) \mbox{ s.t. } f_i(\rho)>K\fint_{\frac{3}{2}\eps}^{2\eps}f_i(\rho)\, d\L^1(\rho)\bigg\},
\end{align}
with
\begin{align}
    f_i(\rho):=\sum_{c\in\mathscr{C}_{\eps}}\int_{\partial Q_{\rho}(c)}\big\lvert  A_i-( A_i)_{Q_{2\eps}(c)}\big\rvert^2\, d\H^4.
\end{align}
Notice that, again by Fubini theorem together with mean value theorem  and integration as above, we get
\begin{align}\label{X-37}
\begin{split}
	\L^1(G_{\eps,K,i})&\le\frac{C\eps}{K}\\[\sep]
	\L^1(H_{\eps,K,i})&\le\frac{\eps}{K}
\end{split}
\end{align}
for every $i\in\n$, where $C>0$ is a universal cover.  By Fubini theorem, the mean value theorem and Poincar\'e inequality, we have
\begin{align}
    \int_{\frac{3}{2}\eps}^{2\eps}\sum_{c\in\mathscr{C}_{\eps}}\int_{\partial Q_{\rho}(c)}\big\lvert  A_i-( A_i)_{Q_{2\eps}(c)}\big\rvert^2\, d\H^4\,d\L^1(\rho)&\le\sum_{c\in\mathscr{C}_{\eps}}\int_{Q_{2\eps}(c)}\big\lvert  A_i-( A_i)_{Q_{2\eps}(c)}\big\rvert^2\, d\L^5\\
    &\le \sum_{c\in\mathscr{C}_{\eps}}C_P\big(Q_{2\eps}(c)\big)\int_{Q_{\frac{3\eps}{4}}(c)}\lvert \nabla A_i\rvert^2\, d\L^5\\
    &\le 4C_P(Q_1^5(0))\eps^2\sum_{c\in\mathscr{C}_{\eps}}\int_{Q_{2\eps}(c)}\lvert\nabla A_i\rvert^2\, d\L^5\\
    &\le C\eps^2\int_{Q_1^5(0)}\lvert\nabla A_i\rvert^2\, d\L^5,
\end{align}
where we denote $C_P(\Omega)>0$ for the open and bounded domain $\Omega\subset\r^5$ and $C>0$ is universal. Hence for any $\rho\in (G_{\eps,K,i}\cup H_{\eps,K,i})^c$ (where the superscript ``$c$'' denotes the complement of the set in $(\eps,2\eps)$) one deduces
\begin{align}
    \sum_{c\in\mathscr{C}_{\eps}}\int_{\partial Q_{\rho}(c)}\big\lvert A_i-( A_i)_{Q_{2\eps}(c)}\big\rvert^2\, d\H^4&\le K\fint_{\frac{3}{2}\eps}^{2\eps}\sum_{c\in\mathscr{C}_{\eps}}\int_{\partial Q_{\rho}(c)}\lvert A_i-( A_i)_{Q_{2\eps}(c)}\rvert^2\, d\H^4\,d\L^1(\rho)\\
    &\le KC\eps\int_{Q_1^5(0)}\lvert \nabla A_i\rvert^2\, d\L^5.
\end{align}
Moreover, for every $\rho\in(\eps,2\eps)$ we have 
\begin{align}\label{X-37-a}
	\eps\sum_{c\in\mathscr{C}_{\eps}}\int_{\partial Q_{\rho}(c)}\big\lvert ( A_i)_{Q_{2\eps}(c)}-(A)_{Q_{2\eps}(c)}\big\rvert^2\,d\H^4&\le \eps\sum_{c\in\mathscr{C}_{\eps}} |\p Q_\rho(c)|\, \lf(\fint_{Q_{2\eps}(c)} |A-A_i|\ d\L^5\rg)^2\\[\sep]
	&\le C\, \sum_{c\in\mathscr{C}_{\eps}}\int_{Q_{2\eps}(c)}|A-A_i|^2\ d\L^5\le C\ 2^{-i}.
\end{align}
Hence, by triangle inequality, we get that for every $\rho\in (G_{\eps,K,i}\cup H_{\eps,K,i})^c$ we have
\begin{align}
    \eps\sum_{c\in\mathscr{C}_{\eps}}\int_{\partial Q_{\rho}(c)}\big\lvert A-(A)_{Q_{2\eps}(c)}\big\rvert^2\, d\H^4\le\big(K+C\big)2^{-i}+ KC\eps\int_{Q_1^5(0)}\lvert \nabla\tilde A_i\rvert^2\, d\L^5.
\end{align}
Analogously, we consider $F_i\in C_c^{\infty}(\opwedge^2Q_1^5(0)\otimes\mathfrak{g})$ such that
\begin{align}
    &\int_{Q_1^5(0)}\lvert F_i-F_A\rvert^2\, dx^5\le 2^{-i}
\end{align}
and we get that for every $\eps\in\big(0,\frac{1}{4}\big)$ and for every $i\in\n$ there are two $\L^1$-measurable sets $G_{\eps,K,i}',H_{\eps,K,i}'\subset\big(\frac{3}{2}\eps,2\eps\big)$ satisfying
\begin{align}\label{X-38}
\begin{split}
	\L^1(G_{\eps,K,i}')&\le\frac{C\eps}{K}\\[\sep]
	\L^1(H_{\eps,K,i}')&\ds\le\frac{\eps}{K}
\end{split}
\end{align}
such that for every $\rho\in (G_{\eps,K,i}'\cup H_{\eps,K,i}')^c$ we have 
\begin{align}
    \eps\sum_{c\in\mathscr{C}_{\eps}}\int_{\partial Q_{\rho}(c)}\big\lvert F_A-(F_A)_{Q_{2\eps}(c)}\big\rvert^2\, d\H^4\le\big(K+ C\big)2^{-i}+ KC\eps\int_{Q_1^5(0)}\lvert \nabla F_i\rvert^2\, d\L^5.
\end{align}
Now, first we let $\tilde K>0$ be big enough (independent on $\eps$ and $i$) so that 
\begin{align}
    \L^1(G_{\eps,\tilde K,i}\cup H_{\eps,\tilde K,i}\cup G_{\eps,\tilde K,i}'\cup H_{\eps,\tilde K,i}')\le\frac{\eps}{2}.
\end{align}
Then, we notice that for every $i\in\n$ there exists $\eps_i<2^{-(i+1)}$ small enough so that
\begin{align}
    \tilde KC\eps_i\bigg(\int_{Q_1^5(0)}\lvert \nabla A_i\rvert^2\, d\L^5+\int_{Q_1^5(0)}\lvert \nabla F_i\rvert^2\, d\L^5\bigg)\le 2^{-(i+1)}.
\end{align}
Then, for every $i\in\n$ and for every $c\in\mathscr{C}_{\eps_i}$ we pick $\rho_{i,c}>0$ in the non-empty set
\begin{align}
    (G_{\eps_i,\tilde K,i}\cup H_{\eps_i,\tilde K,i}\cup G_{\eps_i,\tilde K,i}'\cup H_{\eps_i,\tilde K,i}')^c\cap R_{A,c}\cap E_{\eps_i,\tilde K,c}.
\end{align}
Notice that, for such admissible edge-lengths, we have
\begin{align}
    \eps_i\sum_{c\in\mathscr{C}_{\eps_i}}\int_{\partial Q_{\rho_{i,c}}(c)}\big\lvert A-(A)_{Q_{2\eps_i}(c)}\big\rvert^2\, d\H^4\le\big(K+C\big)\,2^{-i}+2^{-(i+1)}
\end{align}
and 
\begin{align}
    \eps_i\sum_{c\in\mathscr{C}_{\eps_i}}\int_{\partial Q_{\rho_{i,c}}(c)}\big\lvert F_A-(F_A)_{Q_{2\eps_i}(c)}\big\rvert^2\, d\H^4\le\big(K+C\big)\,2^{-i}+2^{-(i+1)}.
\end{align}
This concludes the proof of Lemma \ref{Lemma: choice of an admissible cubic cover}.
\end{proof}
\begin{Dfi}\label{Definition: admissible cubic eps-cover}
Let $A\in\A_G(Q_1^5(0))$. Under the same notation that we have used in the previous Lemma \ref{Lemma: choice of an admissible cubic cover}, for every $\eps_i\in\big(0,\frac{1}{4}\big)$ we say that $\eps_i$ is an \textit{admissible scale} for $A$ and that the collection of cubes $\{Q_{\rho_{i,c}}(c)\}_{c\in\mathscr{C}_{\eps_i}}$ is a \textit{admissible cubic $\eps_i$-cover} relative to $A$.
\end{Dfi}
\begin{Rm}\label{Remark: estimate on the sum of the norm of A resticted to the boundaries}
Let $A\in\A_G(Q_1^5(0))$. Let $\eps\in\big(0,\frac{1}{4}\big)$ and let $\mathcal{Q}_{\eps}$ be an admissible cubic $\eps$-cover relative to $A$. For every $Q\in\mathcal{Q}_{\eps}$, denote by $\bar A_Q$ the constant $2$-form given by $(A)_{Q_{2\eps}(c_Q)}$, where $c_Q$ denotes the center of $Q$. Notice that, by Jensen inequality, we have
\begin{align}
    \lvert\bar A_Q\rvert^2\le\fint_{Q_{2\eps}(c_Q)}\lvert A\rvert^2\, d\L^5, \qquad\forall\, Q\in\mathcal{Q}_{\eps}.
\end{align}
We have
\begin{align}
    \sum_{Q\in\mathcal{Q}_{\eps}}\int_{\partial Q}\lvert A\rvert^2\, d\H^4&\le 2\sum_{Q\in\mathcal{Q}_{\eps}}\bigg(\int_{\partial Q}\lvert A-\bar A_Q\rvert^2\, d\H^4+\int_{\partial Q}\lvert\bar A_Q\rvert^2\, d\H^4\bigg)\\
    &\le C\sum_{Q\in\mathcal{Q}_{\eps}}\bigg(\int_{\partial Q}\lvert A-\bar A_Q\rvert^2\, d\H^4+\eps^4\lvert\bar A_Q\rvert^2\bigg)\\
    &\le C\sum_{Q\in\mathcal{Q}_{\eps}}\bigg(\int_{\partial Q}\lvert A-\bar A_Q\rvert^2\, d\H^4+\eps^4\fint_{Q_{2\eps}(c_Q)}\lvert A\rvert^2\, d\L^5\bigg)\\
    &\le C\bigg(\sum_{Q\in\mathcal{Q}_{\eps}}\int_{\partial Q}\lvert A-\bar A_Q\rvert^2\, d\H^4+\frac{1}{\eps}\int_{Q_1^5(0)}\lvert A\rvert^2\, d\L^5\bigg),
\end{align}
where $C>0$ is independent on $\eps$. By \eqref{equation: convergence of A minus its averages on an admissible cover}, for $\eps$ small enough we have
\begin{align}
    \eps\sum_{Q\in\mathcal{Q}_{\eps}}\int_{\partial Q}\lvert A-\bar A_Q\rvert^2\, d\H^4\le\int_{Q_1^5(0)}\lvert A\rvert^2\, d\L^5.
\end{align}
This implies that for such small values of the parameter $\eps$ it holds that
\begin{align}\label{equation: estimate on the sum of the boundaries of the connection}
    \eps\sum_{Q\in\mathcal{Q}_{\eps}}\int_{\partial Q}\lvert A\rvert^2\, d\H^4\le C\int_{Q_1^5(0)}\lvert A\rvert^2\, d\L^5,
\end{align}
where $C>0$ is independent on $\eps$. Exactly by the same procedure, we get that 
\begin{align}\label{equation: estimate on the sum of the boundaries of the curvature}
    \eps\sum_{Q\in\mathcal{Q}_{\eps}}\int_{\partial Q}\lvert F_A\rvert^2\, d\H^4\le C\int_{Q_1^5(0)}\lvert F_A\rvert^2\, d\L^5
\end{align}
5
for $\eps>0$ sufficiently small. 
\end{Rm}
\subsection{The notion of good and bad cubes}
\begin{Dfi}[Good and bad cubes]\label{Definition: good cubes eps-cover}
    Let $A\in\A_G(Q_1^5(0))$. Let $\eps\in\big(0,\frac{1}{4}\big)$ be an admissible scale for $A$ and let $\mathcal{Q}_{\eps}$ be an admissible cubic $\eps$-cover relative to $A$.
    Given any $\Lambda>0$, we say that $Q\in\mathcal{Q}_{\eps}$ is a $\Lambda$-\textit{good} cube if the following conditions hold:
    \begin{enumerate}[(1)]
        \item $\displaystyle{\frac{1}{\eps^3}\int_{\partial Q}\lvert F_A\rvert^2\, d\H^4\le\eps^{\frac{1}{2}}\int_{Q_1^5(0)}\lvert F_A\rvert^2\, d\L^5}$,
        \item $\displaystyle{\frac{1}{\eps^3}\int_{\partial Q}\lvert A\rvert^2\, d\H^4\le\eps^{\frac{1}{2}}\int_{Q_1^5(0)}\lvert F_A\rvert^2\, d\L^5}$,
        \item $\displaystyle{\int_{\partial Q}\lvert A-\bar A_Q\rvert^2\, d\H^4\le\frac{1}{\eps}\int_{Q_{2\eps}(c_Q)}\lvert A\rvert^2\, d\L^5}$,
        \item $\displaystyle{\frac{1}{\eps^4}\int_{\partial Q}\lvert A-\bar A_Q\rvert^2\, d\H^4\le\Lambda^{-1}\int_{Q_1^5(0)}\lvert F_A\rvert^2\, d\L^5}$,
        \item $\displaystyle{\fint_{Q_{2\eps}(c_Q)}\lvert A\rvert^2\, d\L^5\le\Lambda}$.
    \end{enumerate}
    Otherwise, we say that $Q$ is a $\Lambda$-\textit{bad} cube. We denote by $\mathcal{Q}_{\eps,\Lambda}^g$ the set of all the $\Lambda$-good cubes and $\mathcal{Q}_{\eps,\Lambda}^b:=\mathcal{Q}_{\eps}\smallsetminus\mathcal{Q}_{\eps,\Lambda}^g$.
\end{Dfi}
The following technical lemma is important for our argument. Given a sequence of admissible scales  $\{\eps_i\}_{i\in\n}$, the lemma states that as $i\in\n$ increases, asymptotically, the $L^2$-norm of $A$ on the $\Lambda$-bad cubes is controlled by the $L^2$ norm of $A$ on  the union of cubes where property (5) in Definition~\ref{Definition: good cubes eps-cover} fails. 
\begin{Lm}\label{Lemma: norm of A on bad cubes}
    Let $A\in\A_G(Q_1^5(0))$ and let $\{\eps_i\}_{i\in\n}\subset\big(0,\frac{1}{4}\big)$ be a sequence of admissible scales  for $A$. Let $\Lambda>0$ and let $\mathcal{Q}_{\eps_i}$ be an admissible cubic $\eps_i$-cover relative to $A$ and $\Lambda$, for every $i\in\n$. Then, for some universal constant $C>0$ (depending on the intersection property of the cubic cover) there holds
    \begin{align}
        \lim_{i\to +\infty}\sum_{Q\in\mathcal{Q}_{\eps_i,\Lambda}^b}\int_{Q_{2\eps_i}(c_Q)}\lvert A\rvert^2\, d\L^5\le C\lim_{i\to+\infty}\int_{\Omega_{i,\Lambda,A}}\lvert A\rvert^2\, d\L^5,
    \end{align}
    where $\Omega_{i,\Lambda,A}\subset Q_1^5(0)$ is given by
    \begin{align}
        \Omega_{i,\Lambda,A}:=\bigcup\bigg\{Q\in\mathcal{Q}_{\eps_i} \mbox{ s.t. } \fint_{Q_{2\eps_i}(c_Q)}\lvert A\rvert^2\,d\L^5>\Lambda\bigg\}.
    \end{align}
\end{Lm}
\begin{proof}[\textbf{\textup{Proof of Lemma \ref{Lemma: norm of A on bad cubes}}}]
    Observe the following.
    \begin{itemize}
        \item If $Q\in\mathcal{Q}_{\eps_i}$ is such that (1) fails, we have
            \begin{align}
                \eps_i\int_{\partial Q}\lvert F_A\rvert^2\, d\H^4>\eps_i^{\frac{9}{2}}\int_{Q_1^5(0)}\lvert F_A\rvert^2\, d\L^5.
            \end{align}
            \item If $Q\in\mathcal{Q}_{\eps_i}$ is such that (2) fails, we have 
            \begin{align}
                \eps_i\int_{\partial Q}\lvert A\rvert^2\, d\H^4>\eps_i^{\frac{9}{2}}\int_{Q_1^5(0)}\lvert F_A\rvert^2\, d\L^5,
            \end{align}
            \item If $Q\in\mathcal{Q}_{\eps_i}$ is such that (3) fails, we have 
            \begin{align}
                \eps_i\int_{\partial Q}\lvert A-\bar A_Q\rvert^2\, d\H^4>\int_{Q_{2\eps_i}(c_Q)}\lvert A\rvert^2\, d\L^5.
            \end{align}
            \item If $Q\in\mathcal{Q}_{\eps_i}$ is such that (4) fails, we have
            \begin{align}
                \eps_i\int_{\partial Q}\lvert A-\bar A_Q\rvert^2\, d\H^4>\eps_i^5\Lambda^{-1}\int_{Q_1^5(0)}\lvert F_A\rvert^2\, d\L^5.
            \end{align}
    \end{itemize}
    Summing up over the appropriate cubes and by Remark \ref{Remark: estimate on the sum of the norm of A resticted to the boundaries}, for $\eps_i>0$ small enough we get the following set of estimates:
    \begin{align}
        &\operatorname{card}(\{Q\in\mathcal{Q}_{\eps_i}\mbox{ s.t. (1) fails}\})<C\eps_i^{-\frac{9}{2}},\\ 
        &\operatorname{card}(\{Q\in\mathcal{Q}_{\eps_i}\mbox{ s.t. (2) fails}\})<C\bigg(\int_{Q_1^5(0)}\lvert A\rvert^2\, d\L^5\bigg)\bigg(\int_{Q_1^5(0)}\lvert F_A\rvert^2\, d\L^5\bigg)^{-1}\eps_i^{-\frac{9}{2}},\\
        &\sum_{\substack{Q\in\mathcal{Q}_{\eps_i}\\\text{ s.t. (3) fails}}}\int_{Q_{2\eps_i}(c_Q)}\lvert A\rvert^2\, d\L^5<\eps_i\sum_{Q\in\mathcal{Q}_{\eps_i}}\int_{\partial Q}\lvert A-\bar A_Q\rvert^2\, d\H^4,\\
        &\operatorname{card}(\{Q\in\mathcal{Q}_{\eps_i}\mbox{ s.t. (4) fails}\})<\Lambda\bigg(\int_{Q_1^5(0)}\lvert F_A\rvert^2\, d\L^5\bigg)^{-1}\eps_i^{-5}\bigg(\eps_i\sum_{Q\in\mathcal{Q}_{\eps_i}}\int_{\partial Q}\lvert A-\bar A_Q\rvert^2\, d\H^4\bigg).
    \end{align}
    In particular, we note that the $\L^5$-measure of the union of the cubes in $\mathcal{Q}_{\eps_i}$ such that (1),(2),(4) fail vanishes at the limit $i\to+\infty$. Thus,
    \begin{align}
        \lim_{i\to+\infty}\sum_{\substack{Q\in\mathcal{Q}_{\eps_i} \mbox{ s.t.}\\ (1),(2),(3),(4)\\\mbox{ fail}}}\int_{Q_{2\eps_i}(c_Q)}\lvert A\rvert^2\, d\L^5=0.
    \end{align}
    Finally, we get
    \begin{align}
        \lim_{i\to+\infty}\sum_{Q\in\mathcal{Q}_{\eps_i,\Lambda}^b}\int_{Q_{2\eps_i}(c_Q)}\lvert A\rvert^2\, d\L^5\le\lim_{i\to+\infty}\sum_{\substack{Q\in\mathcal{Q}_{\eps_i}\\\text{ s.t. (5) fails}}}\int_{Q_{2\eps_i}(c_Q)}\lvert A\rvert^2\, d\L^5
    \end{align}
    and the statement follows by the intersection properties of the cubic cover. This concludes the proof of Lemma \ref{Lemma: norm of A on bad cubes}. 
\end{proof}
\begin{Rm}\label{Remark: comparison boundary interior for the connection}
    Let $A\in\A_G(Q_1^5(0))$. Let $\eps\in\big(0,\frac{1}{4}\big)$ be a good scale for $A$ and let $\mathcal{Q}_{\eps}$ be an admissible cubic $\eps$-cover relative to $A$. Notice that the property (3) in Definition \ref{Definition: good cubes eps-cover} immediately implies that
    \begin{align}
        \int_{\partial Q}\lvert A\rvert^2\, d\H^4&\le\int_{\partial Q}\lvert A-\bar A_Q\rvert^2\, d\H^4+\int_{\partial Q}\lvert\bar A_Q\rvert^2\, d\H^4\\
        &\le\frac{1}{\eps}\int_{Q_{2\eps}(c_Q)}\lvert A\rvert^2\, d\L^5+C\eps^4\frac{1}{\eps^5}\int_{Q_{2\eps}(c_Q)}\lvert A\rvert^2\, d\L^5\le\frac{C}{\eps}\int_{Q_{2\eps}(c_Q)}\lvert A\rvert^2\, d\L^5
    \end{align}
    for every good $\Lambda$-good cube $Q\in\mathcal{Q}_{\eps}$, where $C>0$ depends only on the dimension. 
\end{Rm}
\section{The first smoothification}\label{section: approximation of weak connections under Morrey norm control}
The goal of this section is to approximate in $L^2$-norm any weak connection $A$ whose curvature is small in Morrey norm by forms $A_{i,\La}$ whose curvature is suitably controlled in $L^2$-norm on a set of good slices at the same scale $\rho$.
\begin{Lm}\label{Lemma: weak connection}
Let $Q=Q_{\rho}(c)\subset\r^5$ be an open cube in $\,\r^5$. Fix any $\g$-valued $1$-form $A\in L^2(\partial Q)$ with $F_A\in L^2(\partial Q)$. Let $\{Q_i\}_{i=0,...,k-1}$ be a collection of open cubes such that $\p Q$ is transversal to $\p Q_i$ for all $i$ and assume that $\p Q\subset  \bigcup_{i=0}^{k-1}Q_i$, $A\in {\mathfrak a}_G(\p Q\cap Q_i))$ for every $i=0,...,k-1$. Then, there exists $\eps_G\in(0,1)$ depending only on $G$ such that if
\begin{align}\label{small-curv-I}
	\int_{\partial Q}\lvert F_A\rvert^2\, d\H^4<\eps_G
\end{align}
we can find a global gauge $g\in W^{1,2}(\partial Q,G)$ such that $A^g\in L^4(\partial Q)$.
\end{Lm}
\begin{proof}[\textbf{\textup{Proof of Lemma \ref{Lemma: weak connection}}}] Since $A\in {\mathfrak a}_G(\p Q\cap Q_i))$ and since $\p Q\cap Q_i$ is bi-Lipschitz homeomorphic to the $4$-dimensional ball, there exists $h_i\in W^{1,2}(\p Q\cap Q_i,G)$ such that $A^{h_i}\in L^4(\p Q\cap Q_i)$. We consider $\psi_{c,\rho}:=(\rho\,\cdot\,+\,c)\circ\varphi$ which realizes a bi-Lipschitz homeomorphism
from $\s^4$ into $\p Q$ and we denote $U_i$ the Lipschitz open subset of $\s^4$ given by $U_i:=\psi_{c,\rho}^{-1}(\p Q\cap Q_i)$. Let $(\om_i)_{i=0\cdots k-1}$ be an open cover of $\s^4$ such that $\ov{\om_i}\subset U_i$. We then cover each
$\om_i$ by finitely many smooth convex geodesic balls $(V_{ij})_{j\in I_i}$  of $S^4$ with $V_{il}\subset U_i$ so that the whole cover given by the $V_{ij}$ is a good cover of $\s^4$ (i.e. the intersections of the $V_{ij}$s are again diffeomorphic to balls) by convex geodesic balls. To simplify the notations we re-index the cover 
$(V_{ij})_{i\in\{0\cdots k-1\},j\in I_i}$ by $(O_l)_{l=1\cdots L}$ and we denote by $i_l$ the index $i\in\{0\cdots k-1\}$ such that $O_l\subset U_i$. Observe that
\begin{align}
    \|\psi_{c,\rho}^\ast F_A\|_{L^2(\s^4)}\le C\, \|d\varphi\|^2_{L^\infty(\s^4)}\ \| F_A\|_{L^2(\p Q)}\le C\, \|d\varphi\|^2_\infty\ \sqrt{\ep_G}.
\end{align}
We first choose $\ep_G>0$ in \eqref{small-curv-I} such that we can apply Proposition \ref{appendix: Coulomb gauge extraction for  forms} to $\,\psi_{c,\rho}^\ast (A^{h_{i_l}})$ on the geodesic ball $O_l$. Let $g_l\in W^{1,2}(O_l,G)$ given by Proposition \ref{appendix: Coulomb gauge extraction for  forms} such that
\begin{align}
    B_l:=\lf(\psi_{c,\rho}^*(A^{h_{i_l}})\rg)^{g_l}
\end{align}
satisfies
\begin{align}
    \|B_l\|_{W^{1,2}(O_l)}\le C\, \|\psi_{c,\rho}^\ast F_A\|_{L^2(\s^4)}\le C\, \|d\varphi\|^2_{L^\infty(\s^4)}\ \| F_A\|_{L^2(\p Q)}
\end{align}
and
\begin{align}
    d^{\ast_{\s^4}}B_l=0.
\end{align}
Using the same argument as in \cite[proof of Theorem V.5]{Riv}, we have that the transition functions $\sigma_{mn}:=g_n^{-1}\,(h_{i_n}\circ\psi_{c,\rho}^{-1})^{-1}h_{i_m}\circ\psi_{c,\rho}^{-1}\,g_m$ are continuous on $O_n\cap O_m$.
and the co-cycle generated by $(\sigma_{mn})_{m,n\in\{1\cdots L\}}$ is $C^0$ approximable by a sequence of smooth ones. Hence, $\sigma_{mn}$ are the transition functions of a smooth bundle and $\psi_{c,\rho}^*A$ defines a $W^{1,2}-$Sobolev connection on this $G-$bundle. Using \cite{Uh3}, for $\ep_G$ small enough the bundle is trivial and there exists a global $W^{1,2}$ representative of $\psi_{c,\rho}^* A$. Pulling back this representative by
$\psi_{c,\rho}^{-1}$ we have the existence of $g\in W^{1,2}(\partial Q,G)$ such that $A^g\in L^4(\partial Q)$. This concludes the proof of Lemma \ref{Lemma: weak connection}.
\end{proof}
\noindent
We recall the definition of the Morrey seminorms in $\Om\subset {\R}^n$: 
\begin{align*}
	|f|_{M^s_{p,q}(\Om)}:=\sup_{\substack{x\in\Om\\\rho>0}}\lf(\frac{1}{\rho^{n-pq}}\int_{Q_\rho(x)}\sum_{s_1+\cdots+s_n=s}\lf|\p^{s_1}_{x_1}\cdots \p_{x_n}^{s_n}f\rg|^p(x)\, d\L^n  \rg)^{\frac{1}{p}},
\end{align*}
for every $s\in{\N}$, $1\le p<+\infty$, $0<q<\frac{n}{p}$.
\begin{Dfi}
    Given an admissible cubic $\eps$-cover $\mathcal{Q}_{\eps}$, we say that a cube $Q\subset\r^5$ is \textit{uniformly transversal} to $\mathcal{Q}_{\eps}$ if there are universal constants $0<\alpha\le k$ such that for every $Q'\in\mathcal{Q}_{\eps}$ we have that $Q\cap Q'$ is either empty or a rectangle with minimum edge length $\alpha\eps$ and maximum edge length $k\eps$.
\end{Dfi}
\begin{Th}[Approximation under controlled traces of the curvature]\label{Theorem: approximation under controlled traces of the curvatures}
Let $G$ be a compact matrix Lie group and let $\Lambda>0$. There exist $\eps_G\in (0,1)$ and $C_G>0$ depending only on $G$ such that for every $A\in\A_G(Q_1^5(0))$ satisfying 
\begin{align*}
   |F_A|_{M^0_{2,2}(Q_1^5(0))}^2:=\sup_{\substack{x\in Q_1^5(0)\\\rho>0}}\frac{1}{\rho}\int_{Q_{\rho}(x)}\lvert F_A\rvert^2\, d\L^5<\eps_G
\end{align*}
we can find a set $\{\eps_i\}_{i\in\n}\subset\big(0,\frac{1}{4}\big)$ of admissible scales  for $A$ (with associated admissible cubic $\eps_i$-covers $\mathcal{Q}_{\eps_i}$) and a family of $\g$-valued 1-forms $\{A_{i,\Lambda}\}_{i\in\n}\subset L^2(Q_{1}^5(0))$ such that the following facts hold.
\begin{enumerate}[(i)]
    \item For every $i\in\n$, $F_{A_{i,\Lambda}}\in M^0_{2,2}(Q_{1}^5(0))$ with 
    \begin{align*}
        |F_{A_{i,\Lambda}}|_{M^0_{2,2}(Q_{1}^5(0))}<C_G\,|F_{A}|_{M^0_{2,2}(Q_{1}^5(0))}.
    \end{align*}
    \item There exist a universal integer $N\in\n$ and a non-negative $f\in M^0_{1,4}(Q_1^5(0))$ satisfying 
    \begin{align*}
    |f|_{M^0_{1,4}(Q_1^5(0))}\le |F_A|_{M^0_{2,2}(Q_1^5(0))}^2
    \end{align*}
    such that for every $i\in\n$, for every $x\in Q_1^5(0)$ and for every $\rho>0$ for which $Q_{\rho}(x)$ is uniformly transversal to $\mathcal{Q}_{\eps_i}$ and $\partial Q_{\rho}(x)\subset Q_{1-\frac{\eps_i}{2}}(0)$ we have
    \begin{align}\label{equation: properties of F_A on the boundaries}
        \int_{\partial Q_{\rho}(x)}\lvert F_{A_{i,\Lambda}}\rvert^2\, d\H^4\le \frac{C_G}{\eps_i}\sum_{k=1}^N\int_{Q_{2\eps_i}(x_k)}f\, d\L^5,
    \end{align}
    for an $N$-tuple of points $\{x_k=x_k(i,x,\rho)\}_{k=1,...,N}\subset Q_1^5(0)$. 
    
    Moreover, for every $i\in\n$, for every $x\in Q_1^5(0)$ and for $\L^1$-a.e. $\rho>0$ for which $Q_{\rho}(x)$ is uniformly transversal to $\mathcal{Q}_{\eps_i}$ and $\partial Q_{\rho}(x)\subset Q_{1-\frac{\eps_i}{2}}(0)$ we have $\iota_{\partial Q_{\rho}(x)}^*A_{i,\Lambda}\in{\A}_G(\partial Q_{\rho}(x))$.
    \item We have 
    \begin{align*}
        \lim_{i\to+\infty}\|A_{i,\Lambda}-A\|_{L^2(Q_{1}^5(0))}^2\le C_G\lim_{i\to+\infty}\int_{\Omega_{i,\Lambda,A}}\lvert A\rvert^2\, d\L^5,
    \end{align*}
    where $\Omega_{i,\Lambda,A}\subset Q_1^5(0)$ is defined as in Lemma \ref{Lemma: norm of A on bad cubes}.
\end{enumerate}
\end{Th}
\begin{Rm}
\label{rm-app}
It is not claimed in Theorem~\ref{Theorem: approximation under controlled traces of the curvatures} that the elements $A_{i,\La}$ are in ${\mathfrak a}_G(Q_1^5(0))$. It could very well be that this is the case and that it could be checked
out of the construction given in the proof but, since it is not needed in the sequel, this property is left open.
\end{Rm}
\noindent
Before proving Theorem \ref{Theorem: approximation under controlled traces of the curvatures}, we shall need the following lemma, which will be used at several steps of the proof of Theorem \ref{Theorem: approximation under controlled traces of the curvatures} and deserves to be stressed separately.
\begin{Lm}
\label{lm-cont}
Let $A\in L^2(Q_1^5(0))$, let $Q_\rho(c)\subset Q_1^5(0)$ and let $p\in[1,+\infty)$ and $q\in[1,+\infty]$. Assume moreover that the following facts hold.
\begin{enumerate}[(i)]
\item $F_A:=dA+A\wedge A \in L^{p,q}(Q_\rho(c))\cap L^{p,q}(Q_1^5(0)\setminus Q_\rho(c))$.
\item There exists an $L^2$ $1$-form $\iota_{\p Q_\rho(c)}^\ast A$ on $\p Q_\rho(c)$ such that
\begin{align}
    \lim_{\ep\rightarrow 0}\frac{1}{\ep}\int_{\rho-\ep}^{\rho+\ep}\int_{\p Q_\rho(c)}\lf|D_r ^\ast\iota_{\p Q_r(c)}^\ast A-\iota_{\p Q_\rho(c)}^\ast A\rg|^2\ d\H^4\, d\L^1(r)=0,
\end{align}
where $D_r(x):= \frac{r}{\rho}\,(x-c)+c$.
\end{enumerate}
Then, $F_A:=dA+A\wedge A \in L^{p,q}(Q_1^5(0))\ $.
\end{Lm}
\begin{proof}[\textbf{\textup{Proof of Lemma \ref{lm-cont}}}]
Since the proof is identical for $q\neq p$, to fix the ideas we assume that $q=p$, i.e. $L^{p,q}(Q_1^5(0))=L^p(Q_1^5(0))$. We first claim that $F_A\in L^1(Q_1^5(0))$. Since $A\in L^2(Q_1^5(0))$, automatically we have $A\wedge A\in L^1(Q_1^5(0))$ and it suffices to show that the distributional differential of $A$ is a $2$-form in $L^1(Q_1^5(0))$. We claim that the distributional differential of $A$ on $Q_1^5(0)$ is exactly $dA\in L^1(Q_1^5(0))$. Let $\varphi\in C_c^{\infty}(Q_1^5(0))$ be any smooth and compactly supported $3$-form on $Q_1^5(0)$. Fix any $\eps>0$ and let $\eta_{\eps}\in C_c^{\infty}(Q_1^5(0))$ a cut-off function that vanishes on $\p Q_{\rho}(c)$ and satisfies 
\begin{align}
    \eta_{\eps}\equiv 1 \qquad&\mbox{ on } \big(Q_1^5(0)\smallsetminus Q_{\rho+\eps}(c)\big)\cup Q_{\rho-\eps}(c),\\
    0\le\eta_{\eps}\le 1 \qquad&\mbox{ on } Q_1^5(0)
\end{align}
and
\begin{align}
    \|d\eta_{\eps}\|_{L^{\infty}(Q_1^5(0))}\le\frac{C}{\eps}.
\end{align}
Notice that $\eta_{\eps}\varphi\in W_0^{1,1}(Q_1^5(0)\smallsetminus\p Q_{\rho}(c))$. Hence,
\begin{align}
    \int_{Q_1^5(0)}A\wedge d(\eta_{\eps}\varphi)=\int_{Q_1^5(0)}\eta_{\eps}(dA\wedge\varphi).
\end{align}
Moreover, 
\begin{align}
    d(\eta_{\eps}\varphi)=d\eta_{\eps}\wedge\varphi+\eta_{\eps}d\varphi.
\end{align}
Thus,
\begin{align}
    \int_{Q_1^5(0)}A\wedge d(\eta_{\eps}\varphi)=\int_{Q_1^5(0)}\eta_{\eps}(A\wedge d\varphi)+\int_{Q_1^5(0)}A\wedge d\eta_{\eps}\wedge\varphi.
\end{align}
This implies that
\begin{align}\label{74}
    \int_{Q_1^5(0)}\eta_{\eps}(A\wedge d\varphi)=\int_{Q_1^5(0)}\eta_{\eps}(dA\wedge\varphi)-\int_{Q_1^5(0)}A\wedge d\eta_{\eps}\wedge\varphi.
\end{align}
By assumption (ii), we have
\begin{align}
    \int_{Q_1^5(0)}A\wedge d\eta_{\eps}\wedge\varphi\to 0
\end{align}
as $\eps\to 0$. Hence, by passing to the limit as $\eps\to 0$ in \eqref{74} we get
\begin{align}
    \int_{Q_1^5(0)}A\wedge d\varphi=\int_{Q_1^5(0)}dA\wedge\varphi,
\end{align}
which implies that $dA\in L^1(Q_1^5(0))$ is the distributional differential of $A$ on $Q_1^5(0)$ and our claim follows.

\noindent
Let $\varphi\in C_c^{\infty}(Q_1^5(0))$ a smooth and compactly supported $2$-form on $Q_1^5(0)$ such that $\|\varphi\|_{L^{p'}(Q_1^5(0))}\le 1$. Fix any $\eps>0$ and let
\begin{align*}
    \int_{Q_1^5(0)}F_A\wedge *\varphi&=\int_{Q_1^5(0)\smallsetminus Q_{\rho+\eps}(c)}F_A\wedge *\varphi+\int_{Q_{\rho+\eps}(c)\smallsetminus Q_{\rho-\eps}(c)}F_A\wedge *\varphi+\int_{Q_{\rho-\eps}(c)}F_A\wedge *\varphi\\
    &\le \int_{Q_1^5(0)\smallsetminus Q_{\rho}(c)}F_A\wedge *\varphi+\int_{Q_{\rho+\eps}(c)\smallsetminus Q_{\rho-\eps}(c)}F_A\wedge *\varphi+\int_{Q_{\rho}(c)}F_A\wedge *\varphi\\
    &\le \|F_A\|_{L^p(Q_1^5(0)\smallsetminus Q_{\rho}(c))}+\|F_A\|_{L^p(Q_{\rho}(c))}+\int_{Q_{\rho+\eps}(c)\smallsetminus Q_{\rho-\eps}(c)}F_A\wedge *\varphi.
\end{align*}
Since we have already shown that $F_A\in L^1(Q_1^5(0))$, by taking the limit as $\eps\to 0^+$ in the above inequality we have 
\begin{align*}
    \int_{Q_1^5(0)}F_A\wedge *\varphi&\le\|F_A\|_{L^p(Q_1^5(0)\smallsetminus Q_{\rho}(c))}+\|F_A\|_{L^p(Q_{\rho}(c))}.
\end{align*}
By taking the supremum over $\varphi\in C_c^{\infty}(Q_1^5(0))$ a smooth and compactly supported $2$-form on $Q_1^5(0)$ such that $\|\varphi\|_{L^2(Q_1^5(0))}\le 1$, we finally get
\begin{align*}
    \|F_A\|_{L^p(Q_1^5(0))}&\le\|F_A\|_{L^p(Q_1^5(0)\smallsetminus Q_{\rho}(c))}+\|F_A\|_{L^p(Q_{\rho}(c))}.
\end{align*}
This concludes the proof of Lemma \ref{lm-cont}.
\end{proof}
\bigskip
\begin{proof}[\textbf{\textup{Proof of Theorem \ref{Theorem: approximation under controlled traces of the curvatures}}}]
Let $\{\ep_i\}\subset\big(0,\frac{1}{4}\big)$ be a set of admissible scales for $A$ with associated admissible cubic $\eps_i$-cover $\mathcal{Q}_{\eps_i}$. We subdivide the family of cubes $\mathcal{Q}_{\eps_i}$ into a uniformly bounded number $N$ of disjoint subfamilies
\begin{align}
\mathcal{Q}_{\eps_i}^1,...,\mathcal{Q}_{\eps_i}^{N}.
\end{align}
We denote by $\mathscr{C}^s_{\eps_i}$ the set of centers of the cubes in $ \mathcal{Q}_{\eps_i}^{s}$. The subfamilies $\mathcal{Q}_{\eps_i}^{s}$ are chosen such that the union of the cubes in $\mathcal{Q}_{\eps_i}^j$ with same centers and radii multiplied by 2 have no intersections with the union of the cubes in $\mathcal{Q}_{\eps_i}^k$ with same centers and radii multiplied by 2 whenever $j\ne k$. 

\medskip
\noindent
We claim that for every $s=0,...,N$ we can build a $\g$-valued $1$-form $A_{i,
\Lambda}^s\in L^2(Q_{1}^5(0))$ and we can choose a family of radii $\rho_i^s\in \big(\frac{3}{2}\ep_i,2\ep_i\big)$ for each center $c\in \mathscr{C}^s_{\eps_i}$ in the family of cubes in ${\mathcal Q}_{\eps_i}^s$ such that for some constant $C_s>0$ depending only on $G$ we have the following.
\begin{enumerate}[(a)]
    \item We have
    \begin{align}\label{X-35sa}
        \forall\, c\in\mathscr{C}^s_{\eps_i} \quad A_{i,\La}^{s-1}\in{\mathfrak a}_G(\p Q_{\rho_i^s}(c)),
    \end{align}
    and the estimates
    \begin{align}\label{X-35a-b}
        \int_{\partial Q^s_{\rho_{i,c}}(c)}\lvert F_{A_{i,\La}^{s-1}}\rvert^2\, d\H^4\le\frac{C_s}{2\eps_i}\int_{Q^s_{2\eps_i}(c)}\lvert  F_{A_{i,\La}^{s-1}}\rvert^2\, d\L^5,
    \end{align}
    \begin{align}\label{X-355a-b}
        \int_{\partial Q^s_{\rho_{i,c}}(c)}\lvert {A_{i,\La}^{s-1}}\rvert^2\, d\H^4\le\frac{C_s}{2\eps_i}\int_{Q^s_{2\eps_i}(c)}\lvert {A_{i,\La}^{s-1}}\rvert^2\, d\L^5.
    \end{align}
    \item $F_{A_{i,\Lambda}^s}\in M^0_{2,2}(Q_{1}^5(0))$ with
    \begin{align}
        |F_{A_{i,\Lambda}^s}|_{M^0_{2,2}(Q_{1-3\eps_i}(0))}<C_s\,|F_{A}|_{M^0_{2,2}(Q_{1}^5(0))}.
    \end{align}
    \item There exist $N_s\in\n$ and a real-valued, non-negative function $f_s\in M_{2,2}^0(Q_1^5(0))$ satisfying 
    \begin{align}
        |f_s|_{M^0_{1,4}(Q_1^5(0))}\le\ |F_A|_{M_{2,2}^0(Q_1^5(0))}^2
    \end{align}
    such that, for every $x\in Q_1^5(0)$ and for every $\rho>0$ for which $Q_{\rho}(x)$ is uniformly transversal to $\mathcal{Q}_{\eps_i}$, we have
    \begin{align}\label{31}
        \int_{\partial Q_{\rho}(x)\cap\Omega_i^s}\lvert F_{A_{i,\Lambda}}\rvert^2\, d\H^4\le \frac{C_s}{\eps_i}\sum_{k=1}^{N_s}\int_{Q_{2\eps_i}(x_k)}f_s\, d\L^5
    \end{align}
    for a family of $N_s$ points $\{x_k^s=x_k^s(i,x,\rho)\}_{k=1,...,N_s}\subset Q_1^5(0)$, where $\Omega_i^s\subset Q_1^5(0)$ is the open set given by
    \begin{align}
        \Omega_i^s:=\operatorname{int}\bigg(\bigcup_{k=1}^s\bigcup_{Q\in\mathcal{Q}_{\eps_i}^k}\overline{Q}\bigg).
    \end{align}
    \item It holds that
    \begin{align}
        \lim_{i\to +\infty}\|A_{i,\Lambda}^s-A\|_{L^2(Q_{1}^5(0))}^2\le C_s\lim_{i\to+\infty}\int_{\Omega_{i,\Lambda,A}}\lvert A\rvert^2\, d\L^5,
    \end{align}
    where $\Omega_{i,\Lambda,A}\subset Q_1^5(0)$ is as in Lemma \ref{Lemma: norm of A on bad cubes}.
\end{enumerate}

\medskip
\noindent
\textit{Base of the induction}. For $s=0$, since $\Omega_i^0=\emptyset$ we have that $A_{i,\Lambda}^0:=A$ satisfies (a), (b), (c) and (d).

\medskip
\noindent
\textit{Induction step}. Let $s\ge 1$. Using one more time Fubini theorem combined with the mean value theorem we adjust the radii $\rho_{i,c}\in \big(\frac{3}{2}\eps_i,2\eps_i\big)$ in ${\mathcal Q}_{\eps_i}^s$ such that $A_{i,\Lambda}^{s-1}$ satisfies the properties in (a). From now on, in the induction procedure, the family ${\mathcal Q}^s_{\ep_i}$ is fixed.

\noindent
By assumptions (1) and (2) in Definition \ref{Definition: good cubes eps-cover}, if $\eps_G>0$ is small enough then $\iota_{\partial Q}^*A_{i,\Lambda}^{s-1}\in\A_G(\partial Q)$ satisfies the hypotheses of Corollary \ref{Corollary: extension in the interior of good cubes} for every $\Lambda$-good cube $Q\in\mathcal{Q}_{\eps_i}^{s,g}$. Moreover, by our choice of the cubic cover (see Lemma \ref{Lemma: choice of an admissible cubic cover}-(i)), we have that $\iota_{\partial Q}^*A_{i,\Lambda}^{s-1}\in\A_G(\partial Q)$ satisfies the hypotheses of Corollary \ref{Corollary: extension in the interior of bad cubes} for every $Q\in\mathcal{Q}_{\eps_i}^s$. 

\noindent
For every $Q\in\mathcal{Q}_{\eps_i}^s$, if $Q$ is $\Lambda$-good then we let $A_Q^s$ be the $\g$-valued 1-form given by applying Corollary \ref{Corollary: extension in the interior of good cubes}--(ii) to $\iota_{\partial Q}^*A_{i,\Lambda}^{s-1}$. On the other hand, if $Q$ is $\Lambda$-bad we let $A_Q^s$ be the $\g$-valued 1-form given by applying Corollary \ref{Corollary: extension in the interior of bad cubes}--(ii) to $\iota_{\partial Q}^*A_{i,\Lambda}^{s-1}$.

\noindent
Let $A_{i,\Lambda}^s$ be defined on $Q_{1}^5(0)$ by
\begin{align*}
    A_{i,\Lambda}^s:=\begin{cases}A_Q^s & \text{ on } Q \text{ for every } Q\in\mathcal{Q}_{\eps_i}^s,\\A_{i,\Lambda}^{s-1} & \text{ otherwise.}\end{cases}
\end{align*}
By construction and using Lemma \ref{lm-cont}, it is clear that $A_{i,\Lambda}^s,F_{A_{i,\Lambda}^s}\in L^2(Q_{1}^5(0))$. First, we show that $A_{i,\Lambda}^s$ satisfies (b). Fix any point $x\in Q_{1}^5(0)$ and let $\rho\in\left(0,\dist(x,\partial Q_{1}^5(0))\right)$. First we assume that $\rho\le\eps_i$. Then, $Q_{\rho}(x)$ intersects at most $\tilde N$ cubes in $\mathcal{Q}_{\eps_i}^s$, say $Q_1,...,Q_{\tilde N}$, with $\tilde N\in\n$ depending only on the choice of the cubic cover and independent on $i$, $x$ and $\rho$. Hence, by H\"older inequality, by the estimates \eqref{equation: estimate for the curvature on the good cubes} for the good $Q_{\ell}$s and \eqref{equation: estimate for the curvature on the bad cubes} for the bad ones for each $A_Q^s$, as well as \eqref{X-35a} and the inductive assumption (b) on $A_{i,\Lambda}^{s-1}$, we get
\begin{align}\label{indu-morr-sm}
    \nonumber
    \frac{1}{\rho}\int_{Q_{\rho}(x)\cap(\Omega_{i}^s\smallsetminus\Omega_i^{s-1})}&\lvert F_{A_{i,\Lambda}^s}\rvert^2\, d\L^5\le\|F_{A_{i,\Lambda}^s}\|_{L^{\frac{5}{2}}(Q_{\rho}(x)\cap(\Omega_{i}^s\smallsetminus\Omega_i^{s-1}))}^2\\[\sep]
    \nonumber
    &\le C\sum_{\ell=1}^{\tilde N}\lf\| F_{A_{i,\Lambda}^s}\rg\|_{L^{\frac{5}{2}}(Q_\ell)}^2=C\sum_{i=1}^{\tilde N}\lVert F_{A_{Q_\ell}^s}\rVert_{L^{\frac{5}{2}}(Q_\ell)}^2\\[\sep]
    \nonumber
    &\le C\sum_{Q_\ell\in\mathcal{Q}_{\eps_i,\Lambda}^g}\, \lf(\lVert F_{A_{i,\Lambda}^{s-1}}\rVert_{L^{2}(\partial Q_\ell)}^2+\eps_i^{-2}\lVert A_{i,\Lambda}^{s-1}\rVert_{L^{2}(\partial Q_\ell)}^2\rg)+\sum_{Q_\ell\in\mathcal{Q}_{\eps_i,\Lambda}^b}\lVert F_{A_{i,\Lambda}^{s-1}}\rVert_{L^{2}(\partial Q_\ell)}^2\\[\sep]
    &\le C\ |F_{A_{i,\Lambda}^{s-1}}|_{M^0_{2,2}(Q_1^5(0))}^2\le C\ |F_A|_{M^0_{2,2}(Q_1^5(0))}^2,
\end{align}
for some constant $C>0$ depending only on $G$.

\noindent
Assume now that $\,\eps_i<\rho<1$. Then there exists a universal constant $k>1$ such that $Q_{\rho}(x)\cap(\Omega_{i}^s\smallsetminus\Omega_i^{s-1})$ can be covered with a finite number of cubes $\{Q_\ell\}_{\ell=1,...,N_{i,x,\rho}}$ in $\mathcal{Q}_{\eps_i}^s$ such that 
\begin{align*}
    \bigcup_{\ell=1}^{N_{i,x,\rho}}Q_{2\eps_i}(c_{Q_\ell})\subset Q_{k\rho}(x).
\end{align*}
Notice that now the number of cubes $N_{i,x,\rho}$ may depend on $i$, $x$ and $\rho$. Then, by the estimates \eqref{equation: estimate for the curvature on the good cubes} and \eqref{equation: estimate for the curvature on the bad cubes} we obtain
\begin{align}\label{indu-morr-lar}
    \nonumber
    \int_{Q_{\rho}(x)\cap(\Omega_{i}^s\smallsetminus\Omega_i^{s-1})}&\lvert F_{A_{i,\Lambda}^s}\rvert^2\, d\L^5\le\sum_{\ell=1}^{N_{i,x,\rho}}\int_{Q_\ell}\lvert F_{A_{i,\Lambda}^s}\rvert^2\, d\L^5=\sum_{\ell=1}^{N_{i,x,\rho}}\int_{Q_\ell}\lvert F_{A_{Q_\ell}}\rvert^2\, d\L^5\\[\sep]
    &\le C\sum_{Q_\ell\in\mathcal{Q}_{\eps_i,\Lambda}^g}\,\eps_i^{-4}\int_{\partial Q_\ell}\lvert A_{i,\Lambda}^{s-1}- (A_{i,\Lambda}^{s-1})_{Q_\ell}\rvert^2\, d\H^4\ \ \eps_i\int_{\partial Q_\ell}\lvert A_{i,\Lambda}^{s-1}\rvert^2\, d\H^4\\[\sep]
    \nonumber
    &\quad+\ C\sum_{Q_\ell\in\mathcal{Q}_{\eps_i,\Lambda}^g}\,\eps_i\,\lf(\eps_i^{-2}\int_{\partial Q_\ell}\lvert A_{i,\Lambda}^{s-1}\rvert^2\, d\H^4\rg)^3+C\eps_i\sum_{\ell=1}^{N_{i,x,\rho}}\int_{\partial Q_\ell}\lvert F_{A_{i,\Lambda}^{s-1}}\rvert^2\, d\H^4,
\end{align}
for some constant $C>0$ depending only on $G$. By property (4) in Definition \ref{Definition: good cubes eps-cover}, by Remark \ref{Remark: comparison boundary interior for the connection} and by inductive hypothesis (a) on $A_{i,\Lambda}^{s-1}$, we get
\begin{align*}
    \sum_{Q_\ell\in\mathcal{Q}_{\eps_i,\Lambda}^g}\eps_i^{-4}\int_{\partial Q_\ell}&\lvert A_{i,\Lambda}^{s-1}- (A_{i,\Lambda}^{s-1})_{Q_\ell}\rvert^2\, d\H^4\ \ \eps_i\int_{\partial Q_\ell}\lvert A_{i,\Lambda}^{s-1}\rvert^2\, d\H^4\\
    &\le C\Lambda^{-1}\, \lf|F_{A_{i,\Lambda}^{s-1}}\rg|_{M_{2,2}(Q_1^5(0))}^2\sum_{Q_\ell\in\mathcal{Q}_{\eps_i,\Lambda}^g}\int_{Q_{2\eps_i}(c_{Q_\ell})}\lvert A_{i,\Lambda}^{s-1}\rvert^2\, d\L^5,\\
    &\le C\Lambda^{-1} \, \lf|F_{A}\rg|_{M_{2,2}^0(Q_1^5(0))}^2\sum_{Q_\ell\in\mathcal{Q}_{\eps_i,\Lambda}^g}\int_{Q_{2\eps_i}(c_{Q_\ell})}\lvert A_{i,\Lambda}^{s-1}\rvert^2\, d\L^5,
\end{align*}
for some constant $C>0$ depending only on $G$. By property (5) in Definition \ref{Definition: good cubes eps-cover} and by our choice of $k>0$, we get
\begin{align*}
    \sum_{Q_\ell\in\mathcal{Q}_{\eps_i,\Lambda}^g}\int_{Q_{2\eps_i}(c_{Q_\ell})}\lvert A_{i,\Lambda}^{s-1}\rvert^2\, d\L^5\le C\Lambda\sum_{Q_\ell\in\mathcal{Q}_{\eps_i,\Lambda}^g}(2\eps_i)^5\le C\Lambda(k\rho)^5,
\end{align*}
for some constant $C>0$ depending only on $G$. Hence, since by assumption $\rho\in (0,1)$, we have obtained the estimate
\be
\label{X-39-0}
    \sum_{Q_\ell\in\mathcal{Q}_{\eps_i,\Lambda}^g}\eps_i^{-4}\,\int_{\partial Q_\ell}\lvert A_{i,\Lambda}^{s-1}- (A_{i,\Lambda}^{s-1})_{Q_\ell}\rvert^2\, d\H^4\ \ \eps_i\int_{\partial Q_\ell}\lvert A_{i,\Lambda}^{s-1}\rvert^2\, d\H^4\le C\,k^5\rho\,|F_A|_{M^0_{2,2}(Q_1^5(0))}^2,
\ee
for some constant $C>0$ depending only on $G$. By Remark \ref{Remark: comparison boundary interior for the connection} and by property (5) in Definition \ref{Definition: good cubes eps-cover}, provided $i\ge i_0$ is large enough so that $\eps_i<\Lambda^{-3}\,|F_A|_{M_{2,2}^0(Q_1^5(0))}^2$, we obtain
\begin{align}\label{X-39}
    \nonumber
    \sum_{Q_\ell\in\mathcal{Q}_{\eps_i,\Lambda}^g}\eps_i\ \lf(\eps_i^{-2}\int_{\partial Q_\ell}\lvert A_{i,\Lambda}^{s-1}\rvert^2\, d\H^4\rg)^3&\le\sum_{Q_\ell\in\mathcal{Q}_{\eps_i,\Lambda}^g}\eps_i\,\lf(\eps_i^{-3}\int_{Q_{2\eps_i}(c_{Q_\ell})}\lvert A_{i,\Lambda}^{s-1}\rvert^2\, d\L^5\rg)^3\\[\sep]
    &\le\sum_{Q_\ell\in\mathcal{Q}_{\eps_i,\Lambda}^g}\eps_i(\eps_i^{-3}\Lambda\eps_i^5)^3\\[\sep]
    \nonumber
    &\le\eps_i^2\,\Lambda^3\sum_{Q_\ell\in\mathcal{Q}_{\eps_i,\Lambda}^g}\eps_i^5\le C\eps_i^2\,\Lambda^3\,\le C\rho\ |F_A|_{M_{2,2}^0(Q_1^5(0))}^2,
\end{align}
for some constant $C>0$ depending only on $G$. By \eqref{X-35a} and the inductive hypothesis (a) on $A_{i,\Lambda}^{s-1}$, we get
\begin{align}\label{X-40-a}
\begin{split}
    \eps_i\sum_{\ell=1}^{N_{i,x\rho}}\int_{\partial Q_\ell}\lvert F_{A_{i,\Lambda}^{s-1}}\rvert^2\, d\H^4&\le C\sum_{\ell=1}^{N_{i,x\rho}}\int_{Q_{2\eps_i}(c_{Q_\ell})}\lvert F_{A_{i,\Lambda}^{s-1}}\rvert^2\, d\L^5\\[\sep]
    &\le C\int_{Q_{k\rho}(x)}\lvert F_{A_{i,\Lambda}^{s-1}}\rvert^2\, d\L^5\le C\,k\,\rho\ \lf|F_{A_{i,\Lambda}^{s-1}}\rg|_{M_{2,2}^0(Q_1^5(0))}^2\\
    &\le C\,k\,\rho\,|F_{A}|_{M^0_{2,2}(Q_1^5(0))}^2,
\end{split}
\end{align}
for some constant $C>0$ depending only on the choice of the cubic cover. Hence, provided $i\ge i_0$ is sufficiently large, combining \eqref{indu-morr-lar}, \eqref{X-39-0}, \eqref{X-39} and \eqref{X-40-a} we get this time for $\ep_i<\rho<1$
\begin{align}\label{XXIII}
    \frac{1}{\rho}\int_{Q_{\rho}(x)\cap(\Omega_{i}^s\smallsetminus\Omega_{i}^{s-1})}\lvert F_{A_{i,\Lambda}^s}\rvert^2\, d\L^5\le C\ \lf|F_A\rg|_{M_{2,2}^0(Q_1^5(0))}^2,
\end{align}
for some constant $C>0$ depending only on $G$. Since
\begin{align*}
    \frac{1}{\rho}\int_{Q_{\rho}(x)}\lvert F_{A_{i,\Lambda}^s}\rvert^2\, d\L^5&=\frac{1}{\rho}\int_{Q_{\rho}(x)\cap(\Omega_{i}^s\smallsetminus\Omega_{i}^{s-1})}\lvert F_{A_{i,\Lambda}^s}\rvert^2\, d\L^5+\frac{1}{\rho}\int_{Q_{\rho}(x)\smallsetminus(\Omega_{i}^s\smallsetminus\Omega_{i}^{s-1})}\lvert F_{A_{i,\Lambda}^s}\rvert^2\, d\L^5\\
    &\le\frac{1}{\rho}\int_{Q_{\rho}(x)\cap(\Omega_{i}^s\smallsetminus\Omega_{i}^{s-1})}\lvert F_{A_{i,\Lambda}^s}\rvert^2\, d\L^5+\frac{1}{\rho}\int_{Q_{\rho}(x)}\lvert F_{A_{i,\Lambda}^{s-1}}\rvert^2\, d\L^5,
\end{align*}
property (b) for $A_{i,\Lambda}^s$ follows by \eqref{XXIII} and inductive hypothesis b) for $A_{i,\Lambda}^{s-1}$.

\medskip
\noindent
Now we turn to show property (c) for $A_{i,\Lambda}^{s}$. Notice that 
\begin{align*}
     \|A_{i,\Lambda}^s-A\|_{L^2(Q_{1}^5(0))}^2&\le\|A_{i,\Lambda}^s-A_{i,\Lambda}^{s-1}\|_{L^2(Q_{1}^5(0))}^2+\|A_{i,\Lambda}^{s-1}-A\|_{L^2(Q_{1}^5(0))}^2\\[4mm]
     &\le\sum_{Q\in\mathcal{Q}_{\eps_i}^s}\int_{Q}\lvert A_Q^s-A_{i,\Lambda}^{s-1}\rvert^2\, d\L^5+\|A_{i,\Lambda}^{s-1}-A\|_{L^2(Q_{1}^5(0))}^2\\
     &\le\sum_{Q\in\mathcal{Q}_{\eps_i}^s\cap\mathcal{Q}_{\eps_i,\Lambda}^g}\int_{Q}\lvert A_Q^s-A_{i,\Lambda}^{s-1}\rvert^2\, d\L^5+\|A_{i,\Lambda}^{s-1}-A\|_{L^2(Q_{1}^5(0))}^2\\
     &\quad+\sum_{Q\in\mathcal{Q}_{\eps_i}^s\cap\mathcal{Q}_{\eps_i,\Lambda}^b}\int_{Q}\lvert A_Q^s-A_{i,\Lambda}^{s-1}\rvert^2\, d\L^5\\
     &\le 2\,I_{i,\Lambda,s}+2\,J_{i,\Lambda,s}+K_{i,\Lambda,s}+\|A_{i,\Lambda}^{s-1}-A\|_{L^2(Q_{1}^5(0))}^2,
\end{align*}
with
\begin{align*}
     I_{i,\Lambda,s}&:=\sum_{Q\in\mathcal{Q}_{\eps_i}^s\cap\mathcal{Q}_{\eps_i,\Lambda}^g}\int_{Q}\lvert A_Q^s-(A_{i,\Lambda}^{s-1})_Q\rvert^2\, d\L^5,\\
     J_{i,\Lambda,s}&:=\sum_{Q\in\mathcal{Q}_{\eps_i}^s\cap\mathcal{Q}_{\eps_i,\Lambda}^g}\int_{Q}\lvert A_{i,\Lambda}^{s-1}-(A_{i,\Lambda}^{s-1})_Q\rvert^2\, d\L^5,\\
     K_{i,\Lambda,s}&:=\sum_{Q\in\mathcal{Q}_{\eps_i}^s\cap\mathcal{Q}_{\eps_i,\Lambda}^b}\int_{Q}\lvert A_Q^s-A_{i,\Lambda}^{s-1}\rvert^2\, d\L^5.
\end{align*}
By using \eqref{equation: estimate for the harmonic extension-cube} with $\bar A=(A_{i,\Lambda}^{s-1})_Q$ for every good cube $Q$, by Lemma \ref{Lemma: choice of an admissible cubic cover} (in particular \eqref{equation: convergence of A minus its averages on an admissible cover}) and by properties (1),(2) in Definition \ref{Definition: good cubes eps-cover}, we have
\begin{align}\label{X-40-b}
    \nonumber
    I_{i,\Lambda,s}&=\sum_{Q\in\mathcal{Q}_{\eps_i}^s\cap\mathcal{Q}_{\eps_i,\Lambda}^g}\int_{Q}\lvert A_Q^s-(A_{i,\Lambda}^{s-1})_Q\rvert^2\, d\L^5\\[\sep]
    \nonumber
    &\le C\,\eps_i\sum_{Q\in\mathcal{Q}_{\eps_i,\Lambda}^g}\int_{\partial Q}\lvert A_{i,\Lambda}^{s-1}-(A_{i,\Lambda}^{s-1})_Q\rvert^2\, d\L^5
    % &\ds\quad+\bigg(\sum_{Q\in\mathcal{Q}_{\eps_i,\Lambda}^g}\int_{Q}\lvert A_{i,\Lambda}^{s-1}\rvert^2\, dx^5\bigg)(\eps_i^{\frac{3}{2}}+\eps_i^{\frac{7}{2}})\|F_{A_{i,\Lambda}^{s-1}}\|_{L^2(Q_1^5(0)}^2\\
    +\eps^3_i\sum_{Q\in\mathcal{Q}_{\eps_i,\Lambda}^g}\int_{\partial Q}\lvert F_{A_{i,\Lambda}^{s-1}}\rvert^2\, d\H^4\\[\sep]
    \nonumber
    &\quad\quad+\eps_i^{-1}\sum_{Q\in\mathcal{Q}_{\eps_i,\Lambda}^g}\|{A_{i,\Lambda}^{s-1}}\|_{L^2(\p Q)}^4\\[\sep]
    \nonumber
    &\le C\,\eps_i\sum_{Q\in\mathcal{Q}_{\eps_i,\Lambda}^g}\int_{\partial Q}\lvert A_{i,\Lambda}^{s-1}-(A_{i,\Lambda}^{s-1})_Q\rvert^2\, d\L^5+C\,\ep_i^{6+1/2} \ \sum_{Q\in\mathcal{Q}_{\eps_i,\Lambda}^g}|F_{A_{i,\Lambda}^{s-1}}|^2_{M^0_{2,2}(Q_1(0))}\\[\sep]
    \nonumber
    &\quad+\, C\, \eps_i^6\ \sum_{Q\in\mathcal{Q}_{\eps_i,\Lambda}^g}|F_{A_{i,\Lambda}^{s-1}}|^4_{M^0_{2,2}(Q_1(0))}\\[\sep]
    &\le C\,\eps_i\sum_{Q\in\mathcal{Q}_{\eps_i,\Lambda}^g}\int_{\partial Q}\lvert A_{i,\Lambda}^{s-1}-(A_{i,\Lambda}^{s-1})_Q\rvert^2\, d\L^5+O(\eps_i)\ \longrightarrow 0
\end{align}
as $\eps_i\to 0^+$ for a fixed $\La$. Now, we estimate $J_{i,\Lambda,s}$. We write
\begin{align}\label{X-41}
    \nonumber
    J_{i,\Lambda,s}&\le 3\,\sum_{Q\in\mathcal{Q}_{\eps_i}^s\cap\mathcal{Q}_{\eps_i,\Lambda}^g}\int_{Q}\lvert A_{i,\Lambda}^{s-1}-A\rvert^2\ d\L^5+3\,\sum_{Q\in\mathcal{Q}_{\eps_i}^s\cap\mathcal{Q}_{\eps_i,\Lambda}^g}\int_{Q}\lf|A-(A)_Q\rg|^2\ d\L^5\\[\sep]
    &\quad+3\,\sum_{Q\in\mathcal{Q}_{\eps_i}^s\cap\mathcal{Q}_{\eps_i,\Lambda}^g}\int_{Q}\lvert(A_{i,\Lambda}^{s-1})_Q-(A)_Q\rvert^2\ d\L^5 \\[\sep]
    \nonumber
    &\le 6\,\|A_{i,\Lambda}^{s-1}-A\|^2_{L^2(Q_1(0))}+3\,\sum_{Q\in\mathcal{Q}_{\eps_i}^s}\int_{Q}\lf|A-(A)_Q\rg|^2\ d\L^5.
\end{align}
Using the induction hypothesis together with \eqref{app-grid} we obtain that
\be
\label{X-42}
\lim_{i\rightarrow+\infty}  J_{i,\Lambda,s}\le  C\lim_{i\to+\infty}\int_{\Omega^{s-1}_{i,\Lambda,A}}\lvert A\rvert^2\, d\L^5.
\ee
Finally, we bound $K_{i,\Lambda,s}$. Using respectively \eqref{co-ball} together with the induction hypothesis on $(Q_{\eps_i}^s, A_{i,\Lambda}^{s-1}  )$ \eqref{X-35a-b} and \eqref{X-355a-b}, we obtain
\begin{align}\label{X-43}
    \nonumber K_{i,\Lambda,s}&\le\sum_{Q\in\mathcal{Q}_{\eps_i,\Lambda}^{s,b}}\lf(\int_{Q}\lvert A_Q^s\rvert^2\, d\L^5+\int_{Q}\lvert A_{i,\Lambda}^{s-1}\rvert^2\, d\L^5\rg)\\[\sep]
    \nonumber 
    &\le C\sum_{Q\in\mathcal{Q}_{\eps_i,\Lambda}^{s,b}}\lf(\eps_i\int_{\partial Q}\lvert A_{i,\Lambda}^{s-1}\rvert^2\, d\H^4+  \ep_i^3\ \int_{\p Q}| F_{A_{i,\Lambda}^{s-1} } |\ d\H^4+\int_{Q}\lvert A_{i,\Lambda}^{s-1}\rvert^2\, d\L^5\rg)\\[\sep]
    \nonumber
    &\le C\sum_{Q\in\mathcal{Q}_{\eps_i,\Lambda}^{s,b}}\int_{Q_{2\eps_i}(c_Q)}\lvert A_{i,\Lambda}^{s-1}\rvert^2\, d\L^5+ C\, \ep_i^2\ \sum_{Q\in\mathcal{Q}_{\eps_i}^s\cap\mathcal{Q}_{\eps_i,\Lambda}^b}\ \int_{ Q}| F_{A_{i,\Lambda}^{s-1} } |\ d\L^5\\[\sep]
    \nonumber
    &\le C\sum_{Q\in\mathcal{Q}_{\eps_i,\Lambda}^{s,b}}\int_{ Q_{2\eps}(c_Q)}\lvert A_{i,\Lambda}^{s-1}\rvert^2\, d\L^5+C\,\eps_i^2\ \int_{Q_1^5(0)}|F_{A_{i,\Lambda}^{s-1}}|^2\ d\L^5\\[\sep]
    &\le C\sum_{Q\in\mathcal{Q}_{\eps_i,\Lambda}^{s,b}}\int_{Q_{2\eps}(c_Q)}\lvert A\rvert^2\, d\L^5+\|A_{i,\Lambda}^{s-1}-A\|_{L^2(Q_{1}^5(0))}^2 +o_{\eps_i}(1).
\end{align}
Hence, by Lemma \ref{Lemma: norm of A on bad cubes} and by inductive hypothesis (d) on $A_{i,\Lambda}^{s-1}$, property (d) for $A_{i,\Lambda}^{s}$ follows.

\medskip
\noindent
We are just left to show property (c) for $A_{i,\Lambda}^s$. Fix any $x\in Q_1^5(0)$ and $\rho>0$ such that $Q_{\rho}(x)$ is uniformly transversal to $\mathcal{Q}^l_{\eps_i}$ for $l\le s$. Notice that,
\begin{align}\label{29}
    \nonumber
    \|F_{A_{i,\Lambda}^s}\|_{L^2(\partial Q_{\rho}(x)\cap\Omega_i^s)}^2&=\|F_{A_{i,\Lambda}^s}\|_{L^2(\partial Q_{\rho}(x)\cap\Omega_{i}^{s-1})}^2+\|F_{A_{i,\Lambda}^s}\|_{L^2(\partial Q_{\rho}(x)\cap(\Omega_{i}^{s}\smallsetminus\Omega_i^{s-1}))}^2\\
    &=\|F_{A_{i,\Lambda}^{s-1}}\|_{L^2(\partial Q_{\rho}(x)\cap\Omega_{i}^{s-1})}^2+\|F_{A_{i,\Lambda}^s}\|_{L^2(\partial Q_{\rho}(x)\cap(\Omega_{i}^{s}\smallsetminus\Omega_i^{s-1}))}^2.
\end{align}
By inductive hypothesis (c) on $A_{i,\Lambda}^{s-1}$, there exists $N_{s-1}\in\n$ independent on $x$, $\rho$ and $i$ such that
\begin{align}\label{30}
    \int_{\partial Q_{\rho}(x)\cap\Omega_{i}^{s-1}}\lvert F_{A_{i,\Lambda}^{s-1}}\rvert^2\, d\H^4\le \frac{C_{s-1}}{\eps_i}\sum_{k=1}^{N_{s-1}}\int_{Q_{2\eps_i}(x_k)}f_{s-1}\, d\L^5,
\end{align}
for $C_{s-1}>0$ depending only on $G$, for an $N_{s-1}$-tuple $\{x_k^{s-1}=x_k^{s-1}(x,\rho,i)\}_{k=1,...,N_{s-1}}\subset Q_1^5(0)$ and for a real-valued, non-negative $f_{s-1}\in M^0_{1,4}(Q_1^5(0))$ independent on $i$ such that 
\begin{align*}
    |f_{s-1}|_{M^0_{1,4}(Q_1^5(0))}\le |F_A|_{M^0_{2,2}(Q_1^5(0))}^2.
\end{align*}
Hence, in order to prove property (c) for $A_{i,\Lambda}^s$, it remains to control $\|F_{A_{i,\Lambda}^s}\|_{L^2(\partial Q_{\rho}(0)\cap(\Omega_{i}^{s}\smallsetminus\Omega_i^{s-1}))}^2$. Notice that, since $\rho\in\big(\frac{3}{2}\eps_i,\frac{7}{4}\eps_i\big)$, there exists $\hat N$ depending only on the choice of the cubic cover such that $\partial Q_{\rho}$ intersects just $\hat N$ cubes in $\mathcal{Q}_{\eps_i}^s$, say $\{Q_1,...,Q_{\hat N}\}$. Then, by construction, we have
\begin{align*}
    F_{A_{i,\Lambda}^s}\mathds{1}_{\partial Q_{\rho}(x)\cap(\Omega_{i}^{s}\smallsetminus\Omega_i^{s-1})}=\sum_{Q\in\mathcal{Q}_{\eps_i}^s}F_{A_Q^s}\mathds{1}_{\partial Q_{\rho}(x)\cap Q}=\sum_{\ell=1}^{\hat N}F_{A_{Q_\ell}}\mathds{1}_{\partial Q_{\rho}(x)\cap Q_{\ell}}.
\end{align*}
Thus, by \eqref{equation: estimate on the trace of the curvature on cube} and \eqref{equation: estimate on the trace of the curvature on cube/2} we get
\begin{align}\label{XXIV}
    \nonumber
    \int_{\partial Q_{\rho}(x)\cap(\Omega_{i}^{s}\smallsetminus\Omega_i^{s-1})}&\lvert F_{A_{i,\Lambda}^{s}}\rvert^2\, d\H^4\le\sum_{\ell=1}^{\hat N}\int_{\partial Q_{\rho}(x)\cap Q_{\ell}}\lvert F_{A_{Q_{\ell}}^s}\rvert^2\, d\H^4\\
    \nonumber
    &\le C\,\sum_{\ell=1}^{\hat N}\int_{\partial Q_\ell}\lvert F_{A_{i,\Lambda}^{s-1}}\rvert^2\, d\H^4+C\,\eps_i^{-6}\sum_{Q_\ell\in\mathcal{Q}_{\eps_i,\Lambda}^g}\lf(\int_{\partial Q_\ell}\lvert A_{i,\Lambda}^{s-1}\rvert^2\, d\H^4\rg)^3\\
    &\quad+C\sum_{Q_\ell\in\mathcal{Q}_{\eps_i,\Lambda}^g}\eps_i^{-4}\int_{\partial Q_\ell}\lvert A_{i,\Lambda}^{s-1}-(A_{i,\Lambda}^{s-1})_{Q_\ell}\rvert^2\, d\H^4\ \ \int_{\partial Q_\ell}\lvert A_{i,\Lambda}^{s-1}\rvert^2\, d\H^4.
\end{align}
Fix any $\ell\in\{1,...,\hat N\}$. Notice that we have have chosen $Q^s_{\eps_i}$ so that \eqref{X-35a-b} holds. Then, we have
\begin{align}\label{XXV}
    \int_{\partial Q_\ell}\lvert F_{A_{i,\Lambda}^{s-1}}\rvert^2\, d\H^4\le\frac{C}{\eps_i}\int_{Q_{2\eps_i}(c_{Q_\ell})}\lvert F_A\rvert^2\, d\L^5,
\end{align}
for some constant $C>0$ independent of $\eps_i$. Moreover, by properties (2), (4)and (5) in Definition \ref{Definition: good cubes eps-cover} (characterizing good cubes), by Remark \ref{Remark: comparison boundary interior for the connection} and by \eqref{X-355a-b}, we obtain the estimates
\begin{align}\label{XXVI}
    \eps_i^{-6}\ \lf(\int_{\partial Q_\ell}\lvert A_{i,\Lambda}^{s-1}\rvert^2\, d\H^4\rg)^3&\le\eps_i^{-6}\ \lf(\eps_i^{3+1/2}\,\int_{ Q_1^5(0)}\lvert F_{A_{i,\Lambda}^{s-1}}\rvert^2\, d\L^5\rg)^3\\[\sep]
    \nonumber
    &\le \eps_i^{3+3/2}\, |F_{A_{i,\Lambda}^{s-1}}|_{M^0_{2,2}(Q_1^5(0))}^6\\[\sep]
    &\le C\, \frac{1}{\eps_i}\int_{Q_\ell}\, \eps_i^{1/2}\,|F_{A_{i,\Lambda}^{s-1}}|_{M^0_{2,2}(Q_1^5(0))}^6\ d\L^5
\end{align}
%
%\le \eps_i^{-6}(\Lambda\eps_i^4)^3=(\Lambda^3\eps_i^2)\eps_i^4\le\frac{1}{\eps_i}\int_{Q_{2\eps_i}(c_{Q_{\ell}})}[F_{A}]_{\mathcal{L}^{2,2}(Q_1^5(0))}^2\,dx^5,
%
and
\begin{align}\label{XXVII}
    \nonumber
    \eps_i^{-4}\int_{\partial Q_{\ell}}\lvert A_{i,\Lambda}^{s-1}-(A_{i,\Lambda}^{s-1})_{Q_\ell}\rvert^2\, d\H^4&\int_{\partial Q_\ell}\lvert A_{i,\Lambda}^{s-1}\rvert^2\, d\H^4\\
    \nonumber
    &\le\ \Lambda^{-1}\ |F_{A_{i,\Lambda}^{s-1}}|_{M_{2,2}^0(Q_1^5(0))}^2 \,\eps_i^4\,\La\ |F_{A_{i,\Lambda}^{s-1}}|_{M_{2,2}^0(Q_1^5(0))}^2 \\[\sep]
    &\le\frac{C_{s-1}}{\eps_i}\int_{Q_{2\eps}(c_{Q_\ell})}\, |F_A|_{M^0_{2,2}(Q_1^5(0))}^4\,d\L^5.
\end{align}
By combining \eqref{XXIV}, \eqref{XXV}, \eqref{XXVI} and \eqref{XXVII}, we get
\begin{align}\label{XXVIII}
    \int_{\partial Q_{\rho}(x)\cap(\Omega_{i}^{s}\smallsetminus\Omega_i^{s-1})}\lvert F_{A_{i,\Lambda}^{s}}\rvert^2\, d\H^4\le\frac{C}{\eps_i}\sum_{\ell=1}^{\hat N}\int_{Q_{2\eps_i}(c_{Q_\ell})}g_s\, d\L^5
\end{align}
with $g_s:=\lvert F_A\rvert^2+|F_A|^4_{M^0_{2,2}(Q_1^5(0))}+\eps_i^{1/2}\, |F_A|^6_{M^0_{2,2}(Q_1^5(0))}$. The required estimate \eqref{31} for $F_{A_{i,\Lambda}^s}$ then follows by \eqref{29}, \eqref{30} and \eqref{XXVIII}, letting 
\begin{align*}
    f_s:=\max\{g_s,f_{s-1}\}.
\end{align*}

\medskip
\noindent
Finally, let $A_{i,\Lambda}:=A_{i,\Lambda}^N\in L^2(Q_1^5(0))$. By construction, $F_{A_{i,\Lambda}}\in L^2(Q_1^5(0))$ and $A_{i,\Lambda}$ satisfies the properties (i) and (iii). Moreover, notice that $\Omega_i^N\supset Q_{1-\frac{\eps_i}{2}}(0)$. In particular, $\partial Q_{\rho}(x)\subset\Omega_{i}^N$. Hence, by property (c) for $A_{i,\Lambda}^N$, there exists $\tilde N\in\n$ independent on $x$, $\rho$ and $i$ such that
\begin{align}\label{32}
    \int_{\partial Q_{\rho}(x)}\lvert F_{A_{i,\Lambda}}\rvert^2\, d\H^4\le \frac{C}{\eps_i}\sum_{k=1}^{\tilde N}\int_{Q_{2\eps_i}(x_k)}f\, d\L^5,
\end{align}
for $C>0$ depending only on $G$, for an $N$-tuple of points $\{x_k=x_k(x,\rho,i)\}_{k=1,...,N}\subset Q_1^5(0)$ and for some real-valued, non-negative $f\in M^0_{1,4}(Q_1^5(0))$ independent on $i$ such that
\begin{align}
    |f|_{M^0_{1,4}(Q_1^5(0))}\le\ |F_A|_{M^0_{2,2}(Q_1^5(0))}^2
\end{align}
We are just left to check that, given any $x\in Q_1^5(0)$, for $\L^1$-a.e. choice of $\rho>0$ such that $Q_{\rho}(x)$ is uniformly transversal to $\mathcal{Q}_{\eps_i}$ and $\partial Q_{\rho}(x)\subset Q_{1-\frac{\eps_i}{2}}(0)$ we have $\iota_{\partial Q_{\rho}(x)}^*A_{i,\Lambda}\in\A_G(\partial Q_{\rho}(x))$. First, note that for $\L^1$-a.e. $\rho>0$ we have $\iota_{\partial Q_{\rho}(x)}^*A_{i,\Lambda}\in L^2(\partial Q_{\rho}(x))$. Moreover, by \eqref{32}, we get
\begin{align}\label{33}
    \nonumber
    \int_{\partial Q_{\rho}(x)}\lvert F_{\iota_{\partial Q_{\rho}(x)}^*A_{i,\Lambda}}\rvert^2\, d\H^4&=\int_{\partial Q_{\rho}(x)}\lvert \iota_{\partial Q_{\rho}(x)}^*F_{A_{i,\Lambda}}\rvert^2\, d\H^4\\
    &\le\int_{\partial Q_{\rho}(x)}\lvert F_{A_{i,\Lambda}}\rvert^2\, d\H^4\le C\eps_G,
\end{align}
for some constant $C>0$ depending only on $G$. By construction, we can find a covering $(R_l)_{\ell=1,...,N_{\rho,x}}$ of $\partial Q_{\rho}(x)$ where $R_\ell$ is a rectangle included in some $Q_l\in {\mathcal Q}_{\eps_i}^{s(l)}$ for some $s(l)\in\{0\cdots N \}$ depending on $l$ such that
\begin{align*}
    A_{i,\Lambda}=A_{Q_{\ell}}^{s(\ell)} \qquad \mbox{ on } R_{\ell}
\end{align*}
 By the properties of the extensions on good and bad cubes implying that we have respectively  \eqref{L4-norm-tiA} and \eqref{L4-tr}, if \eqref{33} holds there exist $g_{\ell}\in W^{1,2}(R_{\ell},G)$ such that
 \begin{align*}
    (\iota_{\partial Q_{\rho}(x)}^*A_{i,\Lambda})^{g_{\ell}}\in L^4(\partial Q_{\rho}(x)\cap R_{\ell}).
\end{align*}
Then, we can apply Lemma \ref{Lemma: weak connection} to conclude there exists $g\in W^{1,2}(\p Q_\rho(x),G)$ such that
 \be
 \label{X-4333}
 (\iota_{\partial Q_{\rho}(x)}^*A_{i,\Lambda})^{g}\in L^4(\partial Q_{\rho}(x)).
 \ee
Hence, $\iota_{\partial Q_{\rho}(x)}^*A_{i,\Lambda}\in\A_G(\partial Q_{\rho}(x))$. This concludes the proof of Theorem \ref{Theorem: approximation under controlled traces of the curvatures}.
\end{proof}
\section{The second smoothification: strong \texorpdfstring{$L^2$}{Z}-approximation by smooth connections}
The goal of this section is to prove that any weak connection with curvature having a small Morrey norm is strongly approximable in $L^2$ by smooth connections with small Morrey norm of the curvature as well. 
More precisely the main result obtained in this section is to prove the following theorem. 
\begin{Th}[Smooth approximation under controlled Morrey norm]\label{Theorem: smooth approximation under controlled Morrey norm}
Let $G$ be a compact matrix Lie group. There exists $\eps_G\in (0,1)$ such that for every $A\in\A_G(Q_1^5(0))$ satisfying 
\begin{align*}
    |F_A|_{M^0_{2,2}(Q_1^5(0))}^2<\eps_G
\end{align*}
there exists a sequence of $\g$-valued 1-forms $\{A_{i}\}_{i\in\n}\subset C_c^{\infty}(Q_1^5(0))$ such that:
\begin{enumerate}[(i)]
    \item for every $i\in\n$ we have 
    \begin{align*}
        |F_{A_{i}}|_{M^0_{2,2}(Q_{\frac{1}{2}}(0))}\le C_G\,|F_A|_{M^0_{2,2}(Q_1^5(0))}.
    \end{align*}
    for some constant $C_G>0$ depending on $G$;
    \item $\|A_i-A\big\|_{L^2(Q_{\frac{1}{2}}(0))}\to 0$ as $i\to +\infty$.
\end{enumerate}
\end{Th}
\noindent
In order to prove Theorem \ref{Theorem: smooth approximation under controlled Morrey norm} we shall need some preliminary results. Recall that we have defined \begin{align*}
    \mathscr{C}_{\eps}:=(\eps\z)^5\cap\overline{Q_{1-2\eps}(0)} 
\end{align*}
for every $\eps\in\left(0,\frac{1}{4}\right)$ (see the beginning of Section 3). Under the same assumption and notation of Theorem \ref{Theorem: approximation under controlled traces of the curvatures}, for every $i\in\n$ and for every $t\in Q_{\frac{\eps_i}{8}}(0)$ the cubes in the grid $\{Q_{\eps_i}(c+t)\}_{c\in\mathscr{C}_{\eps_i}}$ are uniformly transverse to the cubes in $\mathcal{Q}_{\eps_i}$ and such that $\partial Q_{\eps_i}(c+t)\subset Q_{1-\frac{\eps_i}{2}}(0)$ for every $c\in\mathscr{C}_{\eps_i}$. Since uniform transversality is stable by small perturbations, there exists $\alpha>0$  small enough and independent of $\eps_i$ so that for every ${r}\in\big((1-\alpha)\eps_i,(1+\alpha)\eps_i\big)$ and every $t\in Q_{\frac{{r}}{8}}(0)$ the cubes in the grid $\{Q_{{r}}(c+t)\}_{c\in\mathscr{C}_{{r}}}$ are still uniformly transverse to the cubes in $\mathcal{Q}_{\eps_i}$ and such that $\partial Q_{{r}}(c+t)\subset Q_{1-\frac{\eps_i}{2}}(0)$ for every $c\in\mathscr{C}_{{r}}$.
\begin{Lm}[Choice of an admissible cubic grid]\label{Lemma: choice of an admissible cubic decomposition}
Under the same assumption and notation of Theorem \ref{Theorem: approximation under controlled traces of the curvatures}, for every $i\in\n$ there exist ${r}_i\in\big((1-\alpha)\eps_i,(1+\alpha)\eps_i\big)$ and a translation $t_{i}\in Q_{\frac{{r}_i}{8}}(0)$ for which the following facts hold.
\begin{enumerate}
    \item $\iota_{\partial Q_{{r}_i}(c+t_i)}^*A_{i,\Lambda}\in\A_G(\partial Q_{{r}_i}(c+t_i))$ for every $c\in\mathscr{C}_{{r}_i}$. 
    \item There exist $N\in\n$ and a real-valued, non-negative $f\in M_{1,4}^0(Q_1^5(0))$ satisfying 
    \[
    |f|_{M^0_{1,4}(Q_1^5(0))}\le\,|F_A|_{M^0_{2,2}(Q_1^5(0))}^2
    \] such that for every $c\in\mathscr{C}_{{r}_i}$ we have
    \begin{align*}
        \int_{\partial Q_{{r}_i}(c+t_i)}\lvert F_{A_{i,\Lambda}}\rvert^2\, d\H^4\le \frac{C_G}{\eps_i}\sum_{k=1}^N\int_{Q_{2\eps_i}(x_k)}f\, d\L^5,
    \end{align*}
    for an $N$-tuple of points $\{x_k=x_k(i,x,{r}_i)\}_{k=1,...,N}\subset Q_1^5(0)$. 
    \item It holds that
    \begin{align}
        \label{Equation: convergence of A minus its the averages on a good decomposition}
        &\lim_{i\to+\infty}{r}_i\sum_{c\in\mathscr{C}_{{r}_i}}\int_{\partial Q_{{r}_i}(c+t_{i})}\lvert A_{i,\Lambda}-(A_{i,\Lambda})_{Q_{{r}_i}(c+t_{i})}\rvert^2\ d\H^{4}=0,\\
        &\lim_{i\to+\infty}{r}_i\sum_{c\in\mathscr{C}_{{r}_i}}\int_{\partial Q_{{r}_i}(c+t_{i})}\lvert F_{A_{i,\Lambda}}-(F_{A_{i,\Lambda}})_{Q_{{r}_i}(c+t_{i})}\rvert^2\ d\H^{4}=0.
    \end{align} 
\end{enumerate}
\end{Lm}
\begin{proof}[\textbf{\textup{Proof of Lemma \ref{Lemma: choice of an admissible cubic decomposition}}}]
By our assumption on $\alpha$ and by Theorem \ref{Theorem: approximation under controlled traces of the curvatures}-(ii), arguing exactly as in the proof of \cite[Lemma 2.1]{caniato-gaia} (with $p=2$ and $q=1$) or \cite[Lemma 3.1]{caniato} we can show that there exists a full measure subset $E\subset\big((1-\alpha)\eps_i,(1+\alpha)\eps_i\big)$ such that for every ${r}\in E$ we have that for a.e. translation $t\in Q_{\frac{{r}}{8}}(0)$ 1. and 2. in the statement hold. We fix ${r}_i\in E$ and again, exactly as in the proof of \cite[Lemma 2.1]{caniato-gaia} (with $p=2$ and $q=1$) or \cite[Lemma 3.1]{caniato} we have that 
\begin{align}
    \label{equation: decay Ieps}
    I_{{r}_i}&:=\int_{Q_{\frac{{r}_i}{8}}(0)}\sum_{c\in\mathscr{C}_{{r}_i}}\int_{\partial Q_{{r}_i}(c+t)}\lvert A_{i,\Lambda}-(A_{i,\Lambda})_{Q_{{r}_i}(c+t)}\rvert^2\, d\H^{4}\, d\L^5=o\big({r}_i^{4}\big)\\
    \label{equation: decay Jeps}
    J_{{r}_i}&:=\int_{Q_{\frac{{r}_i}{8}}(0)}\sum_{c\in\mathscr{C}_{{r}_i}}\int_{\partial Q_{{r}_i}(c+t)}\lvert F_{A_{i,\Lambda}}-(F_{A_{i,\Lambda}})_{Q_{{r}_i}(c+t)}\rvert^2\, d\H^{4}\, d\L^5=o\big({r}_i^{4}\big)
\end{align}
as ${r}_i\to 0^+$. Consider the sets
\begin{align*}
    T_{{r}_i,K,1}&:=\bigg\{t\in Q_{\frac{{r}_i}{8}}(0) \mbox{ s.t. } \varphi_1(t)\ge K\fint_{Q_{\frac{{r}_i}{8}}(0)}\varphi_1(t)\, d\L^5\bigg\}\\
    T_{{r}_i,K,2}&:=\bigg\{t\in Q_{\frac{{r}_i}{8}}(0) \mbox{ s.t. } \varphi_2(t)\ge K\fint_{Q_{\frac{{r}_i}{8}}(0)}\varphi_2(t)\, d\L^5\bigg\},
\end{align*}
with
\begin{align*}
    \varphi_1(t)&:=\sum_{c\in\mathscr{C}_{{r}_i}}\int_{\partial Q_{{r}_i}(c+t)}\lvert A_{i,\Lambda}-(A_{i,\Lambda})_{Q_{{r}_i}(c+t)}\rvert^2\ d\H^{4} \qquad\forall\, t\in Q_{\frac{{r}_i}{8}}(0),\\
    \varphi_2(t)&:=\sum_{c\in\mathscr{C}_{{r}_i}}\int_{\partial Q_{{r}_i}(c+t)}\lvert F_{A_{i,\Lambda}}-(F_{A_{i,\Lambda}})_{Q_{{r}_i}(c+t)}\rvert^2\ d\H^{4} \qquad\forall\, t\in Q_{\frac{{r}_i}{8}}(0).
\end{align*}
By integration on $T_{{r}_i,K,1}$ and $T_{{r}_i,K,2}$ we get
\begin{align*}
    \L^5(T_{{r}_i,K,1})&\le\frac{{r}_i^5}{K}\\
    \L^5(T_{{r}_i,K,2})&\le\frac{{r}_i^5}{K}.
\end{align*}
Moreover, for every $t\in Q_{\frac{{r}_i}{8}}(0)\smallsetminus (T_{{r}_i,K,1}\cup T_{{r}_i,K,2})$ it holds that
\begin{align}
    \label{equation: estimate with Ieps}
    {r}_i\sum_{c\in\mathscr{C}_{{r}_i}}\int_{\partial Q_{{r}_i}(c+t)}\lvert A_{i,\Lambda}-(A_{i,\Lambda})_{Q_{{r}_i}(c+t)}\rvert^2\ d\H^{4}&<{r}_i K\fint_{Q_{{r}_i}(0)}\varphi_1\, d\L^5=K\frac{I_{{r}_i}}{{r}_i^4},\\
    \label{equation: estimate with Jeps}
    {r}_i\sum_{c\in\mathscr{C}_{{r}_i}}\int_{\partial Q_{{r}_i}(c+t)}\lvert F_{A_{i,\Lambda}}-(F_{A_{i,\Lambda}})_{Q_{{r}_i}(c+t)}\rvert^2\ d\H^{4}&<{r}_i K\fint_{Q_{{r}_i}(0)}\varphi_2\, d\L^5=K\frac{J_{{r}_i}}{{r}_i^4}.
\end{align}
Now, we notice that
\begin{align*}
    \L^5(T_{{r}_i,4,1}\cup T_{{r}_i,4,2})\le\frac{{r}_i^5}{2},
\end{align*}
which implies
\begin{align*}
    \L^5(T_{{r}_i,A})\ge\frac{{r}_i^5}{2}>0.
\end{align*}
with $T_{{r}_i}:=T_{{r}_i,4,1}^c\cap T_{{r}_i,4,2}^c$. Hence, we fix $t_{i}\in T_{{r}_i}$. By \eqref{equation: decay Ieps}-\eqref{equation: decay Jeps} and \eqref{equation: estimate with Ieps}-\eqref{equation: estimate with Jeps}, we have
\begin{align*}
    \lim_{{r}_i\to 0^+}{r}_i\sum_{c\in\mathscr{C}_{{r}_i}}\int_{\partial Q_{{r}_i}(c+t_{{r}_i})}\lvert A_{i,\Lambda}-(A_{i,\Lambda})_{Q_{{r}_i}(c+t_{{r}_i})}\rvert^2\ d\H^{4}&=0,\\
    \lim_{{r}_i\to 0^+}{r}_i\sum_{c\in\mathscr{C}_{{r}_i}}\int_{\partial Q_{{r}_i}(c+t_{{r}_i})}\lvert F_{A_{i,\Lambda}}-(F_{A_{i,\Lambda}})_{Q_{{r}_i}(c+t_{{r}_i})}\rvert^2\ d\H^{4}&=0.
\end{align*}
The statement follows. This concludes the proof of Lemma \ref{Lemma: choice of an admissible cubic decomposition}.
\end{proof}
\begin{Dfi}\label{Definition: admissible cubic eps-decomposition}
Under the same notation that we have used in the previous Lemma \ref{Lemma: choice of an admissible cubic decomposition}, we say that the collection of cubes $\mathcal{Q}_{{r}_i}:=\{Q_{{r}_i}(c+t_i)\}_{c\in\mathscr{C}_{{r}_i}}$ is a \textit{admissible cubic ${r}_i$-grid} relative to $A_{i,\Lambda}$.
\end{Dfi}
\begin{Dfi}[Good and bad cubes]\label{Definition: good cubes eps-decomposition}
    Under the same notation that we have used in the previous Lemma \ref{Lemma: choice of an admissible cubic decomposition}, given an admissible cubic ${r}_i$-grid relative to $A_{i,\Lambda}$ we say that $Q\in\mathcal{Q}_{{r}_i}$ is a $\Lambda$-\textit{good} cube if all the following conditions hold:
    \begin{enumerate}[(1)]
        \item $\displaystyle{\frac{1}{{r}_i^3}\int_{\partial Q}\lvert F_{A_{i,\Lambda}}\rvert^2\, d\H^4\le{r}_i^{\frac{1}{2}}\int_{Q_1^5(0)}\lvert F_{A_{i,\Lambda}}\rvert^2\, d\L^5}$,
        \item $\displaystyle{\frac{1}{{r}_i^3}\int_{\partial Q}\lvert {A_{i,\Lambda}}\rvert^2\, d\H^4\le{r}_i^{\frac{1}{2}}\int_{Q_1^5(0)}\lvert F_{A_{i,\Lambda}}\rvert^2\, d\L^5}$,
        \item $\displaystyle{\int_{\partial Q}\lvert {A_{i,\Lambda}}-(A_{i,\Lambda})_Q\rvert^2\, d\H^4\le\frac{1}{{r}_i}\int_{Q}\lvert A_{i,\Lambda}\rvert^2\, d\L^5}$,
        \item $\displaystyle{\frac{1}{{r}_i^4}\int_{\partial Q}\lvert A_{i,\Lambda}-(A_{i,\Lambda})_Q\rvert^2\, d\H^4\le\Lambda^{-1}\int_{Q_1^5(0)}\lvert F_{A_{i,\Lambda}}\rvert^2\, d\L^5}$,
        \item $\displaystyle{\fint_Q\lvert A_{i,\Lambda}\rvert^2\, d\L^5\le\Lambda}$.
    \end{enumerate}
    Otherwise, we say that that $Q$ is a $\Lambda$-\textit{bad} cube. We denote by $\mathcal{Q}_{{r}_i,\Lambda}^g$ the set of all the $\Lambda$-good cubes and $\mathcal{Q}_{{r}_i,\Lambda}^b:=\mathcal{Q}_{{r}_i}\smallsetminus\mathcal{Q}_{{r}_i,\Lambda}^g$. 
\end{Dfi}
\begin{Lm}\label{Lemma: norm of A on bad cubes - dec}
    Under the same notation that we have used in the previous Lemma \ref{Lemma: choice of an admissible cubic decomposition}, let $\mathcal{Q}_{{r}_i}$ be a good ${r}_i$-grid relative to $A_{i,\Lambda}$. We have
     \begin{align*}
        \lim_{i\to+\infty}\sum_{Q\in\mathcal{Q}_{{r}_i,\Lambda}^b}\int_{Q}\lvert A_{i,\Lambda}\rvert^2\, d\L^5\le C\lim_{i\to+\infty}\int_{\Omega_{i,\Lambda,A_{i,\Lambda}}}\lvert A_{i,\Lambda}\rvert^2\, d\L^5,
    \end{align*}
     where $\Omega_{i,\Lambda,A_{i,\Lambda}}\subset Q_1^5(0)$ is given by
    \begin{align*}
        \Omega_{i,\Lambda,A_{i,\Lambda}}:=\bigcup\bigg\{Q\in\mathcal{Q}_{{r}_i} \mbox{ s.t. } \fint_{Q}\lvert A_{i,\Lambda}\rvert^2\,d\L^5>\Lambda\bigg\}.
    \end{align*}
\end{Lm}
\begin{proof}[\textbf{\textup{Proof of Lemma \ref{Lemma: norm of A on bad cubes - dec}}}]
    The proof is identical to the one of Lemma \ref{Lemma: norm of A on bad cubes}.
\end{proof}
\begin{Rm}\label{r-lim-bad}
    Notice that we have
    \begin{align*}
        r_i\int_{Q\in\mathcal{Q}_{r_i,\Lambda}^b}\int_{\partial Q}\lvert A_{i,\Lambda}\rvert^2\, d\H^4&\le r_i\sum_{Q\in\mathcal{Q}_{r_i,\Lambda}}\int_{\partial Q}\lvert A_{i,\Lambda}-(A_{i,\Lambda})_Q\rvert^2\, d\H^4\\
        &\quad+ r_i\sum_{Q\in\mathcal{Q}_{r_i,\Lambda}^b}\int_{\partial Q}\lvert (A_{i,\Lambda})_Q\rvert^2\, d\H^4\\
        &\le r_i\sum_{Q\in\mathcal{Q}_{r_i,\Lambda}}\int_{\partial Q}\lvert A_{i,\Lambda}-(A_{i,\Lambda})_Q\rvert^2\, d\H^4+\sum_{Q\in\mathcal{Q}_{r_i,\Lambda}^b}\int_{Q}\lvert A_{i,\Lambda}\rvert^2\, d\L^5.
    \end{align*}
    Thus, by \eqref{Equation: convergence of A minus its the averages on a good decomposition} and by Lemma \ref{Lemma: norm of A on bad cubes - dec} we have
    \begin{align*}
        \lim_{i\to+\infty}r_i\int_{Q\in\mathcal{Q}_{r_i,\Lambda}^b}\int_{\partial Q}\lvert A_{i,\Lambda}\rvert^2\, d\H^4\le C\lim_{i\to+\infty}\int_{\Omega_{i,\Lambda,A_{i,\Lambda}}}\lvert A_{i,\Lambda}\rvert^2\, d\L^5,
    \end{align*}
     where $\Omega_{i,\Lambda,A_{i,\Lambda}}\subset Q_1^5(0)$ is given as in Lemma \ref{Lemma: norm of A on bad cubes - dec}.
\end{Rm}
\noindent
We now have all the tools that are needed to prove the following strong $L^2$-approximation result by smooth connections under controlled Morrey norm of their curvatures.
\begin{proof}[\textbf{\textup{Proof of Theorem \ref{Theorem: smooth approximation under controlled Morrey norm}}}]
We assume that $\eps_G>0$ is small enough so that we can apply Theorem \ref{Theorem: approximation under controlled traces of the curvatures} to $A$. Hence, we find a sequence of admissible scales  $\{\eps_i\}_{i\in\n}$ such that $\eps_i\to 0^+$ as $i\to +\infty$ and a sequence $\{A_{i,\Lambda\}_{i\in\n}\subset L^2(Q_1^5(0))}$ satisfying (i), (ii) and (iii) in Theorem \ref{Theorem: approximation under controlled traces of the curvatures}. 

\medskip
\noindent
Let $\mathcal{Q}_{{r}_i}$ be an admissible cubic ${r}_i$-grid relative to $A_{i,\Lambda}$ in the sense of Definition \ref{Definition: admissible cubic eps-decomposition}. Recall that by definition ${r}_i\in\big((1-\alpha)\eps_i,(1+\alpha)\eps_i\big)$, for some fixed $\alpha\in\big(0,\frac{1}{2}\big)$ sufficiently small. We fix $i\in\n$ big enough so that
\begin{align*}
    \overline{Q_{\frac{1}{2}}(0)}\subset\bigcup_{Q\in\mathcal{Q}_{{r}_i}}Q.
\end{align*}

\medskip
\noindent
\textbf{Step 1: construction of the approximating $1$-forms.}

\medskip
\noindent
Let $N_{i}$ be the number of cubes in $\mathcal{Q}_{{r}_i}$. We first enumerate the family 
\begin{align*}
    \mathcal{Q}_{{r}_i}=\big\{Q_1,...,Q_{N_{i}}\big\}
\end{align*}
in such a way that for every for every $n\in\{1,...,N_{i}\}$ we have
\begin{align*}
    \Omega_{i}^n:=\operatorname{int}\bigg(\bigcup_{k=1}^n\overline{Q_k}\bigg).
\end{align*}
satisfies $\emptyset\neq H_i^n:=\partial\Omega_i^n\cap\partial Q_{n+1}\neq\partial Q_{n+1}$ and $\Omega_{i}^n$ is bi-Lipschitz equivalent to a $5$-dimensional ball, for every $n=1,...,N_{i}$.\footnote{Notice that, by construction, $H_n^i$ is always non-empty and bi-Lipschitz equivalent to a $4$-dimensional closed ball.} 

\noindent
If $\eps_G>0$ is small enough, by 1. and 2. in Lemma \ref{Lemma: choice of an admissible cubic decomposition} we have that $\iota_{\partial Q}^*A_{i,\Lambda}$ satisfies the assumptions of Corollary \ref{Corollary: extension in the interior of bad cubes} on $\partial Q$, for every $Q\in\mathcal{Q}_{{r}_i}$. Moreover, if $Q\in\mathcal{Q}_{{r}_i}$ is $\Lambda$-good then $\iota_{\partial Q}^*A_{i,\Lambda}$ satisfies the assumptions of Corollary \ref{Corollary: extension in the interior of good cubes} on $\partial Q$. For every $k\in\{1,...,N_i\}$ we denote by $A_{Q_k}\in L^5(Q_k)$ the extension given by applying Corollary \ref{Corollary: extension in the interior of good cubes}--(i) (if $Q_k$ is $\Lambda$-good) or Corollary \ref{Corollary: extension in the interior of bad cubes}--(i) (if $Q_k$ is $\Lambda$-bad) to $\iota_{\partial Q_k}^*A_{i,\Lambda}$.
Hence, for every $k\in\{1,...,N_{i}\}$ such that $Q_k$ is $\Lambda$-good there exists $g_k\in W^{1,2}(\partial Q_k,G)$ such that for every $4$-dimensional face $F$ of $\partial Q_k$ we have  
\begin{align}\label{60}
    \begin{cases}
        \|(\iota_{\partial Q_k}^* A_{i,\Lambda})^{g_k}\|_{W^{1,2}(F)}\le C\big(\|F_{A_{i,\Lambda}}\|_{L^2(\partial Q_k)}+{r}_i^{-1}\|A_{i,\Lambda}\|_{L^2(\partial Q_k)}\big)\\
        \iota_{F}^*A_{Q_k}=(\iota_{F}^*A_{i,\Lambda})^{g_{Q_k}}
    \end{cases}
\end{align}
and the estimates
\begin{align}
        \label{61}
        \|F_{A_{Q_k}}\|_{L^\frac{5}{2}(Q_k)}^2&\le C(\|F_{A_{i,\Lambda}}\|_{L^2(\partial Q_k)}^2+{r}_i^{-2}\|A_{i,\Lambda}\|_{L^2(\partial Q_k)}^2\|A_{i,\Lambda}-(A_{i,\Lambda})_{Q_k}\|_{L^2(\partial Q_k)}^2\\
        \nonumber
        &\quad+{r}_i^{-4}\|A_{i,\Lambda}\|_{L^2(\partial Q_k)}^6)\\
        \label{62}
        \|A_{Q_k}-\bar A\|_{L^2(Q_k)}^2&\le C\big({r}_i\|F_{A_{i,\Lambda}}\|_{L^2(\partial Q_k)}^2+{r}_i^{-1}\|A_{i,\Lambda}-(A_{i,\Lambda})_{Q_k}\|_{L^2(\partial Q_k)}^2+{r}_i^{-3}\|A_{i,\Lambda}\|_{L^2(\partial Q_k)}^2\big)
\end{align}
for every constant $\g$-valued 1-form $\bar A$ on $\r^5$, where $C>0$ is a constant depending only on $G$. On the other hand, for every $k\in\{1,...,N_{i}\}$ such that $Q_k$ is $\Lambda$-bad there exists $g_k\in W^{1,2}(\partial Q_k,G)$ such that for every $4$-dimensional face $F$ of $\partial Q_k$ we have
\begin{align}\label{63}
    \begin{cases}
        \|(\iota_{\partial Q_k}^*A_{i,\Lambda})^{g_k}\|_{W^{1,2}(F)}\le C\|F_{A_{i,\Lambda}}\|_{L^2(\partial Q_k)}\\
        \iota_{F}^*A_{Q_k}=(\iota_{F}^*A_{i,\Lambda})^{g_{Q_k}},
    \end{cases}
\end{align}
and the estimate
\begin{align}\label{64}
        \|F_{A_{Q_k}}\|_{L^\frac{5}{2}(Q_k)}^2&\le C\|F_{A_{i,\Lambda}}\|_{L^2(\partial Q_k)}^2,
\end{align}
where $C>0$ is again a constant depending only on $G$. Lastly, for every $k=1,...,N_i$, let $\tilde g_k\in W^{1,2}(Q_k,G)$ be the extension of $g_k$ given by Corollary \ref{Corollary: extension in the interior of good cubes}-(ii), if $Q_k$ is $\Lambda$-good, or Corollary \ref{Corollary: extension in the interior of bad cubes}-(ii), if $Q_k$ is $\Lambda$-bad.

\medskip
\noindent
In order to construct the approximating 1-forms, we proceed by induction on $n=1,...,N_i$. In particular, for every $n=1,...,N_i$ we build a $\g$-valued 1-form $\tilde A_{i,\Lambda}^n\in L^{5,\infty}(\Omega_i^{n})$ with $F_{\tilde A_{i,\Lambda}^n}\in L^{\frac{5}{2},\infty}(\Omega_i^{n})$ such that there exists a gauge transformation $\sigma_n\in W^{1,2}(\Omega_i^n,G)\cap W^{1,2}(\partial\Omega_i^n,G)$ satisfying
\begin{align*}
    \iota_{\partial\Omega_i^n}^*\tilde A_{i,\Lambda}^n=(\iota_{\partial\Omega_i^n}^*A_{i,\Lambda})^{\sigma_n}\in L^4(\partial\Omega_i^n)
\end{align*}

\medskip
\noindent
\textit{Base of the induction}. At the initial step $n=1$, we have $\Omega_i^1=Q_1$ and we set $\tilde A_{i,\Lambda}^0:=A_{Q_1}\in L^5(\Omega_i^1)$. By \eqref{61} (if $Q_1$ is $\Lambda$-good) or \eqref{64} (if $Q_1$ is $\Lambda$-bad) we get $F_{\tilde A_{i,\Lambda}^1}\in L^{\frac{5}{2}}(\Omega_i^1)$. Let
\begin{align*}
    \sigma_1:=\tilde g_1\in W^{1,2}(\Omega_i^1,G)\cap W^{1,2}(\partial\Omega_i^1,G),
\end{align*} 
we have
\begin{align*}
    \iota_{\partial\Omega_i^1}^*\tilde A_{i,\Lambda}^1=\iota_{\partial\Omega_i^1}^*A_{Q_1}=(\iota_{\partial\Omega_i^1}^*A_{i,\Lambda})^{g_1}=(\iota_{\partial\Omega_i^1}^*A_{i,\Lambda})^{\sigma_1}.
\end{align*}
Hence, we have proved our claim for $n=1$. 

\medskip
\noindent
\textit{Induction step}. Let $\eta_n:=g_{n}^{-1}\sigma_{n-1}\in W^{1,4}(H_i^{n-1},G)$ and notice that, by the inductive assumption, we have
\begin{align}
    \big((\iota_{H_i^{n-1}}^*A_{i,\Lambda})^{g_n}\big)^{\eta_n}=(\iota_{H_i^{n-1}}^* A_{i,\Lambda})^{\sigma_{n-1}}=\iota_{H_i^{n-1}}^*\tilde A_{i,\Lambda}^{n-1}\in L^4(H_i^{n-1}).
\end{align}
Hence, we have 
\begin{align}
    d\eta_n=\eta_n\big(\iota_{H_i^{n-1}}^* A_{i,\Lambda}\big)^{g_n}-\big(\iota_{H_i^{n-1}}^*\tilde A_{i,\Lambda}^{n-1}\big)\eta_n\in L^4(H_i^{n-1}).
\end{align}
Thus, by Corollary \ref{c-gluing}, if $\eps_G>0$ is small enough we can find an extension $h_{n}\in W^{1,(5,\infty)}(Q_{n},G)$ of $\eta_n$. Let $\tilde A_{i,\Lambda}^{n}\in L^{5,\infty}(\Omega_i^{n})$ and $\sigma_n\in W^{1,2}(\Omega_i^n,G)\cap W^{1,2}(\partial\Omega_i^n,G)$ be given by
\begin{align}
    \tilde A_{i,\Lambda}^{n}:=\begin{cases}
        \tilde A_{i,\Lambda}^{n-1} & \mbox{ on } \Omega_i^{n-1}\\
        (A_{Q_{n}})^{h_{n}} & \mbox{ on } Q_{n}
    \end{cases}
\end{align}
and
\begin{align}
    \sigma_n:=\begin{cases}
        \sigma_{n-1} & \mbox{ on } \Omega_i^{n-1}\\
        \tilde g_nh_n & \mbox{ on } Q_{n}.
    \end{cases}
\end{align}
By construction and by Lemma \ref{lm-cont}, we have $F_{\tilde A_{i,\Lambda}^{n}}\in L^{\frac{5}{2},\infty}(\Omega_i^{n})$. Moreover, since $\partial\Omega_i^n\smallsetminus\partial\Omega_i^{n-1}\subset\partial Q_n$ and we have 
\begin{align*}
    \iota_{H_i^n}^*\tilde A_{i,\Lambda}^n=\iota_{H_i^n}^*(A_{Q_n}^{h_n})=(\iota_{H_i^n}^*A_{i,\Lambda})^{g_nh_n}=(\iota_{H_i^n}^*A_{i,\Lambda})^{\sigma_n},
\end{align*}
we conclude that
\begin{align*}
    \iota_{\partial\Omega_i^n}^*\tilde A_{i,\Lambda}^n=(\iota_{\partial\Omega_i^n}^*A_{i,\Lambda})^{\sigma_n}\in L^4(\partial\Omega_i^n).
\end{align*}
This concludes the proof of the induction step and of our claim at once.

\medskip
\noindent
Now, to lighten the notation let $\Omega_i:=\Omega_i^{N_i}$ and define
\begin{align*}
    \tilde A_{i,\Lambda}:=\tilde A_{i,\Lambda}^{N_i}\in L^{5,\infty}(\Omega_i).
\end{align*}
\begin{align*}
    \sigma_{i,\Lambda}:=\sigma_{N_i}\in W^{1,2}(\Omega_i,G).
\end{align*}
\medskip
\noindent
\textbf{Step 2: Morrey norm control on the curvatures.}

\medskip
\noindent
Note that, by construction, we have $\overline{Q_{\frac{1}{2}}(0)}\subset\Omega_i$. Fix any point $x\in Q_{\frac{1}{2}}(0)$ and let ${r}\in\big(0,\dist(x,\partial Q_{\frac{1}{2}}(0))\big)$. Assume that ${r}\le\eps_i$. Then, $Q_{{r}}(x)$ intersects at most $\tilde N$ cubes in $\mathcal{Q}_{{r}_i}^s$, say $Q_1,...,Q_{\tilde N}$, with $\tilde N\in\n$ depending only on the choice of the cubic cover. Hence, by H\"older inequality, by the estimates \eqref{equation: properties of the harmonic extension good cubes}, \eqref{equation: properties of the harmonic extension bad cubes} for each $A_Q$, by the property (1) of $\Lambda$-good cubes and by the property (i) of $A_{i,\Lambda}$ given by Theorem \ref{Theorem: approximation under controlled traces of the curvatures}, we get
\begin{align*}
    \frac{1}{{r}}\int_{Q_{{r}}(x)}\lvert F_{\tilde A_{i,\Lambda}}\rvert^2\, d\L^5&\le\|F_{\tilde A_{i,\Lambda}}\|_{L^{\frac{5}{2}}(Q_{{r}}(x))}^2\\
    &\le C\sum_{\ell=1}^{\tilde N}\lVert F_{\tilde A_{i,\Lambda}}\rVert_{L^{\frac{5}{2}}(Q_\ell)}^2=C\sum_{i=1}^{\tilde N}\lVert F_{A_{Q_\ell}}\rVert_{L^{\frac{5}{2}}(Q_\ell)}^2\\
    &\le C\sum_{Q_\ell\in\mathcal{Q}_{{r}_i,\Lambda}^g}\big(\lVert F_{A_{i,\Lambda}}\rVert_{L^{2}(\partial Q_\ell)}^2+\eps_i^{-2}\lVert A_{i,\Lambda}\rVert_{L^{2}(\partial Q_\ell)}^2\big)\\
    &\quad+\sum_{Q_\ell\in\mathcal{Q}_{{r}_i,\Lambda}^b}\lVert F_{A_{i,\Lambda}}\rVert_{L^{2}(\partial Q_\ell)}^2\\
    &<C\lvert F_{A_{i,\Lambda}}\rvert_{M_{2,2}^0(Q_1^5(0))}^2\le C\lvert F_A\rvert_{M_{2,2}^0(Q_1^5(0))}^2,
\end{align*}
for some constant $C>0$ depending only on $G$.

\noindent
Assume now that $\,\eps_i<{r}<1$. Then there exists a universal constant $k>1$ such that $Q_{{r}}(x)$ can be covered with a finite number of cubes $\{Q_\ell\}_{\ell=1,...,N_{i,x,{r}}}$ in $\mathcal{Q}_{{r}_i}^s$ such that 
\begin{align*}
    \bigcup_{\ell=1}^{N_{i,x,{r}}}Q_{2\eps_i}(c_{Q_\ell})\subset Q_{k{r}}(x).
\end{align*}
Notice that now the number of cubes $N_{i,x,{r}}$ may depend on $i$, $x$ and ${r}$. Then, by the estimates \eqref{equation: estimate for the curvature on the good cubes} and \eqref{equation: estimate for the curvature on the bad cubes} we obtain
\begin{align*}
    \int_{Q_{{r}}(x)}&\lvert F_{\tilde A_{i,\Lambda}}\rvert^2\, d\L^5\le\sum_{\ell=1}^{N_{i,x,{r}}}\int_{Q_\ell}\lvert F_{\tilde A_{i,\Lambda}}\rvert^2\, d\L^5=\sum_{\ell=1}^{N_{i,x,{r}}}\int_{Q_\ell}\lvert F_{A_{Q_\ell}}\rvert^2\, d\L^5\\
    &\le C\sum_{Q_\ell\in\mathcal{Q}_{{r}_i,\Lambda}^g}\bigg(\eps_i^{-4}\int_{\partial Q_\ell}\lvert A_{i,\Lambda}- (A_{i,\Lambda})_{Q_\ell}\rvert^2\, d\H^4\bigg)\bigg(\eps_i\int_{\partial Q_\ell}\lvert A_{i,\Lambda}\rvert^2\, d\H^4\bigg)\\
    &\quad+\eps_i\bigg(\eps_i^{-2}\int_{\partial Q_\ell}\lvert A_{i,\Lambda}\rvert^2\, d\H^4\bigg)^3\bigg)+C\eps_i\sum_{\ell=1}^{N_{i,x,{r}}}\int_{\partial Q_\ell}\lvert F_{A_{i,\Lambda}}\rvert^2\, d\H^4,
\end{align*}
for some constant $C>0$ depending only on $G$. By property (4) in Definition \ref{Definition: good cubes eps-decomposition} and by the property (i) of $A_{i,\Lambda}$ given by Theorem \ref{Theorem: approximation under controlled traces of the curvatures} we get
\begin{align*}
    \sum_{Q_\ell\in\mathcal{Q}_{{r}_i,\Lambda}^g}&\bigg(\bigg(\eps_i^{-4}\int_{\partial Q_\ell}\lvert A_{i,\Lambda}- (A_{i,\Lambda})_{Q_\ell}\rvert^2\, d\H^4\bigg)\bigg(\eps_i\int_{\partial Q_\ell}\lvert A_{i,\Lambda}\rvert^2\, d\H^4\bigg)\\
    &\le C\Lambda^{-1}[F_{A_{i,\Lambda}}]_{M_{2,2}^0(Q_1^5(0))}^2\sum_{Q_\ell\in\mathcal{Q}_{{r}_i,\Lambda}^g}\int_{Q_{2\eps_i}(c_{Q_\ell})}\lvert A_{i,\Lambda}\rvert^2\, d\L^5,\\
    &\le C\Lambda^{-1}\lvert F_{A}\rvert_{M_{2,2}^0(Q_1^5(0))}^2\sum_{Q_\ell\in\mathcal{Q}_{\eps_i,\Lambda}^g}\int_{Q_{2\eps_i}(c_{Q_\ell})}\lvert A_{i,\Lambda}\rvert^2\, d\L^5,
\end{align*}
for some constant $C>0$ depending only on $G$. By property (5) in Definition \ref{Definition: good cubes eps-decomposition} and by our choice of $k>0$, we get
\begin{align*}
    \sum_{Q_\ell\in\mathcal{Q}_{{r}_i,\Lambda}^g}\int_{Q_{2\eps_i}(c_{Q_\ell})}\lvert A_{i,\Lambda}\rvert^2\, d\L^5\le C\Lambda\sum_{Q_\ell\in\mathcal{Q}_{{r}_i,\Lambda}^g}(2\eps_i)^5\le C\Lambda(k{r})^5,
\end{align*}
for some constant $C>0$ depending only on $G$. Hence, since by assumption ${r}\in (0,1)$, we have obtained the estimate
\begin{align*}
    \sum_{Q_\ell\in\mathcal{Q}_{{r}_i,\Lambda}^g}\bigg(\bigg(\eps_i^{-4}\int_{\partial Q_\ell}\lvert A_{i,\Lambda}- (A_{i,\Lambda})_{Q_\ell}\rvert^2\, d\H^4\bigg)\bigg(\eps_i\int_{\partial Q_\ell}\lvert A_{i,\Lambda}\rvert^2\, d\H^4\bigg)\le Ck^5{r}\lvert F_A\rvert_{M_{2,2}^0(Q_1^5(0))}^2,
\end{align*}
for some constant $C>0$ depending only on $G$. By property (5) in Definition \ref{Definition: good cubes eps-decomposition}, provided $i\ge i_0$ is large enough so that $\eps_i<\Lambda^{-3}\lvert F_A\rvert_{M_{2,2}^0(Q_1^5(0))}^2$ we obtain
\begin{align*}
    \sum_{Q_\ell\in\mathcal{Q}_{\eps_i,\Lambda}^g}\eps_i\bigg(\eps_i^{-2}\int_{\partial Q_\ell}\lvert A_{i,\Lambda}\rvert^2\, d\H^4\bigg)^3&\le\sum_{Q_\ell\in\mathcal{Q}_{{r}_i,\Lambda}^g}\eps_i\bigg(\eps_i^{-3}\int_{Q_{2\eps_i}(c_{Q_\ell})}\lvert A_{i,\Lambda}\rvert^2\, d\L^5\bigg)^3\\
    &\le\sum_{Q_\ell\in\mathcal{Q}_{\eps_i,\Lambda}^g}\eps_i(\eps_i^{-3}\Lambda\eps_i^5)^3\le\eps_i^2\Lambda^3\sum_{Q_\ell\in\mathcal{Q}_{\eps_i,\Lambda}^g}\eps_i^5\\
    &\le C\eps_i^2\Lambda^3\le C{r}\lvert F_A\rvert_{M_{2,2}^0(Q_1^5(0))}^2,
\end{align*}
for some constant $C>0$ depending only on $G$. By Lemma \ref{Lemma: choice of an admissible cubic decomposition}--(ii) and by the property (i) of $A_{i,\Lambda}$ given by Theorem \ref{Theorem: approximation under controlled traces of the curvatures} we get 
\begin{align*}
    \eps_i\sum_{\ell=1}^{N_{i,x{r}}}\int_{\partial Q_\ell}\lvert F_{A_{i,\Lambda}}\rvert^2\, d\H^4&\le C\sum_{\ell=1}^{N_{i,x,{r}}}\int_{Q_{2\eps_i}(c_{Q_\ell})}\lvert F_{A_{i,\Lambda}}\rvert^2\, d\L^5\\
    &\le C\int_{Q_{k{r}}(x)}\lvert F_{A_{i,\Lambda}}\rvert^2\, d\L^5\le Ck{r}\lvert F_{A_{i,\Lambda}}\rvert_{M_{2,2}^0(Q_1^5(0))}^2\\
    &\le Ck{r}\lvert F_{A}\rvert_{M_{2,2}^0(Q_1^5(0))}^2,
\end{align*}
for some constant $C>0$ depending only on the choice of the cubic cover. Hence, provided $i\ge i_0$ i sufficiently large we get
\begin{align}\label{66}
    \frac{1}{{r}}\int_{Q_{{r}}(x)}\lvert F_{\tilde A_{i,\Lambda}}\rvert^2\, d\L^5\le C\lvert F_A\rvert_{M_{2,2}^0(Q_1^5(0))}^2,
\end{align}
for some constant $C>0$ depending only on $G$.

\bigskip
\noindent
\textbf{Step 3: strong $L^2$-convergence of the connections.}

\medskip
\noindent
Notice that, by construction, we have
\begin{align*}
    \Big\|\tilde A_{i,\Lambda}^{\sigma_{i,\Lambda}^{-1}}-A\Big\|_{L^2(\Omega_{i})}^2&=\sum_{n=1}^{N_i}\Big\|A_{Q_n}^{\tilde g_n^{-1}}- A\Big\|_{L^2(Q_n)}^2\\
    &\le C\bigg(\sum_{n=1}^{N_i}\Big\|A_{Q_n}^{\tilde g_n^{-1}}- A_{i,\Lambda}\Big\|_{L^2(Q_n)}^2+\sum_{n=1}^{N_i}\big\|A_{i,\Lambda}- A\big\|_{L^2(Q_n)}^2\bigg)\\
    &\le C\bigg(\sum_{n=1}^{N_i}\Big\|A_{Q_n}^{\tilde g_n^{-1}}- A_{i,\Lambda}\Big\|_{L^2(Q_n)}^2+\big\|A_{i,\Lambda}-A\big\|_{L^2(Q_1(0))}^2\bigg)
\end{align*}
for some constant $C>0$ depending on $G$ and $A$. By the same procedure that we have used in the proof of Theorem \ref{Theorem: approximation under controlled traces of the curvatures}, we get
\begin{align*}
    &\lim_{i\to+\infty}\sum_{n=1}^{N_i}\Big\|A_{Q_n}^{\tilde g_n^{-1}}- A_{i,\Lambda}\Big\|_{L^2(Q_n)}\le C\lim_{i\to+\infty} \int_{\Omega_{i,\Lambda,A_{i,\Lambda}}}\lvert A_{i,\Lambda}\rvert^2\, d\L^5,
\end{align*}
where $C>0$ is a universal constant depending only on $G$ and $\Omega_{i,\Lambda,A_{i,\Lambda}}\subset Q_1^5(0)$ is given as in Lemma \ref{Lemma: norm of A on bad cubes - dec}.
Now we notice that 
\begin{align*}
    \L^5(\Omega_{i,\Lambda,A_{i,\Lambda}})&={r}_i^5\operatorname{card}\bigg(\bigg\{Q\in\mathcal{Q}_{{r}_i} \mbox{ s.t. } \fint_{Q}\lvert A_{i,\Lambda}\rvert^2\, dx^5>\Lambda\bigg\}\bigg)\\
    &\le{r}_i^5\operatorname{card}\bigg(\bigg\{Q\in\mathcal{Q}_{{r}_i} \mbox{ s.t. } \fint_{Q}\lvert A-A_{i,\Lambda}\rvert^2\, dx^5+\fint_{Q}\lvert A\rvert^2\, d\L^5>\Lambda\bigg\}\bigg)\le\frac{C}{\Lambda}
\end{align*}
and, recalling the definition of $\Omega_{i,\Lambda,A}$ from Lemma \ref{Lemma: norm of A on bad cubes}, we have
\begin{align*}
    \L^5(\Omega_{i,\Lambda,A})\le\frac{C}{\Lambda}
\end{align*}
for every given $i$, where $C>0$ is independent on $i$. Thus, by Theorem \ref{Theorem: approximation under controlled traces of the curvatures}-(iii), we get that for every $j\in\n$ there exists $i(j)>0$ big enough so that be letting $\tilde A_j:=\tilde A_{i(j),j}$ and $\sigma_j:=\sigma_{i(j),j}^{-1}\in W^{1,2}(\Omega_{i(j)})$ we have
\begin{align*}
    \Big\|\tilde A_j^{\sigma_j^{-1}}-A_{i(j),j}\Big\|_{L^2(\Omega_{i(j)})}^2&\le C\int_{\Omega_{i(j),j,A_{i(j),j}}}\lvert A_{i(j),j}\rvert^2\, d\L^5\\
    &\le C\bigg(\int_{\Omega_{i(j),j,A_{i(j),j}}}\lvert A\rvert^2\, d\L^5+\int_{Q_1^5(0)}\lvert A_{i(j),j}-A\rvert^2\, d\L^5\bigg)\\
    &\le C\int_{\Omega_{i(j),j, A_{i(j),j}}}\lvert A\rvert^2\, d\L^5+\int_{\Omega_{i(j),j,A}}\lvert A\rvert^2\, d\L^5,
\end{align*}
where $C>0$ doesn't depend on $j$. Finally, we get 
\begin{align*}
    \Big\|\tilde A_j^{\sigma_j^{-1}}-A\Big\|_{L^2(\Omega_{i(j)})}^2&\le C\int_{\Omega_{i(j),j, A_{i(j),j}}}\lvert A\rvert^2\, d\L^5\to 0
\end{align*}
as $j\to +\infty$. 

\medskip\noindent
Notice that, by assumption, $\Omega_{i(j)}\supset\overline{Q_{\frac{1}{2}}(0)}$ for every $j\in\n$. Since $\pi_2(G)=0$, we can find a sequence $\{\tilde\sigma_j\}_{j\in\n}\subset C^{\infty}\big(\overline{Q_{\frac{1}{2}}(0)}\big)$ such that 
\begin{align*}
    \|\tilde\sigma_j-\sigma_j\|_{W^{1,2}(Q_{\frac{1}{2}}(0))}\rightarrow 0
\end{align*}
and $\tilde\sigma_j-\sigma_j\to 0$ pointwise $\L^5$-a.e. on $Q_{\frac{1}{2}}(0)$ as $j\to+\infty$. Moreover, since we have the improved integrability properties $A\in L^{5,\infty}(Q_{\frac{1}{2}}(0))$ and $F_A\in L^{\frac{5}{2},\infty}(Q_{\frac{1}{2}}(0))$, by standard convolution with a smooth mollifying kernel we can find a sequence $\{\hat A_j\}_{j\in\n}\in C_c^{\infty}\big(Q_{\frac{1}{2}}(0)\big)$ such that
\begin{align*}
    |F_{\hat A_j}|_{M_{2,2}^0(Q_{\frac{1}{2}}(0))}^2\le 2|F_{\tilde A_j}|_{M_{2,2}^0(Q_{\frac{1}{2}}(0))}^2\le 2C\lvert F_A\rvert_{M_{2,2}^0(Q_1^5(0))}^2 \qquad\forall\,j\in\N,
\end{align*}
\begin{align*}
    \|\hat A_j-\tilde A_j\|_{L^2(Q_{\frac{1}{2}}(0))}\to 0
\end{align*}
and $\hat A_j-\tilde A_j\to 0$ pointwise $\L^5$-a.e. on $Q_{\frac{1}{2}}(0)$ as $j\to+\infty$. 
Now define 
\begin{align*}
    A_j:=\hat A_j^{\tilde\sigma_j^{-1}}\in C^{\infty}\big(\overline{Q_{\frac{1}{2}}(0)}\big) \qquad\forall\,j\in\n.
\end{align*}
By construction and $\ad$-invariance of the norm on $\g$, we have 
\begin{align*}
    |F_{\tilde A_j}|_{M_{2,2}^0(Q_{\frac{1}{2}}(0))}^2\le 2|F_{\hat A_j}|_{M_{2,2}^0(Q_{\frac{1}{2}}(0))}^2\le 2C\lvert F_A\rvert_{M_{2,2}^0(Q_1^5(0))}^2 \qquad\forall\,j\in\N.
\end{align*}
By construction and dominated convergence, we have 
\begin{align*}
    \|A_j-A\|_{L^2(Q_{\frac{1}{2}}(0))}\le\big\|A_j-\tilde A_j^{\sigma_j^{-1}}\big\|_{L^2(Q_{\frac{1}{2}}(0))}+\big\|\tilde A_j^{\sigma_j^{-1}}-A\big\|_{L^2(Q_{\frac{1}{2}}(0))}\to 0
\end{align*}
as $j\to+\infty$. This concludes the proof of Theorem \ref{Theorem: smooth approximation under controlled Morrey norm}.
\end{proof}
\section{Partial regularity for stationary weak Yang--Mills fields}
In this section we prove our main result about the partial regularity of stationary weak Yang--Mills fields on the unit cube $Q_1^5(0)$. In order to get that, we first use Theorem \ref{Theorem: smooth approximation under controlled Morrey norm} and the main result in \cite{meyer-riviere} to obtain the following Coulomb gauge extraction theorem for weak connections whose curvature has sufficiently small Morrey norm.
\begin{Th}[Coulomb gauge extraction under small Morrey norm assumption]\label{Theorem: Coulomb gauge extraction}
    Let $G$ be a compact matrix Lie group. There exists $\eps_G\in (0,1)$ such that for every weak connection $A\in\A_G(Q_1^5(0))$ satisfying 
    \begin{align*}
        \lvert F_A\rvert_{M_{2,2}^0(Q_1^5(0))}^2<\eps_G
    \end{align*}
    there exist $g\in W^{1,2}(Q_{\frac{1}{2}}(0),G)$ such that $A^g\in (M_{2,2}^1\cap M_{4,1}^0)(Q_{\frac{1}{2}}(0))$ satisfies
    \begin{align*}
        \begin{cases}
            d^*A^g=0,\\
            \iota_{\partial Q_1^5(0)}^*(*A^g)=0,\\
            \lvert\nabla A^g\rvert_{M_{2,2}^0(Q_{\frac{1}{2}}(0))}+\lvert A^g\rvert_{M_{4,1}^0(Q_{\frac{1}{2}}(0))}\le C_G\lvert F_{A}\rvert_{M_{2,2}^0(Q_{\frac{1}{2}}(0))},
        \end{cases}
    \end{align*}
    for some constant $C_G>0$ depending only on $G$.
\end{Th}
\begin{proof}[\textbf{\textup{Proof of Theorem \ref{Theorem: Coulomb gauge extraction}}}]
    By Theorem \ref{Theorem: smooth approximation under controlled Morrey norm}, if $\eps_G>0$ is small enough we can find a sequence of $\g$-valued $1$-forms $\{A_{i}\}_{i\in\n}\subset C_c^{\infty}(Q_{\frac{1}{2}}(0))$ such that 
    \begin{enumerate}
    \item for every $i\in\n$ we have 
    \begin{align*}
        \lvert F_{A_{i}}\rvert_{M_{2,2}^0(Q_{\frac{1}{2}}(0))}^2\le K_G\lvert F_A\rvert_{M_{2,2}^0(Q_1^5(0))}^2.
    \end{align*}
    for some constant $K_G>0$ depending on $G$;
    \item $\|A_i-A\big\|_{L^2(Q_{\frac{1}{2}}(0))}\to 0$ as $i\to +\infty$.
    \end{enumerate}
    By choosing $\eps_G>0$ possibly smaller, we can make sure that $\sqrt{K_G\eps_G}$ is smaller than the constant given by \cite[Theorem 1.3]{meyer-riviere} and we can apply such statement to conclude that for every $i\in\n$ there exists a gauge $g_i$ such that $d^*A_{i}^{g_{i}}=0$ on $Q_{\frac{1}{2}}(0)$, $\iota_{\partial Q_1^5(0)}^*(*A^g)=0$ and 
    \begin{align}\label{equation: Coulomb gauge}
        \lvert\nabla A_{i}^{g_{i}}\rvert_{M_{2,2}^0(Q_{\frac{1}{2}}(0))}+\lvert A_{i}^{g_{i}}\rvert_{M_{4,1}^0(Q_{\frac{1}{2}}(0))}\le C_G\lvert F_{A_{i}}\rvert_{M_{2,2}^0(Q_{\frac{1}{2}}(0))},
    \end{align}
    for some constant $C_G>0$ depending only on $G$. By standard Sobolev embedding theorems, there exists a subsequence (not relabeled) such that $A_{i}^{g_{i}}\rightharpoonup\tilde A\in M_{4,1}^0$ weakly in $L^4$ and $\nabla A_{i}^{g_{i}}\rightharpoonup\nabla\tilde A\in M_{2,2}^0$ weakly in $L^2$. Clearly, we have $d^*\tilde A=0$. Notice that 
    \begin{align*}
        dg_{i}=g_{i}A_{i}^{g_{i}}-A_{i}g_{i}.
    \end{align*}
    By \eqref{equation: Coulomb gauge} and 1, we get that $\{g_{i}\}_{i\in\n}$ is uniformly bounded in $W^{1,2}(Q_{\frac{1}{2}}(0),G)$. Hence, there exists a subsequence (not relabeled) such that $g_{i}\rightharpoonup g$ weakly in $W^{1,2}$. By Sobolev embedding theorems, we have $g_{i}\to g$ strongly in $L^2$. Thus, since $A_{i}\to A$ strongly in $L^2$ and $G$ is bounded we conclude that $A^g=\tilde A$ on $Q_{\frac{1}{2}}(0)$. This concludes the proof of Theorem \ref{Theorem: Coulomb gauge extraction}. 
\end{proof}
\begin{Dfi}[Weak Yang--Mills connections]
    Let $G$ be a compact matrix Lie group. We say that $A\in\A_G(Q_1^5(0))$ is \textit{weak Yang--Mills connection} on $Q_1^5(0)$ if 
    \begin{align*}
        d_A^*F_A=0 \qquad\mbox{ distrubutionally on } Q_1^5(0),
    \end{align*}
    i.e.
    \begin{align*}
        \int_{Q_1^5(0)}F_A\cdot d_A\varphi=0 \quad\mbox{ for every } \varphi\in C_c^{\infty}(\wedge^1Q_1^5(0)\otimes\g).
    \end{align*}
\end{Dfi}
\begin{Dfi}[Stationary weak Yang--Mills connections]
    Let $G$ be a compact matrix Lie group. We say that a weak Yang--Mills connection $A\in\A_G(Q_1^5(0))$ on $Q_1^5(0)$ is \textit{stationary} if 
    \begin{align}\label{stationarity condition}
        \left.\frac{d}{dt}\right|_{t=0}\text{YM}(\Phi_t^*A)=0,
    \end{align}
    for every smooth $1$-parameter group of diffeomorphisms $\Phi_t$ of $Q_1^5(0)$ with compact support. 
\end{Dfi}
\begin{Rm}\label{Remark: gauge equivalence preserve stationary weak Ynag-Mills property}
    Note that if $A$ is a stationary weak Yang--Mills connection and $\tilde A\in W^{1,2}\cap L^4$ is another $\g$-valued $1$-form on $Q_1^5(0)$ such that $\tilde A=A^g$ for some $g\in W^{1,2}(Q_1^5(0),G)$, then $\tilde A$ is a stationary weak Yang--Mills connection as well. The proof of such fact follows by direct computation exploiting the gauge invariance of the Yang--Mills functional.
\end{Rm}
\noindent
If $A$ is a stationary Yang--Mills connection, it can be shown that the following \textit{monotonicity property} holds true: for every given $x\in Q_1^5(0)$ the function
\begin{align}\label{monotonicity}
    \big(0,\dist(x,\partial Q_1^5(0)\big)\ni\rho\to\frac{e^{c\Lambda\rho}}{\rho}\int_{Q_{\rho}(x)}\lvert F_A\rvert^2\, d\L^5
\end{align}
is non-decreasing, where $c>0$ is a universal constant and $\Lambda$ depends on $Q_1(x)$. In particular, we have 
\begin{align*}
    \lvert F_A\rvert_{M_{2,2}^0(Q_1^5(0))}\le C\|F_A\|_{L^2(Q_1^5(0))},
\end{align*}
for some constant $C>0$ independent on $A$. 
Thanks to \eqref{monotonicity} and to Theorem \ref{Theorem: Coulomb gauge extraction}, we can derive the following $\eps$-regularity statement by using the same arguments presented in \cite[Section 4]{meyer-riviere}. This is a crucial initial step in building a regularity theory for $\text{YM}$-energy minimizers, following the same path that R. Schoen and K. Uhlenbeck walked in their work on energy-minimizing harmonic maps in general supercritical dimension $n>2$ (as documented in \cite{schoen-uhlenbeck}).
\begin{Th}[$\eps$-regularity]\label{Theorem: eps-regularity}
    Let $G$ be a compact matrix Lie group. There exists $\eps_G\in (0,1)$ such that for every stationary weak Yang--Mills $A\in\A_G(Q_1^5(0))$ satisfying 
    \begin{align*}
        \operatorname{YM}(A)=\int_{Q_1^5(0)}\lvert F_A\rvert^2\, d\L^5<\eps_G
    \end{align*}
    there exist $g\in W^{1,2}(Q_{\frac{1}{2}}(0),G)$ such that $A^g\in C^{\infty}(Q_{\frac{1}{2}}(0))$.
\end{Th}
\noindent
Standard covering arguments (see e.g. \cite[Proposition 9.21]{giaquinta-martinazzi}) together with Theorem \ref{Theorem: eps-regularity} and \eqref{monotonicity} give the following bound on the singular set of stationary weak Yang--Mills connections, which is the main result of the present paper. 
\begin{Th}\label{Theorem: regularity for stationary weak Yang-Mills connections}
     Let $G$ be a compact matrix Lie group and let $A\in\A_G(Q_1^5(0))$ be a stationary weak Yang--Mills connection on $Q_1^5(0)$. Then
     \begin{align*}
         \H^1(\operatorname{Sing}(A))=0,
     \end{align*}
     where $\H^1$ is the $1$-dimensional Hausdorff measure on $Q_1^5(0)$ and $\operatorname{Sing}(A)\subset Q_1^5(0)$ is the singular set of $A$, given by $\operatorname{Sing}(A):=Q_1^5(0)\smallsetminus\operatorname{Reg}(A)$ where
     \begin{align*}
         \operatorname{Reg}(A):=\{x\in Q_1^5(0) \mbox{ s.t. } \exists \rho>0, g\in W^{1,2}(B_{\rho}(x),G) \mbox{ s.t. } A^g\in C^{\infty}(B_{\rho}(x))\}.
     \end{align*}
   \end{Th}
\appendix
\section{Adapted Coulomb gauge extraction statements in critical dimension}
\begin{Propa}\label{appendix: global Coulomb gauge extraction on a sphere}
Let $G$ be any compact matrix Lie group. There exist constants $\eps_G,C_G>0$ depending only on $G$ such that for every $\g$-valued $1$-form $A\in W^{1,2}(\s^4)$ on $\s^4$ such that
\begin{align*}
    \int_{\s^4}\lvert F_A\rvert^2\, d\H^4<\eps_G
\end{align*}
we can find a gauge $g\in W^{1,4}(\s^4,G)$ satisfying $A^g\in W^{1,2}(\s^4)$, $d^*A^g=0$ and
\begin{align*}
    \|A^g\|_{W^{1,2}(\s^4)}\le C_G\|F_A\|_{L^2(\s^4)}
\end{align*}
\end{Propa}
\begin{proof}[\textbf{\textup{Proof of Proposition \ref{appendix: global Coulomb gauge extraction on a sphere}}}]
Let $\pi$ be the stereographic projection from $\s^4\smallsetminus\{N\}$ into $\r^4$ which is sending the north pole $N$ of $\s^4$ to infinity and which is conformally invariant. Introduce
\begin{align*}
    D:=(\pi^{-1})^\ast A
\end{align*}    
Because of conformal invariance of the Hodge operator on 2-forms one has
\begin{align*}
    \int_{\r^4} |F_{D}|^2 d\L^4&=-\int_{\r^4}\operatorname{tr}(F_{D}\wedge *F_D)=-\int_{\r^4}(\pi^{-1})^*(\operatorname{tr}(F_A\wedge *F_A)\\
    &=-\int_{\s^4}\operatorname{tr}(F_A\wedge *F_A)=\int_{\s^4}|F_A|^2 d\H^4<\eps_0.
\end{align*}
For every $R>0$, on $B_{R}(0)$ one chooses the Uhlenbeck Coulomb gauge $D_R:=(D)^{g_R}$ (see \cite{uhlenbeck-connections-with-lp}) that satisfies
\begin{align*}
    d^* D_R&=0
\end{align*}
and 
\begin{align*}
    \|D_R\|_{L^4(B_R(0))}+\|d D_R\|_{L^2(B_R(0))}&\le C_G\ \|F_{D_R}\|_{L^2(B_R(0))}\le C_G\ \|F_A\|_{L^2(\s^4)}.
\end{align*}
What is crucial here is that all the norms involved are scaling invariant and then the constant $C_G>0$ is independent of $R$.

Consider now on $\pi^{-1}(B_R(0))=\s^4\smallsetminus B_{\rho_R}(N)$ where $\rho_R\simeq R/(1+R^2) \rightarrow 0$ as $R\rightarrow +\infty$ the 1-form
\begin{align*}
    \pi^\ast(D_R)=A^{h_R}
\end{align*}
where $h_R:= g_R\circ\pi$. Using the conformal invariance of the $L^4$ norms on 1-forms as well as the $L^2$ norms on $2$-forms one has
\begin{align}\label{equation: appendix L^4 control on the gauged connecitons}
    \|A^{h_R}\|_{L^4(\pi^{-1}(B_R(0))}+\|d(A^{h_R})\|_{L^2(\pi^{-1}(B_R(0))}\le C_G\ \|F_A\|_{L^2(\s^4)}
\end{align}
There exists obviously a sequence $R_k\to+\infty$ such that for some $\Omega\in L_{loc}^4(\s^4\smallsetminus\{N\})$ we have
\begin{align*}
    A^{h_{R_k}}\rightharpoonup \Omega \mbox{ weakly in } L^4_{loc}(\s^4\smallsetminus\{N\})\quad 
\end{align*}
and
\begin{align*}
    h_{R_k}\rightharpoonup h_\infty \mbox{ weakly in } W^{1,4}_{loc}(\s^4\smallsetminus\{N\}).
\end{align*}
Moreover
\begin{align*}
    d(A^{h_{R_k}})\rightharpoonup d\Omega \mbox{ weakly in } L^2_{loc}(\s^4\smallsetminus\{N\})
\end{align*}
By \eqref{equation: appendix L^4 control on the gauged connecitons} we have in particular that 
\begin{align}\label{equation: appendix L^4 bound}
    \|\Omega\|_{L^4(\s^4)}\le C_G\|F_A\|_{L^2(\s^4)}\quad\mbox{ and }\quad \Omega=A^{h_\infty}\in L^4(\s^4)
\end{align}
Following Uhlenbeck continuity type argument introduced in \cite{riviere} and adapted in \cite{lamm-riviere} to the 4-dimensional case we construct $g\in W^{1,4}(\s^4,G)$ such that
\begin{align*}
    d^*(g^{-1} dg+g^{-1}\Omega g)=0\quad\mbox{ and }\quad\|dg\|_{L^4(\s^4)}\le C_G\, \|\Omega\|_{L^4(\s^4)}.
\end{align*}
Hence we have the existence of $g\in W^{1,4}(\s^4,G)$ such that $(A^{h_\infty})^g=\Omega^g$ satisfies
\begin{align*}
    d^*((A^{h_\infty})^g)=0\quad\mbox{ and } \|(A^{h_\infty})^g\|_{L^4(\s^4)} \le C_G \|F_A\|_{L^2(\s^4)}
\end{align*}
which gives
\begin{align*}
    \|(A^{h_\infty})^g\|_{W^{1,2}(\s^4)} \le C_G\|F_A\|_{L^2(\s^4)}
\end{align*}
This proves the existence of a controlled global Coulomb gauge $\tilde g:=h_{\infty}g\in W^{1,4}(\s^4,G)$ under small curvature assumption. This concludes the proof of Proposition \ref{appendix: global Coulomb gauge extraction on a sphere}.
\end{proof}

\medskip
\noindent
% %
We recall here the main statement that we will use to extract Coulomb gauges in critical dimension. This is an adaptation of the main Theorem in \cite{uhlenbeck-connections-with-lp}.
\begin{Propa}\label{appendix: Coulomb gauge extraction for  forms}
Let $G$ be a compact matrix Lie group. There exist constants $\eps_G,C_G>0$ depending only on $G$ such that for every $\g$-valued $1$-form $A\in L^4(\b^4)$ on $\b^4$ such that
\begin{align*}
    \|F_A\|_{L^2(\s^4)}<\eps_G
\end{align*}
we can find a gauge $g\in W^{1,4}(\b^4,G)$ satisfying $A^g\in W^{1,2}(\b^4)$, $d^*A^g=0$ and
\begin{align*}
    \|A^g\|_{W^{1,2}(\b^4)}\le C_G\|F_A\|_{L^2(\b^4)}
\end{align*}
\end{Propa}
\begin{proof}[\textbf{\textup{Proof of Proposition \ref{appendix: Coulomb gauge extraction for  forms}}}]
    Fix any $4<p<8$. For every $\eps,C>0$, we define 
    \begin{align*}
        X&:=\{A\in L^p(\b^4,\opwedge^1\b^4\otimes\g) \mbox{ s.t. } dA\in L^{\frac{p}{2}}(\b^4,\opwedge^2\b^4\otimes\g)\}\\
        \mathscr{U}^{\eps}&:=\{A\in X \mbox{ with } \|F_A\|_{L^2(\b^4)}<\eps\}\\
        \mathscr{V}_{C}^{\eps}&:=\{A\in\mathscr{U}^{\eps} \mbox{ : } \exists\,g\in W_{\operatorname{id}_G}^{1,p}(\b^4,G) \mbox{ : } d^*A^g=0, \|A^g\|_{W^{1,\frac{q}{2}}(\b^4)}\le C\|F_A\|_{L^\frac{q}{2}(\b^4)} \mbox{, } q=4,p\}.
    \end{align*}
    On $X$, we consider the topology induced by the following norm:
    \begin{align*}
        \|A\|_X:=\|A\|_{L^p(\b^4)}+\|dA\|_{L^{\frac{p}{2}}(\b^4)}, \qquad\forall\,A\in X.
    \end{align*}
    Notice that $X$ is a Banach vector space with respect to such norm.

    \medskip
    \noindent 
    \textit{Claim}. We claim that there exist $\eps,C>0$ such that $\mathscr{V}_{C}^{\eps}=\mathscr{U}^{\eps}$. In order to achieve such result, we prove separately the following facts.
    \begin{itemize}
        \item   \textit{$\mathscr{U}^{\eps}$ is path-connected}. Notice that $0\in\mathscr{U}^{\eps}$. Given any $A\in\mathscr{U}^{\eps}$, we define
                \begin{align*}
                    A(t):=t A(t\,\cdot\,), \qquad\forall\,t\in[0,1].
                \end{align*}
                Notice that $A(0)=0$ and $A(1)=A$. Moreover, 
                \begin{align*}
                    F_{A(t)}=dA(t)+A(t)\wedge A(t)=t^2dA(t\,\cdot\,)+t^2A(t\,\cdot\,)\wedge A(t\,\cdot\,)=t^2F_A(t\,\cdot\,)
                \end{align*}
                which implies that
                \begin{align}\label{equation: appendix gauge extraction path connectedness}
                    \nonumber
                    \int_{\b^4}\lvert F_{A(t)}\rvert^\frac{p}{2}\, d\L^4&=\int_{\b^4}t^p\lvert F_A(t\,\cdot\,)\rvert^{\frac{p}{2}}\, d\L^4=t^{p-4}\int_{t\b^4}\lvert F_A\vert^{\frac{p}{2}}\,d\L^4\\
                    &\le t^{p-4}\int_{\b^4}\lvert F_A\vert^{\frac{p}{2}}\,d\L^4<+\infty
                \end{align}
                and
                \begin{align*}
                    \int_{\b^4}\lvert F_{A(t)}\rvert^2\, d\L^4=\int_{\b^4}t^4\lvert F_A(t\,\cdot\,)\rvert^2\, d\L^4=\int_{t\b^4}\lvert F_A\vert^2\,d\L^4\le\int_{\b^4}\lvert F_A\vert^2\,d\L^4<\eps.
                \end{align*}
                Hence, $A(t)\in\mathscr{U}^{\eps}$ for every $t\in[0,1]$. At the same time we have
                \begin{align*}
                    \int_{\b^4}\lvert A(t)\rvert^p\, d\L^4=\int_{\b^4}t^p\lvert A(t\,\cdot\,)\rvert^4\, d\L^4\le t^{p-4}\int_{\b^4}\lvert A\rvert^4\, d\L^4\to 0
                \end{align*}
                as $t\to 0^+$. Hence, $A(t)\to 0$ in $L^p(\b^4)$. As a byproduct, we have $A(t)\wedge A(t)\to 0$ in $L^{\frac{p}{2}}(\b^4)$. Moreover, by \eqref{equation: appendix gauge extraction path connectedness} we have $F_{A(t)}\to 0$ in $L^{\frac{p}{2}}(\b^4)$. This implies $dA(t)\to 0$ in $L^{\frac{p}{2}}(\b^4)$.
                
                It follows that $[0,1]\ni t\to A(t)$ is a continuous path in $\mathscr{U}^{\eps}$ joining $A$ and $0$.
        \item   \textit{For every $\eps,C>0$, $\mathscr{V}_C^{\eps}$ is closed in $\mathscr{U}^{\eps}$}. Let $\{A_k\}_{k\in\n}\subset\mathscr{V}_C^{\eps}$ be such                    that $A_k\to A\in\mathscr{U}_{\eps}$. We want to show that $A\in\mathscr{V}_C^{\eps}$. 
        
                Notice that, by assumption, we have $F_{A_k}\to F_A$ strongly in $L^\frac{p}{2}(\b^4)$. Moreover, for every $k\in\n$ there exists $g_k\in W_{\operatorname{id}_G}^{1,p}(\b^4,G)$ such that $d^*A_k^{g_k}=0$ and 
                $$\|A_k^{g_k}\|_{W^{1,\frac{q}{2}}(\b^4)}\le C\|F_{A_k}\|_{L^2(\b^4)}\le C\eps$$
                for $q=p,4$. Since $\{A_k^{g_k}\}_{k\in\n}$ is bounded in $W^{1,\frac{p}{2}}$, there exists a subsequence (not relabeled) such that $A_k^{g_k}\rightharpoonup A_{\infty}$ in $W^{1,\frac{p}{2}}$. Hence, for $q=p,4$, we have 
                \begin{align*}
                    \|A_{\infty}\|_{W^{1,\frac{q}{2}}(\b^4)}\le\liminf_{k\to+\infty}\|A_{k}^{g_k}\|_{W^{1,\frac{q}{2}}(\b^4)}\le C\lim_{k\to+\infty}\|F_{A_k}\|_{L^\frac{q}{2}(\b^4)}=C\|F_A\|_{L^\frac{q}{2}(\b^4)}.
                \end{align*}
                Notice that $d^*A_{\infty}=0$ and that, by Sobolev embedding theorem, we have $A_k^{g_k}\to A_{\infty}$ strongly in $L^p$. 
                Since 
                \begin{align}\label{equation: appendix useful}
                    dg_k=g_kA_k^{g_k}-A_kg_k,
                \end{align}
                we get that $\{g_k\}_{k\in\n}$ is bounded in $W_{\operatorname{id}_G}^{1,p}(\b^4,G)$. Thus, there exists a subsequence (not relabeled) such that $g_k\rightharpoonup g$ in $W_{\operatorname{id}_G}^{1,p}(\b^4,G)$. This is enough to pass to the limit in \eqref{equation: appendix useful} and we get $A_{\infty}=A^g$. Since $g$ is a gauge satisfying all the required properties for $A$, we conclude that $A\in\mathscr{V}_C^{\eps}$.
        \item   \textit{For some choice of $\eps,C>0$, $\mathscr{V}_C^{\eps}$ is open in $\mathscr{U}^{\eps}$}. Let $A\in\mathscr{V}_{C}^{\eps}$. It is clear that if                 we find an open neighborhood of its Coulomb gauge $A^g$ for the topology of $X$ in $\mathscr{V}_C^{\eps}$, then $A$ posses also such a neighborhood.                  So we can assume right away that $d^*A=0$ and 
                \begin{align*}
                    \|A\|_{W^{1,\frac{q}{2}}(\b^4)}\le C\|F_A\|_{L^{\frac{q}{2}}(\b^4)}<C\eps.
                \end{align*}
                Notice that, in such a way, automatically we have $A\in W^{1,\frac{p}{2}}(\b^4)$.

                \begin{align}
                    Y&:=\bigg\{U\in W^{1,p}(\mathbb{B}^4,\g) \mbox{ : } \int_{\mathbb{B}^4}U\, d\L^4=0\bigg\}.\\
                    Z&:=\bigg\{(f,\alpha)\in W^{-1,p}(\mathbb{B}^4,\g)\times W^{1-\frac{1}{p},p}(\partial\mathbb{B}^4,\wedge^{3}T^*\partial\mathbb{B}^4\otimes\g) \mbox{ : } \int_{\mathbb{B}^n}f\, d\L^4=-\int_{\partial\mathbb{B}^4}\alpha\bigg\}.
                \end{align}
                Notice that $Y$ and $Z$ are Banach spaces, as they are closed subspaces of $W^{1,p}(\mathbb{B}^4,\g)$ and of the product $W^{-1,p}(\mathbb{B}^n,\g)\times W^{-\frac{1}{p},p}(\partial\mathbb{B}^4,\wedge^{3}T^*\partial\mathbb{B}^4\otimes\g)$ respectively. We introduce the map
                \begin{align*}
                    \mathscr{F}_A:X\times Y\to Z
                \end{align*}
                given by
                \begin{align*}
                    \mathscr{F}_A(\omega,U):=\Big(d^*\big((A+\omega)^{\operatorname{exp}(U)}\big),\iota_{\partial\mathbb{B}^n}^*\big(\hspace{-0.7mm}*\hspace{-0.5mm}(A+\omega)^{\exp(U)}\big)\Big)
                \end{align*}
                By direct computation we can show that $\mathscr{F}_A$ is a $C^1$-map between Banach spaces. The partial derivative $\partial_U\mathscr{F}_A(0,0):Y\to Z$ is the following linear and continuous operator between Banach spaces:
                \begin{align*}
                    \partial_U\mathscr{F}_A(0,0)[V]&=\big(\Delta V+d^*(AV-VA),\partial_rV\big)\\
                    &=\big(\Delta V+d^*([A,V]),\partial_rV\big), \qquad\forall\, V\in Y. 
                \end{align*}
                In order to apply the implicit function theorem to $\mathscr{F}_A$ we need to show that $\partial_U\mathscr{F}_A(0,0)$ is invertible. First we show that it is injective. Indeed, by standard $L^p$-theory for the Laplacian and since $A$ is in a controlled Coulomb gauge, for every $V\in Y$ we get
                \begin{align*}
                    \|V\|_{Y}&\le \tilde C\Big(\|\Delta V\|_{W^{-1,p}(\b^4)}+\|\partial_rV\|_{W^{-\frac{1}{p},p}(\partial\b^4)}\Big)\\
                    &\le\tilde C\Big(\|\partial_U\mathscr{F}_A(0,0)[V]\|_{W^{-1,p}(\b^4)}+\|d^*([A,V])\|_{W^{-1,p}(\b^4)}\Big)\\
                    &\le\tilde C\Big(\|\partial_U\mathscr{F}_A(0,0)[V]\|_{W^{-1,p}(\b^4)}+\|[A,V]\|_{L^p(\b^4)}\Big)\\
                    &\le\tilde C\Big(\|\partial_U\mathscr{F}_A(0,0)[V]\|_{W^{-1,p}(\b^4)}+\|A\|_{L^{p}(\b^4)}\|V\|_{L^{\infty}(\b^4)}\Big)\\
                    &\le\tilde C\Big(\|\partial_U\mathscr{F}_A(0,0)[V]\|_{L^{\frac{p}{2}}(\b^4)}+C\eps\|V\|_{W^{1,p}(\b^4)}\Big),
                \end{align*}
                for some constant $\tilde C>0$ which just depends on $G$. By choosing $\eps,C>0$ in such a way that $C\eps<\frac{1}{2\tilde C}$, we obtain
                \begin{align*}
                    \|V\|_{Y}&\le 2\tilde C\|\partial_U\mathscr{F}_A(0,0)[V]\|_{W^{-1,p}(\b^4)}.
                \end{align*}
                which implies that $\partial_U\mathscr{F}_A(0,0)$ has trivial kernel. Classical Calderon-Zygmund theory asserts that the operator $\mathcal{L}_0:Y\to Z$ given by 
                \begin{align*}
                    \mathcal{L}_0(V):=\Delta V, \qquad\forall\, V\in Y
                \end{align*}
                is invertible and therefore it has zero index. For every $t\in[0,1]$ we define the operator $\mathcal{L}_t:Y\to Z$ by
                \begin{align*}
                    \mathcal{L}_t(V):=\Delta V-t(*[*A,dV]), \qquad\forall\, V\in Y.
                \end{align*}
                and we notice that $[0,1]\ni t\to\mathcal{L}_t$ is a continuous path of bounded operators joining $\mathcal{L}_0$ and $\mathcal{L}_1=\partial_U\mathscr{F}_A(0,0)$. By continuity of the Fredholm index, we conclude that $\partial_U\mathscr{F}_A(0,0)$ has zero index. This, together with $\ker\big(\partial_U\mathscr{F}_A(0,0)\big)=0$, implies that $\partial_U\mathscr{F}_A(0,0)$ is invertible. Hence, we can apply the implicit function theorem in order to get that there exist an open neighbourhood $\mathcal{O}$ of $0$ in $Y$ and $\delta>0$ such that for every $\omega\in X$ such that $\|\omega\|_{X}<\delta$ there exists $U_{\omega}\in \mathcal{O}$ such that 
                \begin{align*}
                    0=\mathscr{F}_A(\omega,U_{\omega})=d^*\big((A+\omega)^{g_{\omega}}\big),
                \end{align*}
                where we have set $g_{\omega}:=\exp(U_{\omega})\in W^{1,p}(\b^4,G)$. We just need to establish the bounds 
                \begin{align*}
                    \|(A+\omega)^{g_{\omega}}\|_{L^{\frac{q}{2}}(\b^4)}\le C\|F_{A+\omega}\|_{L^{\frac{q}{2}}(\b^4)}
                \end{align*}
                for $q=4,p$. Notice that, by triangular inequality and Sobolev embedding theorem, we get
                \begin{align*}
                     \|(A+\omega)^{g_{\omega}}\|_{W^{1,\frac{q}{2}}(\b^4)}&= \|(A+\omega)^{g_{\omega}}\|_{L^{\frac{q}{2}}(\b^4)}+\|d((A+\omega)^{g_{\omega}})\|_{L^{\frac{q}{2}}(\b^4)}\\
                     &\le\hat C\big(\|F_{A+\omega}\|_{L^{\frac{q}{2}}(\b^4)}+\|(A+\omega)^{g_{\omega}}\|_{L^{q}(\b^4)}^2\big)\\
                     &\le\hat C\Big(\|F_{A+\omega}\|_{L^{\frac{q}{2}}(\b^4)}+\|(A+\omega)^{g_{\omega}}\|_{L^{q}(\b^4)}\|(A+\omega)^{g_{\omega}}\|_{W^{1,\frac{q}{2}}(\b^4)}\Big).
                \end{align*}
                for some constant $\hat C>0$ depending only on $G$. Now we see that 
                \begin{align*}
                    \|(A+\omega)^{g_{\omega}}\|_{L^{q}(\b^4)}&\le K\|A\|_{L^{q}(\b^4)}+K\big(\|\omega\|_{L^{q}(\b^4)}+\|dg_{\omega}\|_{L^{q}(\b^4)}\big)\\
                    &\le K\|A\|_{W^{1,\frac{q}{2}}(\b^4)}+K\big(\|\omega\|_{L^{q}(\b^4)}+\|dg_{\omega}\|_{L^{q}(\b^4)}\big),
                \end{align*}
                where again $K>0$ depends only on $G$. Notice that by possibly reducing $\delta>0$ and the size of the neighbourhood $\mathcal{O}$ we bring the quantity $\|\omega\|_{L^{q}(\b^4)}+\|dg_{\omega}\|_{L^{q}(\b^4)}$ to be arbitrary small. Then, we choose such parameters possibly depending on $A$ in such a way that
                \begin{align*}
                    \|\omega\|_{L^{q}(\b^4)}+\|dg_{\omega}\|_{L^{q}(\b^4)}\le\|A\|_{W^{1,\frac{q}{2}}(\b^4)}
                \end{align*}
                which implies
                \begin{align*}
                    \|(A+\omega)^{g_{\omega}}\|_{L^{q}(\b^4)}&\le 2K\|A\|_{W^{1,\frac{q}{2}}(\b^4)}\le 2KC\eps
                \end{align*}
                and finally
                \begin{align*}
                     \|(A+\omega)^{g_{\omega}}\|_{W^{1,\frac{q}{2}}(\b^4)}&\le\hat C\|F_{A+\omega}\|_{L^{\frac{q}{2}}(\b^4)}+2\hat CKC\eps\|(A+\omega)^{g_{\omega}}\|_{W^{1,\frac{q}{2}}(\b^4)}\Big).
                \end{align*}
                Assuming that $\eps,C>0$ are such that $2\hat CKC\eps\le\frac{1}{2}$ and $C\ge2\hat C$ we get 
                \begin{align*}
                     \|(A+\omega)^{g_{\omega}}\|_{W^{1,\frac{q}{2}}(\b^4)}&\le C\|F_{A+\omega}\|_{L^{\frac{q}{2}}(\b^4)}
                \end{align*}
                and it follows that
                \begin{align*}
                    B_{\delta}^X(A):=\big\{A+\omega \mbox{ : } \|\omega\|_{X}<\delta\big\}\subset\mathscr{V}_C^{\eps}.
                \end{align*}
                This concludes the proof of the openness of $\mathscr{V}_{C}^{\eps}$ in $\mathcal{U}_{\eps}$ and our claim follows. 
    \end{itemize}
    Now that the previous claim is proved, we proceed by approximation in the following way. Let $A$ satisfy the hypothesis of the statement and let $\{A_{k}\}_{k\in\n}\subset C_c^{\infty}(\b^4)$ be such that $A_k\to A$ strongly in $L^4$ and $dA_k\to dA$ strongly in $L^2$. Hence, we have that $F_{A_k}\to F_A$ strongly in $L^2$ and, for $k\in\n$ large enough, we get
    \begin{align*}
        \int_{\b^4}\lvert F_{A_k}\rvert^2\, d\L^4<\eps.
    \end{align*}
    Hence, for $k\in\n$ large enough we have $A_k\in\mathscr{U}^{\eps}$. We choose $\eps,C>0$ so that $\mathscr{V}_{C}^{\eps}=\mathscr{U}^{\eps}$. Thus, we get that for every $k\in\n$ large enough there exists $g_k\in W^{1,p}(\b^4,G)$ such that $d^*A_k^{g_k}=0$ and 
    \begin{align*}
        \|A_{k}^{g_k}\|_{W^{1,\frac{q}{2}}(\b^4)}\le C\|F_{A_k}\|_{L^\frac{q}{2}(\b^4)}
    \end{align*}
    for $q=4,p$. Hence we find a subsequence (not relabeled) such that $A_k^{g_k}\rightharpoonup A_{\infty}$ weakly in $W^{1,\frac{p}{2}}$. Clearly, we have $d^*A_{\infty}=0$. Moreover, $A_k^{g_k}\rightharpoonup A_{\infty}$ weakly in $W^{1,2}$ and we get 
    \begin{align*}
        \|A_{\infty}\|_{W^{1,2}(\b^4)}\le\liminf_{k\to+\infty}\|A_{k}^{g_k}\|_{W^{1,2}(\b^4)}\le C\lim_{k\to+\infty}\|F_{A_k}\|_{L^2}=C\|F_A\|_{L^2(\b^4)}.
    \end{align*}
    Since $p>4$, the weak convergence of $A_k^{g_k}$ in $W^{1,\frac{p}{2}}$ implies that $A_k^{g_k}\to A_{\infty}$ strongly in $L^4$. In particular, by
    \begin{align*}
        dg_k=g_kA_k^{g_k}-A_kg_k,
    \end{align*}
    we get that $\{g_k\}_{k\in\n}$ is uniformly bounded in $W^{1,4}(\b^4)$. Thus, there exists a subsequence (not relabeled) such that $g_k\rightharpoonup g\in W^{1,4}(\b^4)$ weakly in $W^{1,4}$. By passing to the limit in the previous equality, we eventually get
    \begin{align*}
        dg=gA_{\infty}-Ag,
    \end{align*}
    i.e. $A_{\infty}=A^g$. This concludes the proof of Proposition \ref{appendix: Coulomb gauge extraction for  forms}.
\end{proof}
\begin{Rm}\label{appendix: remark}
   A statement that is completely equivalent to Proposition \ref{appendix: Coulomb gauge extraction for  forms} also applies to every open and bounded domain $D$ contained in $\mathbb{R}^4$ with a sufficiently smooth boundary. This is necessary to use the standard tools from elliptic regularity theory in the proof. For the purposes of this paper, it suffices to note that all the elliptic regularity theory can be applied to cubes using standard reflection methods. As a result, Proposition \ref{appendix: Coulomb gauge extraction for  forms} can be extended to open 4-cubes in $\mathbb{R}^4$.
\end{Rm}
\noindent
Finally we extend Proposition \ref{appendix: Coulomb gauge extraction for  forms} to the sphere case following the same proof of Proposition \ref{appendix: global Coulomb gauge extraction on a sphere}.  
\begin{Propa}\label{L4-gauge-sphere}
Let $G$ be any compact and connected Lie group. There exist constants $\eps_G,C_G>0$ depending only on $G$ such that for every $\g$-valued $1$-form $A\in L^4(\s^4)$ on $\s^4$ such that
\begin{align*}
    \int_{\s^4}\lvert F_A\rvert^2\, d\H^4<\eps_G
\end{align*}
we can find a gauge $g\in W^{1,4}(\s^4,G)$ satisfying $A^g\in W^{1,2}(\s^4)$, $d^*A^g=0$ and
\begin{align*}
    \|A^g\|_{W^{1,2}(\s^4)}\le C_G\|F_A\|_{L^2(\s^4)}.
\end{align*}
\end{Propa}
\section{\texorpdfstring{$G$}{Z}-valued map extensions of traces in \texorpdfstring{$W^{1,2}(\p\b^5,G)$}{Z}}
The aim of the present appendix is to show the following extension result for general compact Lie groups. 
\begin{Propa}
\label{lm-exten} Let $G$ be a compact Lie group. We can find a smooth Riemannian manifold $M_G$ such that for every $g\in W^{1,2}(\p\b^5,G)$ there exists a measurable map
\begin{align*}
    \operatorname{Ext}(g):\,\,&M_G \longrightarrow  W^{1,\frac{5}{2}}(\b^5,G)\\[\sep]
    &\quad p \longmapsto\, g_p   
\end{align*}
such that for $\vol_{M_G}$-a.e. $p\in M_G$ we have
\begin{align}\label{res-prop}
    g_p=g \quad\mbox{ on }\,\s^4=\partial\b^5
\end{align}
and for $g\equiv g_0\in G$
\be
\label{b-0001}
g_p\equiv g_0
\ee
We have
\be
\label{5demi}
\int_{M_G}\int_{\b^5}|d_xg_p|^{\frac{5}{2}} d\L^5\,d\vol_{M_G}(p)\le \|dg\|^{\frac{5}{2}}_{L^2(\s^4)}
\ee
and
\begin{align}
    \label{b-001}
    \int_{M_G}\bigg(\int_{\p\Om\cap \b^5}\lvert d_xg_p\rvert^2\, d\H^4\bigg)\,d\vol_{M_G}(p)\le C_G(\Om)\ \|dg\|^2_{L^2(\s^4)}
\end{align}
for every $C^2$-domain $\Omega\subset\r^5$ of $\r^5$. 

\noindent
Moreover, for every $C^2$-domain $\Omega\subset\r^5$ of $\r^5$, $\operatorname{Ext}$ is continuous from $H^{\frac{1}{2}}(\s^4,G)$ into $L^2(\p\Om\times M_G)$. More precisely, for every pair of maps $(g^1,g^2)\in H^{\frac{1}{2}}(\s^4,G)\times H^{\frac{1}{2}}(\s^4,G)$ we have
\be
\label{b-0002}
  \int_{M_G}\bigg(\int_{\p\Om\cap \b^5}\lvert g^1_p-g^2_p\rvert^2\, d\H^4\bigg)\,d\vol_{M_G}(p)\le C_G(\Om)\ \|g^1-g^2\|^2_{H^{\frac{1}{2}}(\s^4)},
\ee
where $C_G(\Om)$ only depend on $G$ and the $C^2$-norm of the domain $\Om$. In particular for any $g\in H^{\frac{1}{2}}(\s^4,G)$ there holds
\begin{align}
    \label{b-002}
    \int_{M_G}\bigg(\int_{\p\Om\cap \b^5}\lvert g_p-\operatorname{id}_G\rvert^2\, d\H^4\bigg)\,d\vol_{M_G}(p)\le C_G(\Om)\|g-\operatorname{id}_G\|^2_{H^{\frac{1}{2}}(\s^4)},
\end{align}
where $C_G(\Om)$ only depend on $G$ and the $C^2$-norm of the domain $\Om$.
\end{Propa}
\begin{Rma}
    Notice that, for every compact Lie group, we can write $G$ is 
    \begin{align}
        G=\bigsqcup_{i=1}^nG_{i}
    \end{align}
    where the $G_{i}$ are connected, compact Lie groups themselves and the one above is a disjoint union. Hence, up to working on each connected component $G_{i}$ of $G$ separately, without losing generality we can assume that $G$ is compact and connected in the statement of Proposition \ref{lm-exten}.
\end{Rma}
\begin{Lma}\label{Lemma: product}
    Let $G_1$ and $G_2$ be Lie groups. Assume that we have proved Proposition \ref{lm-exten} for 
 $G_1$ and $G_2$. Then, Proposition \ref{lm-exten} holds also for $G_1\times G_2$ with $M_{G_1\times G_2}:=M_{G_1}\times M_{G_2}$.
\end{Lma}
    \begin{proof}
        Let $g=(g^1,g^2)\in W^{1,2}(\s^4,G_1\times G_2)$. Applying Proposition \ref{lm-exten} to $g^1$ and $g^2$ separately we find the two maps 
        \begin{align*}
            \operatorname{Ext}(g^1):\,M_{G_1} &\longrightarrow  W^{1,\frac{5}{2}}(\b^5,G_1)\\[\sep]
            p_1 &\longmapsto\, g_{p_1}^1   
        \end{align*}
        and 
        \begin{align*}
            \operatorname{Ext}(g^2):\,M_{G_2} &\longrightarrow  W^{1,\frac{5}{2}}(\b^5,G_2)\\[\sep]
            p_2 &\longmapsto\, g_{p_2}^2
        \end{align*}
        and we define 
        \begin{align*}
            \operatorname{Ext}(g):\,M_{G_1\times G_2}:=M_{G_1}\times M_{G_2} &\longrightarrow  W^{1,\frac{5}{2}}(\b^5,G_1\times G_2)\\[\sep]
            p:=(p_1,p_2) &\longmapsto\, g_p:=(g_{p_1}^1,g_{p_2}^2).
        \end{align*}
        Straightforward computations show that $\operatorname{Ext}(g)$ satisfies \eqref{res-prop}, \eqref{b-001} and \eqref{b-002}.
    \end{proof}
\begin{Lma}\label{Lemma: covering}
    Let $G,H$ be Lie groups and let $\pi:G\to H$ be a smooth covering map. Assume that we have proved Proposition \ref{lm-exten} for $G$. Then, Proposition \ref{lm-exten} holds also for $H$ with $M_{H}:=M_{G}$.
\end{Lma}
    \begin{proof}
        Let $g\in W^{1,2}(\s^4,H)$. By \cite[Theorem 1]{bethuel-chiron}, we can find a lift $\tilde g\in W^{1,2}(\s^4,G)$ such that $\pi\circ\tilde g=g$. Now we use Proposition \ref{lm-exten} to find the map
        \begin{align*}
            \operatorname{Ext}(\tilde g):\,M_{G} &\longrightarrow  W^{1,2}(\b^5,G)\\[\sep]
            p &\longmapsto\, \tilde g_p
        \end{align*}
        We define
        \begin{align*}
            \operatorname{Ext}(g):\,M_{G} &\longrightarrow  W^{1,2}(\b^5,H)\\[\sep]
            p &\longmapsto\, g_p:=\pi\circ\tilde g_p.
        \end{align*}
        Straightforward computations show that $\operatorname{Ext}(g)$ satisfies \eqref{res-prop}, \eqref{b-001} and \eqref{b-002}.
    \end{proof}
\noindent
We recall the following structure theorem for compact Lie groups, whose proof can be found in \cite[§5.2.2 and §5.2.4]{sepanski}.
\begin{Tha}\label{Theorem: structure}
    Let $G$ be a compact and connected Lie group. Then, $G$ is diffeomorphic $(G'\times\mathbb{T}^n)/F$, where:
    \begin{itemize}
        \item $G'$ is a simply connected compact Lie group.
        \item $\mathbb{T}^n=U(1)\times...\times U(1)$ is an $n$-dimensional torus, called the \textit{abelian component} of $G$.
        \item $F$ is a finite abelian normal subgroup of $G'\times\mathbb{T}^n$. 
    \end{itemize}
\end{Tha}
\noindent 
Recall that if $H$ is a discrete normal subgroup of $G$, then the projection map $\pi:G\to G/H$ is a covering map. Hence, by Lemma \ref{Lemma: product} and Lemma \ref{Lemma: covering}, Theorem \ref{Theorem: structure} implies that if we can show Proposition \ref{lm-exten} for any simply connected compact Lie group and for the group $U(1)$, then we have proved it for every compact and connected Lie group. 

\medskip
\noindent
We start by facing the case $G=U(1)$.
\begin{Propa}
    Proposition \ref{lm-exten} holds for $G=U(1)$.
\end{Propa}
\begin{proof}
    Let $g\in W^{1,2}(\s^4,U(1))$. For every $u=u_1+iu_2\in W^{1,2}(\s^4,\c)$, by standard approximation results we get the identity
    \begin{align}\label{identity}
        \begin{cases}
            d(\bar udu)=d\bar u\wedge du=2i(du_1\wedge du_2) & \mbox{ distributionally on } \s^4\\
            d\lvert u\rvert^2=2u_1du_1+2u_2du_2 & \mbox{ in } L^2(\s^4).
        \end{cases}
    \end{align}
    Notice that \eqref{identity} immediately implies that $\bar udu$ is always a purely imaginary $1$-form on $\s^4$. Since $g=g_1+ig_2\in W^{1,2}(\s^4,U(1))\subset W^{1,2}(\s^4,\c)$, by using \eqref{identity} and the constraint $\lvert g\rvert^2\equiv 1$, we get
    \begin{align*}
        0=g_1dg_1+g_2dg_2=2(dg_1\wedge dg_2)=d(\bar gdg) \quad\mbox{ distributionally on } \s^4. 
    \end{align*}
    Hence, $\bar gdg$ is a purely imaginary closed $1$-form on $\s^4$. Since $H_{dR}^1(\s^4)=0$, standard $L^2$-Hodge decomposition on $\s^4$ (see for instance ) gives the existence of $\varphi\in W^{1,2}(\s^4,\r)$ such that
    \begin{align*}
        id\varphi=\bar gdg.
    \end{align*}
    Vol'pert chain rule in $W^{1,2}(\s^4,\r)$ gives
    \begin{align}\label{identity 2}
        d(e^{-i\varphi})=-ie^{-i\varphi}d\varphi=-e^{-i\varphi}\bar gdg\in L^2(\s^4,U(1)).
    \end{align}
    By Leibniz rule in the algebra $(L^{\infty}\cap W^{1,2})(\s^4,\c)$ and by \eqref{identity 2} we get
    \begin{align*}
        d(e^{-i\varphi}g)=d(e^{-i\varphi})g+e^{-i\varphi}dg=-e^{-i\varphi}\bar gdgg+e^{-i\varphi}dg=-e^{-i\varphi}\lvert g\rvert^2dg+e^{-i\varphi}dg=0,
    \end{align*}
    where in the last equality we have used that $\lvert g\rvert^2\equiv 1$. Hence, there exists a constant $g_0\in U(1)$ such that
    \begin{align*}
        g=g_0e^{i\varphi} \quad\mbox{ a.e. on } \s^4.
    \end{align*}
    Now let $\tilde\varphi\in W^{\frac{3}{2},2}(\b^5,\c)$. be the standard harmonic extension of $\varphi$, i.e. the solution of the following PDE
    \begin{align}
        \begin{cases}
            \Delta\tilde\varphi=0 & \mbox{ in } \b^5\\
            \tilde\varphi=\varphi & \mbox{ on } \s^4. 
        \end{cases}
    \end{align}
    Let $M_{U(1)}=\{p\}$ and $g_p:=g_0e^{i\tilde\varphi}$. By construction, \eqref{res-prop} and \eqref{5demi} are satisfied. Moreover, by standard trace theory for Sobolev functions $W^{\frac{3}{2},2}(\b^5,\c)$, the estimates \eqref{b-001}, \eqref{b-002} and hence the statement follow immediately.
\end{proof}
We now tackle the case of a general simply-connected compact Lie group $G$. Recall the following lemma, whose prove can be found in \cite[Lemma 6.1]{hardt-lin}.
\begin{Lma}\label{Lemma: projection}
    Let $N\subset\r^k$ be an $n$-dimensional smooth submanifold of $\r^k$ such that 
    \begin{align*}
        \pi_0(N)=\pi_1(N)=\pi_2(N)=0.
    \end{align*}
    Then, there exists a compact $(k-4)$-dimensional Lipschitz polyhedron $X\subset\r^k$ such that a locally Lipschitz retraction $Q:\r^k\smallsetminus X\to N$ such that 
    \begin{align}
        \int_{B_r(0)}\lvert dQ\rvert^p\, d\L^k<+\infty,
    \end{align}
    for every $p\in(1,4)$ and $r\in(0,+\infty)$. Moreover, the projection map $Q$ is smooth of constant rank $n$ near the manifold $G\subset\r^k$ and coincides with the identity on $N$. 
\end{Lma}
Lastly, we use Lemma \ref{Lemma: projection} to prove the following.
\begin{Propa}
    Proposition \ref{lm-exten} holds if $G$ is a simply connected, compact Lie group.
\end{Propa}
    \begin{proof}
        Notice that $\pi_2(G)=0$ for every compact Lie group and moreover, by Nash embedding theorem, there exists $k\in\n$ such that $G$ is isometrically embedded in $\r^k$. Hence, without losing generality, we can assume that $G\subset\r^k$ is a smooth submanifold of $\r^k$ satisfying the assumptions of Lemma \ref{Lemma: projection}. We denote by $Q:\r^k\smallsetminus X\to G$ the projection map given by applying Lemma \ref{Lemma: projection} to $G$. Recall that both $G$ and $X$ are compact subsets of $\r^k$ and let $B\subset\r^k$ be any open ball in $\r^k$ containing $G\cup X$. Exactly as in the proof of \cite[Theorem 6.2]{hardt-lin}, for a small positive number $\sigma\in (0,\operatorname{dist}(G,\partial B))$ chosen so that $Q$ is smooth on $G+B_{\sigma}^k(0)$ and any arbitrary point $p\in B_{\sigma}^k(0)\subset\r^k$ we define $X_p:=X+p=\{y+a \mbox{ : } y\in X\}$ and the translated projection $Q_p:\r^k\smallsetminus X_a\to G$ given by
        \begin{align}
            Q_p(y):=Q(y-p) \qquad\forall\, y\in\r^k\smallsetminus X_p. 
        \end{align}
        For such a sufficiently small $\sigma$,
        \begin{align*}
            \Lambda:=\sup_{p\in B_{\sigma}^k(0)}\operatorname{Lip}(Q_p|_G)^{-1}
        \end{align*}
        is, by the inverse function theorem, a finite number depending only on $G$.
        
        \noindent
        Fix any $g\in W^{1,2}(\s^4,G)$ and let $\tilde g\in W^{\frac{3}{2},2}(\b^5,\r^k)$ be the solution of the following PDE:
        \begin{align*}
            \begin{cases}
                \Delta\tilde g=0 & \mbox{ in } \b^5,\\
                \tilde g=g & \mbox{ on } \s^4. 
            \end{cases}
        \end{align*}
        Notice that, by standard elliptic regularity, $\tilde g\in C^{\infty}(\b^5,\r^k)$ and it satisfies the estimate
        \begin{align}\label{estimate elliptic}
            \|\tilde g-g_0\|_{W^{\frac{3}{2},2}(\b^5)}\le C\|g-g_0\|_{W^{1,2}(\s^4)},
        \end{align}
        where $g_0\in G$ is any constant in $G$. Define
        \begin{align*}
            \bar g:=\fint_{\s^4} g\,d\H^4 
        \end{align*}
        and notice that, for every $C^2$-domain $\Omega\subset\r^5$ the trace theorem for $C^2$-domains (see \cite[Chapter 1, Section 5.1]{Agr}) \eqref{estimate elliptic} with $g_0=\bar g$ and Poincaré--Wirtinger inequality give the estimate
        \begin{align}\label{estimate 2}
            \int_{\partial\Omega\cap\b^5}\lvert\iota_{\partial\Omega}^*d\tilde g\rvert^2\, d\H^4&=\int_{\partial\Omega\cap\b^5}\lvert\iota_{\partial\Omega}^*d(\tilde g-\bar g)\rvert^2\, d\H^4\\
            &\le C \|\tilde g-\bar g\|_{W^{\frac{3}{2},2}(\b^5)}^2\le C\|g-\bar g\|_{W^{1,2}(\s^4)}^2\le C(\Omega)\|dg\|_{L^2(\s^4)}^2,
        \end{align}
        for a constant $C(\Omega)>0$ depending only on the $C^2$-norm of the domain $\Omega$. We have also obviously using classical Sobolev embeddings 
        \be
        \label{gtilde5demi}
        \|d\ti{g}\|_{L^\frac{5}{2}(\b^5)}\le C\, \|\ti{g}-\ov{g}\|_{W^{\frac{3}{2},2}(\b^5)}\le C\, \|dg\|_{L^2(\s^4)}.
        \ee
        For every $p\in B_{\sigma}^k(0)$, we define 
        $$g_p:=Q_p\circ \tilde g:\b^5\to G$$
        and we notice that, by the area formula and chain rules in Sobolev Spaces $g_p\in W^{1,\frac{5}{2}}(\b^5,G)$ for $\L^k$-a.e. $p\in B_{\sigma}^k(0)$ moreover there holds $g_p|_{\s^4}=g$ for for $\L^k$-a.e. $p\in B_{\sigma}^k(0)$. Using the fact that by the maximum principle $\|\ti{g}\|_\infty\le \|g\|_\infty\le\mbox{diam}(G)$, we have
        \begin{align}\label{gp5demi}
        \begin{split}
            \int_{p\in B_{\sigma}^k(0)}\int_{\b^5}|d g_p|^{\frac{5}{2}}\ d\L^5 \ d\L^k(p)&\le\int_{\b^5}|d\ti{g}(x)|^{\frac{5}{2}}\ \int_{B_{\sigma}^k(0)}|dQ(\ti{g}(x)-p)|^{\frac{5}{2}}\ d\L^k(p)\ d\L^5(x)\\[\sep]
            &\le\int_{\b^5}|d\ti{g}(x)|^{\frac{5}{2}}\ \int_{B_{\sigma+\|g\|_\infty}^k(0)}|dQ(y)|^{\frac{5}{2}}\ d\L^k(y)\ d\L^5(x)\\[\sep]
            &\le C_G\int_{\b^5}|d\ti{g}(x)|^{\frac{5}{2}}\ d\L^5(x)\\[\sep]
            &\le C_G\,\|dg\|^{\frac{5}{2}}_{L^2(\s^4)}.
        \end{split}           
        \end{align}
        Given now any $C^2$-domain $\Omega\subset\r^5$, by using Fubini's theorem, the chain rule and \eqref{estimate 2} we infer that
        \be
        \label{cont-0}
        \begin{array}{rl}
           \ds \int_{B_{\sigma}^k(0)}\int_{\partial\Omega\cap\b^5}\lvert\iota_{\partial\Omega}^*dg_p\rvert^2\, d\H^4\, d\L^k&\ds\le\int_{\partial\Omega\cap\b^5}\lvert \iota_{\partial\Omega}^*d\tilde g(x)\rvert^2\int_{B_{\sigma}^k(0)}\lvert dQ_p(\tilde g(x))\rvert^2\, d\L^k(p)\, d\H^4(x)\\[\sep]
            &\ds\le\int_{\partial\Omega\cap\b^5}\lvert \iota_{\partial\Omega}^*d\tilde g(x)\rvert^2\int_{B_{\sigma}^k(0)}\lvert dQ(\tilde g(x)-p)\rvert^2\, d\L^k(p)\, d\H^4(x)\\[\sep]
            &\ds\le\int_{\partial\Omega\cap\b^5}\lvert\iota_{\partial\Omega}^*d\tilde g(x)\rvert^2\bigg(\int_{B}\lvert dQ(y)\rvert^2\, d\L^k(y)\bigg)\, d\H^4(x)\\[\sep]
            &\ds\le C\int_{\partial\Omega\cap\b^5}\lvert \iota_{\partial\Omega}^*d\tilde g(x)\rvert^2\, d\H^4\le C(\Omega)\int_{\s^4}\lvert dg\rvert^2\, d\H^4,
            \end{array}
        \ee
        and we have shown \eqref{b-001}. It remains to prove the continuity of $\operatorname{Ext}$ from $H^{\frac{5}{2}}(\s^4,G)$ into $L^2(\p\Om\times B_{\sigma}^k(0))$. Let $g^1$ and $g^2$ in $W^{1,2}(\s^4,G)$, we first have
        \be
        \label{cont-1}
        \|\ti{g}^1-\ti{g}^2\|_{W^{1,2}(\b^5)}\le C\, \|g^1-g^2\|_{H^{\frac{1}{2}}(\s^4)}
        \ee 
        We write  using again Fubini's theorem
        \be
        \label{cont-2}
        \begin{array}{l}
         \ds\int_{B_{\sigma}^k(0)}\bigg(\int_{\p\Om\cap \b^5}\lvert g^1_p-g^2_p\rvert^2\, d\H^4\bigg)\,d\L^k(p)  \\[\sep]
         \quad\quad\ds\le\int_{\partial\Omega\cap\b^5}\int_{B_{\sigma}^k(0)}\lvert Q(\tilde g^1(x)-p)-Q(\tilde g^2(x)-p) \rvert^2\, d\L^k(p)\, d\H^4(x)\\[\sep]
        \end{array}
        \ee
        Recall the Lusin--Lipschitz inequality
        \be
        \label{cont-3}
        \begin{array}{l}
        \lvert Q(\tilde g^1(x)-p)-Q(\tilde g^2(x)-p) \rvert^2\\[\sep]
        \ds\quad\quad\le C\, |\tilde g^1(x)-\tilde g^2(x)|^2\ \lf[M(|dQ|)^2(\tilde g^1(x)-p)+M(|dQ|)^2(\tilde g^2(x)-p)\rg]
        \end{array}
        \ee
        where $M(|dQ|)$ denotes the Hardy--Littlewood maximal function of $dQ$. Combining \eqref{cont-2}, \eqref{cont-3} and the trace theorem we get
        \be
        \label{cont-4}
        \begin{array}{l}
         \ds\int_{B_{\sigma}^k(0)}\bigg(\int_{\p\Om\cap \b^5}\lvert g^1_p-g^2_p\rvert^2\, d\H^4\bigg)\,d\L^k(p)  \\[\sep]
        \ds\quad\quad  \le C\,\int_{\partial\Omega\cap\b^5}|\tilde g^1(x)-\tilde g^2(x)|^2\ \, d\H^4(x)\ \|dQ\|_{L^2(B_{\sigma}^k(0))}^2\\[\sep]
        \ds\quad\quad\le C\  \|\tilde g^1-\tilde g^2\|^2_{W^{1,2}(\b^5)}\le C\| g^1- g^2\|^2_{H^{\frac{1}{2}}(\s^4)}.
        \end{array}
        \ee 
        Hence, the statement follows by letting $M_G:=B_{\sigma}^k(0)$.      
\end{proof}
\newpage
\printbibliography
\end{document}